\documentclass[10pt]{amsart}
\usepackage[utf8]{inputenc}
\usepackage[margin=1in]{geometry}
\usepackage{amsfonts,amssymb,amsthm,amsmath,enumitem,comment,bbm}
\usepackage[usenames,dvipsnames]{xcolor}
\usepackage{graphicx}
\usepackage[hidelinks]{hyperref}

\newcounter{problempart}

\allowdisplaybreaks
\usepackage{scalerel,stmaryrd}

\newcommand{\N}{\mathbb{N}}
\newcommand{\R}{\mathbb{R}}
\newcommand{\Q}{\mathbb{Q}}
\newcommand{\Z}{\mathbb{Z}}
\newcommand{\F}{\mathcal{F}}

\newcommand{\B}{\mathcal{B}}

\newcommand{\E}{\mathbb{E}}

\newcommand{\ve}{\varepsilon}
\newcommand{\Ll}{\mathcal{L}}
\newcommand{\Pp}{\mathbb P}
\newcommand{\f}{\frac}
\newcommand{\deq}{\overset{d}{=}}
\newcommand{\mbf}{\mathbf}

\newcommand{\wt}{\widetilde}

\newcommand{\Rup}{\R_{\uparrow}^4}
\newcommand{\wh}{\widehat}

\newcommand{\Exp}{\operatorname{Exp}}

\newcommand{\dir}{\xi}
\newcommand{\hh}{\mathfrak h}
\newcommand{\UC}{\operatorname{UC}}
\newcommand{\hypo}{\text{hypo}}

\newcommand{\Aa}{a}
\newcommand{\Bb}{b}

\newcommand{\Rf}{\mathfrak r}

\newcommand{\Hh}{\mathcal H}

\newcommand{\Uu}{U}

\newcommand{\Tt}{T}

\newcommand{\Dd}{D}

\newcommand{\BL}{L}

\newcommand{\arrv}{\pmb{a}}

\newcommand{\srvv}{\pmb{s}}

\newcommand{\depv}{\pmb{d}}
\newcommand{\Qs}{\mathcal{U}}

\def\bq#1{\pmb{#1}}

\newtheorem{theorem}{Theorem}[section]
\newtheorem{Theorem}[theorem]{Theorem}

\newtheorem{lemma}[theorem]{Lemma}

\theoremstyle{definition}

\numberwithin{equation}{section}
\theoremstyle{remark}
\newtheorem*{remark}{Remark}

\newcommand{\be}{\begin{equation}}
\newcommand{\ee}{\end{equation}}

\def\tspb{\hspace{0.9pt}}

\newcommand\abullet{{\raisebox{2pt}{\scaleobj{0.5}{\bullet}}}}  
\newcommand\aabullet{{\tspb\raisebox{2pt}{\scaleobj{0.5}{\bullet}}\,}}  

\def\ind{\mathbf{1}}

\DeclareMathOperator*{\argmax}{arg\,max}

 \usepackage{enumitem}   

\title[Scaling limit of multi-type invariant measures]{Scaling limit of multi-type invariant measures via  the directed landscape}

\author{Ofer Busani}
\address{Ofer Busani, University of  Edinburgh,
5321, James Clerk Maxwell Building
Peter Guthrie Tait Road
Edinburgh, EH9 3FD, United Kingdom}
\email{obusani@ed.ac.uk}
\author{Timo Sepp{\"a}l{\"a}inen}
\address{Timo Sepp{\"a}l{\"a}inen, University of Wisconsin-Madison, Mathematics Department, Van Vleck Hall, 480
Lincoln Dr., Madison WI 53706-1388, USA.}
\email{seppalai@math.wisc.edu}
\author{Evan Sorensen}
\address{Evan Sorensen, Columbia University, Mathematics Department,  Room 614, MC 4432,
2990 Broadway,
New York, NY 10027, USA.}
\email{evan.sorensen@columbia.edu}


\begin{document}

\begin{abstract}  This paper studies the large scale limits of multi-type invariant distributions and Busemann functions of planar stochastic growth models in the Kardar-Parisi-Zhang (KPZ) class. We identify a set of sufficient hypotheses for convergence of multi-type invariant measures of last-passage percolation (LPP) models to the stationary horizon (SH), which is the unique multi-type stationary measure of the KPZ fixed point. Our limit theorem utilizes conditions that are expected to hold broadly in the Kardar-Parisi-Zhang class, including convergence of the scaled last-passage process to the directed landscape.  We verify these conditions for the six exactly solvable models whose scaled bulk versions converge to the directed landscape, as shown by Dauvergne and Vir\'ag. We also present a second, more general, convergence theorem with future applications to polymer models and particle systems. Our paper is the first to show convergence to the SH without relying on information about the structure of the multi-type invariant measures of the prelimit models. These results are consistent with the conjecture that the SH is the universal scaling limit of multi-type invariant measures in the KPZ class. 
\end{abstract}

\maketitle
\setcounter{tocdepth}{2}
\tableofcontents

\section{Introduction}

\subsection{Stochastic growth models}

Irregular and stochastic growth has motivated a great deal of mathematics and both theoretical and experimental physics for a number of decades. Such phenomena are common in nature. They include tumors, crystals, bacterial colonies, propagation of flame fronts and the spread of fluid in a porous medium.   Mathematical  work on models of random growth began in probability theory   in the early 1960s with the introduction of the {\it Eden model}  \cite{Ede-61} and {\it first-passage percolation}  \cite{Ham-Wel-65}.  

On the theoretical physics side, a decisive impetus for the field  came from the 1986 paper \cite{Kardar-Parisi-Zhang-86} of Kardar, Parisi and Zhang. On the mathematics side, turn of the millennium breakthroughs \cite{baik-deif-joha-99, joha} showed that the fluctuations in certain directed planar growth models obey the limit distribution  of the largest eigenvalue of  the {\it Gaussian unitary ensemble}. This probability law had been derived by Tracy and Widom 
\cite{Tra-Wid-94} in 1994.  Subsequently, the study of models that exhibit this behavior has been conducted under the rubric \textit{Kardar-Parisi-Zhang universality}, or \textit{KPZ} for short.  The relevance of KPZ to physical reality has been confirmed by experiments,  for example,   in slowly burning fronts \cite{maun-etal-97} and turbulent liquid crystals \cite{take-sano-10}.

  The KPZ universality class is expected to be home to a broad class of mathematical models of random growth, random evolutions in random media, and interacting particle systems. Despite differences in their small-scale dynamical rules, all these models are expected to exhibit universal scaling exponents and limiting statistics associated with random matrix theory. The presently available, mathematically rigorous evidence for universality consists of results for   \textit{exactly solvable} models whose special features permit detailed analysis with methods from probability, algebra, combinatorics, and analysis. Any significant degree of universality remains mathematically conjectural. For example, in the undirected first-passage percolation model that launched the area in probability, no exactly solvable version has been discovered and consequently no KPZ properties have been rigorously proved.  

\subsection{Universal objects in the KPZ class}
While mathematically rigorous  universality remains out of reach, much recent  progress has taken place in the discovery of the putative universal limit processes in the KPZ class. 
To illustrate the situation with a metaphor, the present state in KPZ would be analogous to having, in the Gaussian world,  a description of Brownian motion and proof that some particular random walks converge to Brownian motion, but without proofs of  such limits for general classes of stochastic processes.  

As mentioned above, the role of probability distributions from random matrix theory and  related line ensembles in KPZ has been understood since the early 2000s, starting with the works \cite{baik-deif-joha-99,Johansson-2001}. Shortly afterwards, Pr\"ahofer and Spohn \cite{Prahofer-Spohn-02} studied the scaling limit of the spatial height process of the polynuclear growth  (PNG) model, leading to the discovery of what was termed the Airy process. The Airy process can be embedded as the top line of an ensemble of nonintersecting curves, which was  named the Airy line ensemble in \cite{CorwinHammond}. The paper \cite{CorwinHammond} developed the Gibbs property of the Airy line ensemble, which states that the conditional law of the top $k$ lines in an interval, conditioned on the values at the endpoints, is that of $k$ independent Brownian bridges conditioned not to intersect. This work has been the starting point for a huge amount of further investigation into the Airy line ensemble (see, for example, \cite{Dauvergne-Sarkar-Virag-2020,Dauvergne-Zhang-2021,Dauvergne-Virag-21,Aggarwal-Huang-2023} and the references therein). The Airy line ensemble is conjectured to capture the limiting multi-path height profiles for KPZ models when started from the narrow-wedge initial condition. Progress on understanding the long-time evolution of a conjectured universal limiting object for the narrow-wedge and other special choices of initial data was made in the first two decades of the century in several works \cite{Borodin-Ferrari-2008,Borodin-Ferrari-Prahofer-2007,Borodin-Ferrari-Sasamoto-2008,Baik-Liu-2021,Liu-2022}.

The first major  
  breakthrough to capture the full limiting space \textit{and time} evolution from arbitrary initial data was the construction of the  \textit{KPZ fixed point} in  2016 by Matetski, Quastel, and Remenik \cite{KPZfixed}. This is a Markov process in a function-valued state space that describes naturally occurring growing interfaces. Two years later came the construction of the \textit{directed landscape} (DL) by Dauvergne, Ortmann, and Vir\'ag \cite{Directed_Landscape}, derived as  the four-parameter scaling limit of Brownian last-passage percolation. The construction of the DL comes by first constructing an object known as the \textit{Airy sheet} (which describes the DL for two fixed time horizons) via a last-passage percolation problem across the Airy line ensemble. In subsequent contributions, \cite{Dauvergne2019UniformCT,Dauvergne-Virag-21} showed convergence to the Airy line ensemble and DL for several other solvable LPP models, the papers \cite{KPZ_equation_convergence,heat_and_landscape} 
  showed convergence of the KPZ stochastic partial differential equation to the KPZ fixed point, and most recently, Wu \cite{Wu-23} showed that the Green's function of the KPZ equation converges to the DL. 
  
  The KPZ fixed point and DL are intimately tied together. Indeed, one can construct the KPZ fixed point through a variational formula involving the initial data and the DL,  originally proved in \cite{reflected_KPZfixed} and  recorded below as equation \eqref{eqn:KPZvar}.  

\subsection{Stationary horizon}
In \cite{Busani-2021}, the first author initiated the study of  scaling limits of Busemann processes in the KPZ class, starting from the derivation of the distribution of the Busemann process for exponential last-passage percolation in \cite{Fan-Seppalainen-20}. This limiting process was termed the stationary horizon (SH). The Busemann process of a stochastic growth model is an analytic object that is useful for probing the geometric features of random growth. It is a stochastic  analogue of the   Busemann function of metric geometry, hence the term. 
Article \cite{Busani-2021} showed that, under the   diffusive scaling of the initial data consistent with  convergence of the space-time process to the DL, the Busemann process of  exponential last-passage percolation (LPP) converges to a nontrivial limit, which was termed the \textit{stationary horizon} (SH). It was conjectured in \cite{Busani-2021} that the SH is the universal scaling limit of Busemann processes of LPP models in the KPZ class.

Concurrently, the second and third author were  studying the regularity of the Busemann process of Brownian last-passage percolation (BLPP) in unscaled coordinates \cite{Seppalainen-Sorensen-21b}.  This led to the observation  that the unscaled BLPP Busemann process along a horizontal line is, remarkably, the same as the KPZ scaling limit of the Busemann process in LPP. 

The aforementioned papers led to our prior work \cite{Busa-Sepp-Sore-22a}, which proves   that the SH is the unique multi-type stationary measure of the KPZ fixed point  and   describes the Busemann process of the DL. In fact, the SH is an attractor for the coupled KPZ fixed point, meaning that it describes the long-term, recentered joint height functions of the process started from arbitrary coupled initial data satisfying appropriate drift conditions.  The prior work \cite{Rahman-Virag-21} studied infinite geodesics and Busemann functions in a fixed direction of space. Our investigations into the properties of the SH  allowed us to derive deep geometric consequences for infinite geodesics in DL, simultaneously across all directions. 

In a separate work, the present authors demonstrated the presence of the SH in the context of interacting particle systems.  Namely, our article \cite{Busa-Sepp-Sore-22b} showed that the TASEP speed process, introduced by  \cite{Amir_Angel_Valko11} as the process of limiting speeds of second-class particles in the totally asymmetric simple exclusion process (TASEP),  also converges to the SH. In this setting, the TASEP speed process can be understood as a coupling of random walks, indexed by their real-valued drifts, which lives in the Skorokhod space of functions $\R \to C(\R)$. See \cite[Theorem 2.5]{Busa-Sepp-Sore-22b} for the precise statement of convergence. 
Of particular significance  in that work is that convergence to the SH takes place without  a geometric interpretation in terms of Busemann functions or geodesics. The common thread is that, like the Busemann process, the TASEP speed process describes the unique jointly stationary coupling of the ergodic stationary measures of  multi-type TASEP. Prior to \cite{Amir_Angel_Valko11},  the  projections of the speed process onto finitely many types had been studied by \cite{Ferrari-Martin-2007} in the general $k$-species case. See below for a more detailed discussion on the previous literature. 

Prior to the present paper, the most recent instance of SH limit is in
 the  work \cite{GRASS-23} of Groathouse, Rassoul-Agha, and the second and third authors.  It is shown that the Busemann process of the KPZ stochastic partial differential equation converges to the SH under diffusive scaling, equivalently, in the low-temperature/long-time limit. 
 
 The present article compliments these previous works by developing a framework to explain the seemingly universal presence of the SH in the KPZ universality class, and showing convergence to the SH for several exactly solvable LPP models. This provides strong evidence for the centrality of the SH in the KPZ universality class.

The stationary horizon is a stochastic process $G = (G_\mu)_{\mu \in \R}$ living in the path space $D(\R,C(\R))$ of functions $\R \to C(\R)$ that are right-continuous with left limits,  where $C(\R)$ is given its Polish topology of uniform convergence on compact sets. Marginally, each $G_\mu\in C(\R)$ is a Brownian motion with diffusivity $\sqrt 2$ and drift $2\mu$. The law of the SH is described in terms of its finite-dimensional projections, which are given by mappings of independent Brownian motions with drift (see \cite{Busani-2021,Seppalainen-Sorensen-21b,Sorensen-thesis}). The present article does not rely on any detailed information about the joint law of this process. The only input used  is the invariance and uniqueness of the SH under the KPZ fixed point, recorded as Theorem \ref{thm:invariance_of_SH} below.

\subsection{Our contributions}
Informally stated, our general limit criterion in Theorem \ref{thm:LPPSHconv} gives the following limits for particular last-passage percolation models. The models discussed are the six models treated in \cite{Dauvergne-Virag-21}.  They are defined in detail in Section \ref{sec:solvable}.
\begin{theorem} \label{thm:all_conv}
The joint distribution of finitely many Busemann functions/jointly stationary measures  for Brownian {\rm(}BLPP{\rm)}, exponential, geometric, and Poisson LPP, the Poisson lines model, and the Sepp\"al\"ainen-Johansson {\rm(}SJ{\rm)} model, all converge, in the sense of finite-dimensional projections, to the SH under scaling around an arbitrary direction of curvature of the shape function.
\end{theorem}
Theorem \ref{thm:all_conv} is made more precise for each specific model in Theorems \ref{thm:Hamm_SH}  (Poisson LPP), \ref{thm:lines_conv} (Poisson lines), \ref{thm:SJ_conv} (SJ), \ref{thm:exp_geom} (exponential and geometric LPP), and \ref{BLPP_conv} (BLPP). Convergence for exponential LPP was already handled by the first author \cite{Busani-2021}, including tightness on the space $D(\R,C(\R))$. The work \cite{Seppalainen-Sorensen-21b} shows that the BLPP Busemann process is already the SH in the prelimit.  Using invariance of the BLPP Busemann process under diffusive scaling proved \cite{Busani-2021,Seppalainen-Sorensen-21b,Sorensen-thesis}, it is shown that the SH therefore also appears as a limit of the Busemann process in BLPP. The convergence result of \cite{Busani-2021} was used in \cite{Busa-Sepp-Sore-22a} to establish invariance of the SH under the KPZ fixed point, a key input in the present paper. The thesis \cite{Sorensen-thesis} gave an alternate proof of  this invariance by  utilizing Brownian LPP. Here, we generalize the prior results of \cite{Busani-2021,Seppalainen-Sorensen-21b} to scaling around an arbitrary direction. Each of the other five models is also scaled around an arbitrary direction away from the axes, except for the SJ model, which contains an interval of directions for which the shape function is flat. In the flat directions, KPZ fluctuations do not occur, so we require that the SJ model is centered around a direction of curvature.  See Section \ref{sec:SJ} for the precise statement on which directions are permitted. 

The main motivation behind this work is to provide a scheme that shows convergence to the SH and that is independent of the microscopic details of the stationary distributions of the model. As a by-product, we tie the convergence of the multi-type stationary measures to those of the DL, providing additional evidence for the universality of the SH. The present paper is the first to prove convergence to the SH without an  explicit description of the prelimiting object.

Existence, uniqueness, and an explicit description of the invariant measure for 2-species TASEP with specified particle densities was first completed in a series of works in the 1990s \cite{Ferrari-Kipnis-Saada-1991,Derrida-Janowsky-Lebowitz-Speer-1993,Ferrari-Fontes-Kohayakawa-1994}. Later, Angel \cite{ANGEL-2006} gave a combinatorial and queuing interpretation of these invariant measures. Ferrari and Martin \cite{Ferrari-Martin-2007} extended this construction to general $k$-species TASEP. In this work, the multi-type invariant measures are described via multi-line queues. A similar approach was taken by the same authors  in \cite{Ferrari-Martin-2005,Ferrari-Martin-2009} to describe the multi-type stationary measures in various versions of the Hammersley process, which are discussed further in the present paper.   Fan and the second author \cite{Fan-Seppalainen-20} adapted the techniques of \cite{Ferrari-Martin-2007,Ferrari-Martin-2005,Ferrari-Martin-2009} to find the jointly invariant measures for exponential LPP, and this has since been adapted to several other LPP and polymer models \cite{Seppalainen-Sorensen-21b,Groathouse-Janjigian-Rassoul-21,Bates-Fan-Seppalainen,GRASS-23}.  In the present article, we formulate the SJ model as a discrete-time queuing system. Multi-type invariant measures for a closely related queuing system were studied in \cite{Ferrari-Martin-2005}. This same work also studied multi-type invariant measures for the Poisson lines model; although in the present paper, we take the time to be discrete and space as continuous, opposite the convention of \cite{Ferrari-Martin-2005}. The multi-type stationary measures for Poisson LPP (also known as the Hammersley process), were constructed in \cite{Ferrari-Martin-2009}.  For geometric LPP, the $2$-type invariant measures were demonstrated in \cite{Groathouse-Janjigian-Rassoul-21}, and the extension to an arbitrary number of types follows by the methods of \cite{Fan-Seppalainen-20}.  

Our  work relies on  uniqueness of the SH as a multi-type invariant measure for the KPZ fixed point   \cite{Busa-Sepp-Sore-22a}, recorded below as Theorem \ref{thm:invariance_of_SH}. In each of these cases, the jointly invariant measures are described by explicit queuing functionals of independent random walks or Brownian motions. As such, it is reasonable to see how one could show convergence to the SH directly. However, in order to demonstrate our new method and avoid repeating similar techniques, we prove convergence using our framework.  
We show that, for any coupled initial data  jointly stationary for a prelimiting model, any subsequential limit must be jointly invariant for the KPZ fixed point and therefore must be the SH.

The type of convergence in Theorem \ref{thm:all_conv} is weak convergence  of finite-dimensional distributions on $D(\R,C(\R))$. Precisely,  the $k$-type jointly invariant measures of the solvable LPP models converge weakly on the space $C(\R,\R^k)$ to the finite-dimensional projections  $(G_{\mu_1},\ldots,G_{\mu_k})$ of the SH. Convergence to the full process $G$ would require  tightness on the path space $D(\R,C(\R))$. So far, this has been done in two cases:   for exponential LPP by the first author \cite{Busani-2021}, and subsequently  for the TASEP speed process by the present authors   \cite{Busa-Sepp-Sore-22b}.

A key object  of interest that falls within the scope of our convergence result is the difference  $G_{\mu_2}(x) - G_{\mu_1}(x)$ for $x \in \R$. Understanding the distribution of this quantity was key to proving the regularity of the SH in the direction parameter in \cite{Busani-2021,Seppalainen-Sorensen-21b}. In particular, the process $\mu \mapsto G_\mu(x)$ is almost surely a step function whose jumps are isolated. Geometric consequences of this result for  infinite geodesics in the DL  were then derived by the authors in \cite{Busa-Sepp-Sore-22a}. While we do not focus on geometric applications in the present article, the results give further evidence to the conjecture in \cite{Busani-2021}  that, for all models of the KPZ class, backward infinite geodesics started from $(N,N)$  intersect the horizontal axis in a sparse set of locations, when looking at a window of order $N^{2/3}$. 

The main contribution of this article lies not just in proving Theorem \ref{thm:all_conv}, but in providing a general framework for proving convergence to the SH under assumptions widely conjectured to hold for all KPZ models. We prove two theorems of this flavor. The first is Theorem \ref{thm:LPPSHconv}, which is specifically focused on planar LPP models and is used later in the paper to prove convergence to the SH for classical LPP models. The second theorem is \ref{thm:gen_SH_conv}, which is a statement meant for more general models such as polymers and particle systems. 

After the first version of this paper was posted, the groundbreaking article \cite{Aggarwal-Corwin-Hegde-2024} has appeared, which proves convergence of multi-type ASEP and stochastic six-vertex models, under the basic coupling, to the coupling of the KPZ fixed point under the common noise of the directed landscape.  As a Corollary, the article \cite{Aggarwal-Corwin-Hegde-2024} uses Theorem \ref{thm:gen_SH_conv} to prove that the multi-type invariant measures for ASEP (which are the same as those for the stochastic six-vertex model) converge to the SH. This convergence generalizes the authors' previous work \cite{Busa-Sepp-Sore-22b} in showing convergence of the TASEP speed process to the SH, to a general asymmetry parameter for ASEP. To be specific, we note that the convergence in \cite{Busa-Sepp-Sore-22b} is shown to hold on the space $D(\R,C(\R))$ of the SH, while the convergence of the multi-type invariant measures in \cite{Aggarwal-Corwin-Hegde-2024} is only presently known for finite-dimensional projections. A generalization of the TASEP speed process, known as the ASEP speed process, was constructed in \cite{Aggarwal-Corwin-Ghosal-2023}. It is an open problem to show convergence of the ASEP speed process to the SH on the space $D(\R,C(\R))$. Our previous work \cite{Busa-Sepp-Sore-22b} relied on the explicit description of the multi-type invariant measures for TASEP from \cite{Ferrari-Martin-2007}. This description involves a deterministic mapping of independent Bernoulli random walks with drift. Explicit descriptions of the multi-type invariant measures have been derived for ASEP \cite{Prolhac-Evans-Mallick-2009,Martin-2020,Aggarwal-Nicoletti-Petrov-2023}, but these all involve \textit{random and locally defined} mappings of independent Bernoulli random walks. Methods for proving such direct convergence from these descriptions are not currently present in the literature. Our general framework developed in the present paper allowed \cite{Aggarwal-Corwin-Hegde-2024} to overcome this challenge.

In the setting of LPP models, our assumptions are as follows   (see Theorem \ref{thm:LPPSHconv} for precise statements):
\begin{enumerate}
\item  Convergence of the  scaled LPP process to the DL.
\item Convergence of each marginal of the coupled initial data to Brownian motion with drift.
\item Joint invariance of the initial condition in the prelimiting model.
\item Tightness of exit points on the scale $N^{2/3}$ from the (marginally) stationary initial condition. 
\end{enumerate}
We comment on these assumptions. 

\smallskip 

(1)  To verify the convergence to the DL  we import the results of \cite{Dauvergne-Virag-21}. 

\smallskip 

(2) In each of the six models we consider, the marginal invariant measures are generalizations of i.i.d.\ random walks, so convergence to Brownian motion follows. In general one does  not expect random walks, but   diffusive scaling of the stationary initial data is still expected  (see, for example, the discussion in \cite[Appendix B]{Alevy-Krishnan}). We emphasize again that our general SH limit criterion requires only assumptions on the marginally invariant measures and no knowledge whatsoever of the joint distribution (See Theorem \ref{thm:LPPSHconv}).

\smallskip 

(3) Joint invariance of the coupled initial data is a general fact for Busemann processes \cite{Hoffman2008,Damron_Hanson2012,Georgiou-Rassoul-Seppalainen-17b,Janjigian-Rassoul-2020b,Groathouse-Janjigian-Rassoul-2023}.  However, for the SJ and Poisson lines models, existence of Busemann functions does not presently appear in the literature. In Lemma \ref{lem:SJ_joint_exist}, we prove the existence of jointly invariant measures for the Poisson lines model when we take the discrete coordinate as the time coordinate. The existence of multi-type invariant measures for the SJ model follows word-for-word by the same proof. The key idea for the existence proof uses the classical method of taking subsequential limits of Ces\`aro averages as the initial time goes to $-\infty$. This method was applied to finding queuing fixed points in \cite{Mairesse-Prabhakar-2003}, was first explicitly applied to Busemann functions in \cite{Damron_Hanson2012} and has since been adapted to various settings in \cite{geor-rass-sepp-yilm-15,Janjigian-Rassoul-2020b,blpp_utah,Groathouse-Janjigian-Rassoul-2023}. While the methods of these papers should allow one to prove uniqueness of the jointly invariant measures in the prelimit, proving such would be a digression from the main purpose of our paper. In any case, our Theorems \ref{thm:lines_conv} and \ref{thm:SJ_conv} state that \textit{any} such jointly invariant measures in these models converge to the SH.      

\smallskip

(4) Control of exit point deviations  on the scale of  the KPZ wandering exponent $2/3$ have turned out quite useful in KPZ work, motivating these estimates be studied a great deal. The first to study transversal fluctuations of geodesics for a model in the KPZ universality class was Johansson \cite{Johansson-2000}, who identified the $\f{2}{3}$ fluctuation exponent for Poisson LPP.
The seminal coupling work  on exit points of stationary LPP, \cite{Cator-Groeneboom-06},  provided 
polynomial bounds of order $CM^{-3}$ 
for the probability that the exit point in Poisson continuum LPP to $(N,N)$ from a stationary initial condition is greater than $MN^{2/3}$.  Bounds of this order were then obtained in  exponential lattice LPP by \cite{Balazs-Cator-Seppalainen-2006}. Refinements of  \cite{Balazs-Cator-Seppalainen-2006} appeared in \cite{Pimentel-18,Pimentel-21a,Sepp_lecture_notes}.    An improvement of the probability upper bound to $e^{-CM^2}$ was achieved in \cite{Ferrari-Occelli-18, Ferrari-Ghosal-Nejjar-19} by combining the methods of the coupling technique used in \cite{Cator-Groeneboom-06,Balazs-Cator-Seppalainen-2006} with estimates from random matrix theory \cite{Baik-BenArous-Peche-2005}. Bounds on exit points also give  control of  coalescence times of semi-infinite geodesics, which can be accomplished by utilizing the Burke property of solvable models from equilibrium initial data  \cite{pimentel2016, Seppalainen-Shen-2020}.

    Alternatively, if one can obtain estimates for fluctuations of geodesics in the bulk, one can connect these to exit point bounds (see, for example,  \cite[Theorem 2.4]{Bhatia-2020} and \cite[Lemma 2.5]{Ferrari-Occelli-18}).   Probability upper bounds of $e^{-CM}$ (and later refined to $e^{-CM^2}$) for the moderate deviations of the fluctuations of geodesics in exponential LPP  were achieved in \cite{Basu-Sidoravicius-Sly-2014,BasuSarkarSly_Coalescence,Basu-Ganguly-Zhang-2021, Basu-Ganguly-2021}. The key idea was a chaining argument originally introduced in \cite{Basu-Sidoravicius-Sly-2014}. Using this chaining technique combined with optimal moderate deviation results for passage times in the setting of Poisson LPP \cite{Lowe-Merkl-2001,Lowe-Merkl-Rolles-2002}, Hammond and Sarkar \cite{Hammond-Sarkar-2020} obtained the optimal probability upper bound of $e^{-CM^3}$ for geodesic fluctuations for that model.  The optimal probability bound was obtained for the directed landscape in \cite{Directed_Landscape,Dauvergne-Sarkar-Virag-2020}. The starting point in that setting was an estimate on the modulus of continuity for the Airy line ensemble developed in \cite{Dauvergne-Virag-18}, which made heavy use of the Brownian Gibbs property from \cite{CorwinHammond}.

A new method for obtaining optimal probability bounds of order $e^{-CM^3}$ was developed in the work of Emrah, Janjigian, and the second author in \cite{Emrah-Janjigian-Seppalainen-20,Emrah-Janjigian-Seppalainen-21}. The strategy is based on  a moment generating function formula of Rains \cite{Rains-2000},  now commonly known as the \textit{EJS--Rains identity}.  The first application of this method in \cite{Emrah-Janjigian-Seppalainen-20,Emrah-Janjigian-Seppalainen-21} was to exit points in exponential LPP from an arbitrary down-right path.  In subsequent work, the first author and Ferrari \cite{busa-ferr-20} used  \cite{Emrah-Janjigian-Seppalainen-20,Emrah-Janjigian-Seppalainen-21} to develop optimal  fluctuation bounds for  geodesics in exponential LPP.  The EJS--Rains identity has been adapted   to  interacting diffusions, directed polymers, and the stochastic six-vertex model in \cite{Landon-Sosoe-Noack-2020, Landon-Sosoe-22a,Landon-Sosoe-22b,landon2023tail,Xie-22}. Shortly after the first appearance of \cite{Emrah-Janjigian-Seppalainen-20},  \cite{Bhatia-2020} proved  the same bounds for exit points from the axes using an estimate from random matrix theory \cite{Ledoux-Rider-2010}.

Turning to the exit point bounds required for our proofs, 
 for geometric LPP these  bounds   were proved in \cite{Groathouse-Janjigian-Rassoul-21}, and for Brownian LPP in the PhD thesis of the third author \cite{Sorensen-thesis}.   Theorems \ref{thm:Hamexitpt}, \ref{thm:PoiL_exit}, and \ref{thm:SJ_exit_pt} below  give the exit point bounds  for the three solvable models for which these bounds have not previously been established (Poisson LPP, Poisson lines, and SJ).  
 
 To clarify, 
 \cite{Cator-Groeneboom-06} showed tightness of the exit point in  Poisson  LPP, but only  in a fixed direction. We require a stronger result for a  direction perturbed on the order $N^{-1/3}$. While the methods of \cite{Cator-Groeneboom-06} almost certainly can be extended to this case, we instead adapt   \cite{Emrah-Janjigian-Seppalainen-21} to derive optimal-order exit point bounds for the Poisson LPP, SJ and Poisson lines models. Similar moment generating function formulas for Poisson LPP and the SJ model appear in \cite{Rains-2000}, but these are not adequately suited for our purpose. 
 Consequently, our paper serves as a reference for exit point bounds and the EJS--Rains identity for several solvable models previously not covered. A particular (not always obvious) choice of the stationary LPP process is needed to apply the technique of \cite{Emrah-Janjigian-Seppalainen-21}. We develop the correct parametrization in each of these models.

 \subsection{Notation and conventions}
 \begin{enumerate}
     \item  $\Z$, $\Q$ and $\R$ are restricted by subscripts, as in for example $\Z_{> 0}=\N=\{1,2,3,\dotsc\}$.
    \item $\mbf e_1= (1,0)$ and $\mbf e_2 = (0,1)$ denote the standard basis vectors in $\R^2$.
    \item Equality in distribution is   $\tspb\deq\tspb$ and convergence in distribution $\tspb\Longrightarrow$.
    \item   $X \sim \Exp(\alpha)$ means that 
    $\Pp(X>t)=e^{-\alpha t}$ for $t>0$. 
    \item For $v \in (0,1)$, $Y \sim {\rm Geom}(v)$ means $\Pp(Y = k) = (1-v)^k v$  for $k \in \{0,1,2,\ldots\}$.
\item For $p \in (0,1)$, $X \sim {\rm Ber }(p)$ means that $\Pp(X = 1) = p = 1- \Pp(X = 0)$.
    \item A two-sided standard Brownian motion is a continuous random process $\{B(x): x \in \R\}$ such that $B(0) = 0$ almost surely and   $\{B(x):x \ge 0\}$ and $\{B(-x):x \ge 0\}$ are two independent standard Brownian motions on $[0,\infty)$.
    \item\label{def:2BMcmu} If $B$ is a two-sided standard Brownian motion, then 
    $\{c B(x) + \mu x: x \in \R\}$ is a two-sided Brownian motion with diffusivity $c>0$ and drift $\mu\in\R$. 
    \item The parameter domain of the directed landscape is  $\Rup = \{(x,s;y,t) \in \R^4: s < t\}$.
 \end{enumerate}

 \subsection{Organization of the paper}
 Section \ref{sec:DLKPZ} introduces the necessary background and proves the general framework for convergence to the SH. Theorem \ref{thm:LPPSHconv} deals more concretely with LPP models, while Theorem \ref{thm:gen_SH_conv} gives a more general result. Afterward, we provide a roadmap for two models that can use this more general result in future work. In Section \ref{sec:solvable}, we verify the assumptions of Theorem \ref{thm:LPPSHconv} for the six exactly solvable LPP models studied. The appendices prove some technical lemmas. 

\subsection{Acknowledgements}
The work of O.B. was partially supported by the Deutsche Forschungsgemeinschaft (DFG, German Research Foundation) under Germany’s Excellence Strategy--GZ 2047/1, projekt-id 390685813. T.S.\ was partially supported by National Science Foundation grant DMS-2152362, by Simons Foundation grant 1019133, and by the Wisconsin Alumni Research Foundation. E.S. was partially supported by  the Fernholz foundation and an AMS-Simons travel grant. Additionally, this work was partly performed while E.S. was a PhD student at the University of Wisconsin--Madison, where he was partially supported by T.S. under National Science Foundation grant DMS-2152362. We thank the anonymous referees for their thoughtful reading and helpful comments, which has greatly improved the presentation of the results.

\section{The models and general limit theorems}

\subsection{Directed landscape and the KPZ fixed point} \label{sec:DLKPZ}
	The \textit{directed landscape} (DL) is a random continuous function $\Ll:\Rup \to \R$ that arises as the scaling limit of a large class of models in the KPZ universality class, and is expected to be a universal limit of such models. It was first constructed  and shown to be the scaling limit of Brownian last-passage percolation in~\cite{Directed_Landscape}. In~\cite{Dauvergne-Virag-21}, it was shown that the DL is also the scaling limit of several other exactly solvable models. The directed landscape satisfies the metric composition law: for $(x,s;y,u) \in \Rup$ and $t \in (s,u)$,
\[
\Ll(x,s;y,u) = \sup_{z \in \R}\{\Ll(x,s;z,t) + \Ll(z,t;y,u)\}.
\]
This implies the reverse triangle inequality:  for $s < t < u$ and $(x,y,z) \in \R^3$, $\Ll(x,s;z,t) + \Ll(z,t;y,u) \le \Ll(x,s;y,u)$. 
Furthermore, over disjoint time intervals $(s_i,t_i)$, $1 \le i \le n$, the processes $(x,y) \mapsto \Ll(x,s_i;y,t_i)$ are independent. A highly useful fact that we utilize is the following.
\begin{lemma}\cite[Corollary 10.7]{Directed_Landscape}\label{lem:Landscape_global_bound}
There exists a random constant $C$ such that for all $v = (x,s;y,t) \in \Rup$, we have 
\[
\Bigl|\Ll(x,s;y,t) + \f{(x - y)^2}{t - s}\Bigr| \le C (t - s)^{1/3} \log^{4/3} \Bigl(\f{2(\|v\| + 2)}{t - s}\Bigr)\log^{2/3}(\|v\| + 2),
\]
where $\|v\|$ is the Euclidean norm.
\end{lemma}

The \textit{KPZ fixed point} $\{h(t, \aabullet; \hh)\}_{t \ge 0}$, started from initial data $\hh$, is a Markov process on the space of upper semi-continuous functions $\R \to \R$. While more general conditions are possible (see, for example, \cite{KPZfixed,Sarkar-Virag-21}), in the present work we primarily restrict our attention to upper-semicontinuous initial data $\hh:\R \to \R$ satisfying $\hh(x) \le a + b|x|$ for some constants $a,b > 0$ and all $x \in \R$. In its original construction in~\cite{KPZfixed}, the KPZ fixed point was defined through its transition probabilities in terms of Fredholm determinants and shown to be the scaling limit of TASEP. In~\cite{reflected_KPZfixed}, it was shown that the KPZ fixed point may be represented as
\be \label{eqn:KPZvar}
h_\Ll(t,y;\hh) = \sup_{x \in \R}\{\hh(x) + \Ll(x,0;y,t)\}.
\ee
One advantage of the formulation in~\eqref{eqn:KPZvar} is that it defines a natural coupling of the KPZ fixed point started from different initial conditions $\hh$, but with the same driving dynamics. The subscript $\Ll$   emphasizes the choice of the driving noise for the process. This perspective allows us to consider jointly stationary initial data for the KPZ fixed point, analogously as for TASEP. Under appropriate asymptotic slope conditions, the unique such stationary initial data is given by the SH (cited from~\cite{Busa-Sepp-Sore-22a} below as Theorem \ref{thm:invariance_of_SH}).

The state space for the KPZ fixed point is
\be \label{UCdef}
\begin{aligned}
\UC &= \{\text{ upper semi-continuous functions }\hh:\R \to \R \cup \{-\infty\}: \\\  &\quad \text{ there exist }a,b > 0 \text{ such that } \quad \hh(x) \le a + b|x| \text{ for all }x \in \R, \\ &\quad \text{ and }\hh(x) > -\infty \text{ for some }x \in \R\}.
\end{aligned}
\ee
The topology on $\UC$ is that of local convergence of hypographs in the Hausdorff metric. When restricted to continuous functions, this convergence is equivalent to uniform convergence on compact sets. This space is discussed in  \cite[Section 3.1]{KPZfixed}. We review the relevant definitions. Set  
\[
\hypo(f) = \{(x,y) \in \R \times (\R \cup \{-\infty\}): y \le f(x)\},
\]
and for $\ve > 0$, let $B_\ve(\hypo f)$ be the set 
$
\bigcup_{(x,y) \in \hypo(f)} B_\ve(x,y),
$
where $B_\ve(x,y)$ is the ball of radius $\ve$ around $(x,y)$, under the metric $\rho((x_1,y_1),(x_2,y_2)) = |x_1 - x_2| + |e^{y_1} - e^{y_2}|$. A sequence of functions $f_n$ converges to $f$ in $\UC$ if there exist $a,b > 0$ so that $f_n(x) \le a + b|x|$ for all $n \ge 1$ and all $x \in \R$ and if, for all $\ve > 0$ and $M \in (0,\infty)$, there exists sufficiently large $N$ so that $\hypo(f_n |_{[-M,M]}) \subseteq B_\ve(f |_{[-M,M]})$ and $\hypo(f |_{[-M,M]}) \subseteq B_\ve(f_n |_{[-M,M]})$ whenever $n \ge N$.

As shown in \cite[Lemma 1.5]{Beer-82} and \cite[Section 3]{KPZfixed} convergence in UC of $f_n \to f$ is equivalent to the following three conditions: 
\begin{enumerate}
\item[(a)] There exists $a,b > 0$ so that $f_n(x) \le a + b|x|$ for all $n \ge 1$ and $x \in \R$. 
\item[(b)] For all $x \in \R$ and all sequences $x_n \to x$, $\limsup_{x_n\to x} f_n(x_n) \le f(x)$.
\item[(c)] For each compact set $K \subseteq \R$ and $x \in K$, there exists a sequence $x_n \to x$ with $\liminf_{n \to \infty} f_n(x_n) \ge f(x)$. 
\end{enumerate}

By way of definition, let us say that a random $k$-vector of functions $(f_1,\dotsc,f_k)$ is a multi-type invariant distribution for the KPZ fixed point if for all $t\ge0$, 
\be\label{KPZ8}  \bigl\{h_\Ll(t,\aabullet;f_i) - h_\Ll(t,0; f_i)\bigr\}_{1 \le i \le k} \deq \{f_i\}_{1 \le i \le k}. \ee  
We state a uniqueness  result from~\cite{Busa-Sepp-Sore-22a}. For a given parameter $\mu \in \R$, we consider random $\UC$ initial data $\hh:\R \to \R$ almost surely satisfying the following conditions.
\begin{equation} \label{eqn:drift_assumptions}
    \begin{aligned}
    &\text{If } \mu = 0, \qquad &\limsup_{x \to +\infty} \f{\hh(x)}{x} \in [-\infty,0] \qquad &\text{and}\qquad &\liminf_{x \to -\infty} \f{\hh(x)}{x} \in [0,\infty], \\
    &\text{if } \mu > 0,\qquad &\lim_{x \to +\infty} \f{\hh(x)}{x} = 2\mu\qquad&\text{and}\qquad &\liminf_{x \to -\infty} \f{\hh(x)}{x} \in (-2\mu,\infty], \\
    &\text{and if } \mu < 0,\qquad &\lim_{x \to -\infty} \f{\hh(x)}{x} = 2 \mu\qquad&\text{and}\qquad &\limsup_{x \to +\infty} \f{\hh(x)}{x} \in [-\infty, -2\mu).
    \end{aligned}
\end{equation} 
\begin{theorem}[\cite{Busa-Sepp-Sore-22a}] \label{thm:invariance_of_SH}
Let $(\Omega,\F,\Pp)$ be a probability space on which the stationary horizon $G=\{G_\mu\}_{\mu \in \R}$ and directed landscape $\Ll$ are defined, and such that  the processes $\{\Ll(x,0;y,t):x,y \in \R, t > 0\}$ and $G$ are independent. For each $\mu \in \R$, let $G_\mu$ evolve under the KPZ fixed point in the same environment $\Ll$, i.e., for each $\mu \in \R$,
\[
h_\Ll(t,y;G_\mu) = \sup_{x \in \R}\{G_\mu(x) + \Ll(x,0;y,t)\},\qquad\text{for all } y\in\R \text{ and } t > 0.
\]

{\rm(Invariance)} 
For each $t > 0$, we have the equality in distribution  $\{h_\Ll(t,\aabullet;G_\mu) - h_\Ll(t,0;G_\mu)\}_{\mu \in \R} \deq G$  between  random elements of $D(\R,C(\R))$.

\smallskip

\smallskip {\rm(Uniqueness)}
For $k \in \N$, $\bigl(G_{\mu_1}, \dotsc, G_{\mu_k})$ is the unique invariant distribution of the KPZ fixed point on $\UC^k$ 
such that, for each  $i\in\{1,\dotsc,k\}$,  the condition~\eqref{eqn:drift_assumptions} holds for $(G_{\mu_i}, \mu_i)$. 
\end{theorem}

\begin{remark}
Beyond uniqueness, the stationary horizon is attractive for the KPZ fixed point. See \cite[Theorem 2.1]{Busa-Sepp-Sore-22a} for the precise statement. 
\end{remark}

\subsection{General convergence criteria}
This section gives general conditions for convergence to the SH. As discussed in the introduction, these conditions are widely expected to hold for a large class of LPP models.  In later sections we verify these in the six exactly solvable cases. 

Before the next Theorem, we introduce some notations. Let $A,B$ each be equal to  either $\Z$ or $\R$ (they may differ).  In what follows, $\rho,\alpha,\beta,\tau$, and  
$\chi$
will be positive real parameters. We are interested in random last-passage functions $\mbf d:(A \times B)^2 \to \R \cup \{-\infty\}$. We will let $\mbf d_N$ be a sequence of these last-passage functions coupled together in some way. Then, define  
\begin{align} \label{dNtoL}
\Ll_N(x,s;y,t) := \f{\mbf d_N(s\rho N + x\tau N^{2/3} ,sN; t\rho N + y\tau N^{2/3},tN) - \alpha N(t - s) -  \beta \tau N^{2/3}(y - x)}{\chi N^{1/3}}.
\end{align}
Whenever $A$ or $B$ is $\Z$, we replace each corresponding coordinate in the argument of $\mbf d_N$ with the floor function of that coordinate. For a function $\mbf d:(A \times B)^2 \to \R \cup \{-\infty\}$ and for a function $f: A \to \R$, and $y \in A,t \in B \cap \R_{> 0}$, set
\[
h_{\mbf d}(t,y;f) = \sup_{x \in A}\{f(x) + \mbf d(x,0;y,t)\}.
\]
Next, define the rightmost exit point by 
\[
Z_{\mbf d}(t,y;f) = \sup\Bigl\{M \in A: \sup_{x \in A \cap (-\infty,M]}[f(x) + \mbf d(x,0;y,t)] = \sup_{x \in A}[f(x) + \mbf d(x,0;y,t)]\Bigr\}.
\]
Define $Z_\Ll$ similarly for $\Ll$ in place of $\mbf d$. Let $\iota_N:\UC\to\UC$ be the scaling operator   
\be \label{iotaNf}
\iota_N f(x) = \f{f(x\tau N^{2/3}) - \beta \tau N^{2/3}x}{\chi N^{1/3}}.
\ee
For a function $f:\Z \to \R$ satisfying $f(k) \le a + b|k|$ for some $a,b > 0$, the operator $\iota_N$ acts on the $\UC$ function obtained by continuous linear interpolation of $f$. 

In the proof of the Theorem below and in several other parts of the paper, we make frequent use of the Skorokhod representation theorem on arbitrary Polish spaces (\cite[Thm.~11.7.2]{dudl}, \cite[Thm.~3.18]{ethi-kurt}).

\begin{Theorem} \label{thm:LPPSHconv}
For $A,B$ equal to either $\Z$ or $\R$  and possibly different, consider a random function $\mbf d:(A \times B)^2 \to \R \cup \{-\infty\}$ that converges to the DL in the following sense: There exist parameters $\rho > 0$, $\alpha,\beta,\tau,\chi > 0$  
so that for each fixed pair $s < t$, there exists a coupling of copies of $\mbf d$, denoted by $\mbf d_N$, and $\Ll$, so that the function $(x,y) \mapsto \Ll_N(x,s;y,t)$ defined in \eqref{dNtoL}  converges to $(x,y) \mapsto \Ll(x,s;y,t)$, uniformly on compact subsets of $\R^2$.

 Further, assume that for some index set $I \subseteq \R$, there exists an $\N$-indexed sequence of $I$-indexed $\UC$-valued processes $\{f_{\mu}^N\}_{\mu \in I}\subseteq \UC$ so that, for each $N$, $f_{\mu}^N$ is independent of $\{\mbf d_N(x,s;y,t): 0 \le s < t, x,y \in \R\}$ and  satisfies conditions  \ref{ctight'}--\ref{coup'} below. In the case that $A = \Z$, $f_{\mu}^N$ is a $\Z \to \R$ function and the assumption pertains to the continuous linear interpolation of $f_{\mu}^N$ in $\UC$. 
\begin{enumerate} [label={\rm{(\arabic*)}}, ref={\rm{(\arabic*)}}]
			\item {\rm(Marginal convergence)} \label{ctight'} For every $\mu\in I$, $f_\mu^N(0) = 0$, the sequence 
		 $
			\{H_\mu^N\}_{N \in \N}:= \{\iota_N f^N_\mu\}_{N\in \N}$ converges in distribution, in $\UC$, to  a Brownian motion with drift $2\mu$ and diffusivity $\sqrt 2$. In particular, the limit is continuous and satisfies \eqref{eqn:drift_assumptions} for $\mu$.
	\item {\rm (Joint invariance/space-time stationarity) }\label{cst'}  For every $N \in \N$, $k\in\N$, every increasing vector $\overline{\mu}=(\mu_1,\dotsc,\mu_k)\in I^k$, $x \in A$, and $t \in B \cap \R_{\ge 0}$, 
	\begin{align*}
		f_{\overline{\mu}}^N :=(f_{\mu_1}^N,\dotsc,f_{\mu_k}^N) &\deq \{h_{\mbf d}(t,x + y;f_{\overline{\mu}}^N)- h_{\mbf d}(t,x;f_{\overline{\mu}}^N): y \in \R\} \\
  &:= \{h_{\mbf d}(t,x + y;f_{\mu_i}^N)- h_{\mbf d}(t,x;f_{\mu_i}^N): y \in \R\}_{1 \le i \le k},
	\end{align*}
 where the arguments of $f_{\overline \mu}^N$ on the left and $x$ and $x+y$ on the right are replaced by their floor functions when $A = \Z$.
\item {\rm(Tightness of exit point)} \label{coup'}
For each $\mu \in I$, $t > 0$, and compact set $K \subseteq \R$ {\rm(}inputting floor functions if needed{\rm)},
\[
\limsup_{M \to \infty}\limsup_{N \to \infty}\Pp\bigl(\,\sup_{y \in K} \bigl|Z_{\mbf d}(tN,  t\rho N + y\tau N^{2/3};f_\mu^N)\bigr| > MN^{2/3}\,\bigr) = 0.
\]
	\end{enumerate}
Then for  every   increasing vector $(\mu_1,\dotsc,\mu_k)\in I^k$,   as $N\to\infty$, we have the weak convergence  $(H_{\mu_1}^N,\dotsc,H_{\mu_k}^N)$ $\Rightarrow$ $(G_{\mu_1},\dotsc,G_{\mu_k})$ on $\UC^k$.
\end{Theorem}
\begin{remark}
Assumption \ref{cst'} is somewhat delicate. In all examples we consider for $A = \Z$, the model is not stationary under shifts by $x$ when $x \notin \Z$. However, we want to consider it as a process in the variable $y$ because this will converge to a Brownian motion under diffusive scaling.
\end{remark}
\begin{proof}
The proof contains several technical details, which are done carefully to allow for the cases $A = \Z$ and $A = \R$ together. To give the reader the main idea, we provide a quick outline now. The main purpose of Assumption \ref{ctight'} is to infer that the joint vector of initial data $(H_{\mu_1}^N,\ldots,H_{\mu_k}^N)$ is tight and that subsequential limits satisfy the drift condition \eqref{eqn:drift_assumptions}. The proof then proceeds to show that every subsequential limit must be the projection $(G_{\mu_1},\ldots,G_{\mu_k})$ of the SH. The key input is Theorem \ref{thm:invariance_of_SH} from \cite{Busa-Sepp-Sore-22a}, which gives uniqueness of the SH as a multi-type invariant measure for the KPZ fixed point satisfying the asymptotic drift assumptions. It is therefore sufficient to prove that every subsequential limit of $(H_{\mu_1}^N,\ldots,H_{\mu_k}^N)$ is jointly invariant for the KPZ fixed point, under the common coupling through the DL. Assumption \ref{cst'} ensures joint invariance for the scaled prelimiting model $\wt h_{\mbf d_N}$ (defined below). This means that for $t > 0$,
\be \label{deqh}
\{\wt h_{\mbf d_N}(t,y\,;f_{\mu_i}^N) -\wt h_{\mbf d_N}(t,0;f_{\mu_i}^N): y \in \R \}_{1 \le i \le k} \deq \{H_{\mu_i}^N(y): y \in \R\}_{1 \le i \le k}.
\ee
Then, we use assumption \ref{coup'} together with the convergence of the $\mbf d_N$ to $\Ll$ to infer that the prelimiting model at time $t$ converges to the time $t$ joint profile of the coupled KPZ fixed point started from the joint initial data. Taking limits on both sides of the distributional equality \eqref{deqh} implies the desired invariance of subsequential limits.

We turn to the full proof. For $f \in \UC$, $t > 0$, and $y \in \R$, we define 
\be \label{hdrep}
\begin{aligned}
\wt h_{\mbf d_N}^N(t,y;f) &:= \f{1}{\chi N^{1/3}}\Bigl(h_{\mbf d_N}(t N; t\rho N + y\tau N^{2/3};f) - \alpha N t - \beta \tau N^{2/3} y\Bigr) \\
&= \f{1}{\chi N^{1/3}}\Bigl(\sup_{x \in \R}\{f(x) + \mbf d_N(x,0;t\rho N + y\tau N^{2/3},tN) \} - \alpha N t - \beta \tau N^{2/3} y\Bigr) \\
&= \f{1}{\chi N^{1/3}}\Bigl(\sup_{x \in \R}\{f(x\tau N^{2/3}) - \beta \tau N^{2/3}x + \mbf d_N(x\tau N^{2/3},0;t\rho N + y\tau N^{2/3},t N)  + \beta \tau N^{2/3}x\}  \\
&\qquad\qquad\qquad\qquad\qquad\qquad\qquad\qquad\qquad\qquad\qquad\qquad\qquad\qquad\qquad- \alpha N t - \beta \tau N^{2/3} y\Bigr) \\
&=\sup_{x \in \R}\{ \iota_N'f(x)  + \Ll_N(x,0;y,t)\},
\end{aligned}
\ee
where in the second-to-last line, if $A$ or $B$ is $\Z$, we replace the appropriate arguments of $h_{\mbf d_N}$ and $f$ with floor functions, and the operator $\iota_N'$ is defined as
\[
\iota_N' f(x) = \begin{cases}
    \iota_N f(x) = \f{f(x\tau N^{2/3}) - \beta \tau N^{2/3}x}{\chi N^{1/3}}  & A = \R \\[5pt]
    \f{f(\lfloor x\tau N^{2/3} \rfloor) - \beta \tau N^{2/3}x}{\chi N^{1/3}} &A = \Z.
\end{cases}
\]
This is a subtle but important distinction. In the case $A = \Z$, the operator $\iota_N'f(x)$ is defined as above, while $\iota_N$ simply acts on the continuous linearly interpolated version of $f$. The reason we make this distinction is that $\iota_N'f$ is more convenient to work with in the proof, but $\iota_N f$ lives in the nicer topological space: that is, the subspace of $\UC$ consisting of continuous functions. Both $\iota_N f$ and $\iota_N' f$ are determined by their values when $x\tau N^{2/3} \in \Z$, so we may always obtain one from the other. Then, if $\iota_N f$ converges, in $\UC$, to a continuous function, it converges with respect to the topology of uniform convergence on compact sets, and we have that $\iota_N' f - \iota_N f$ converges to $0$ uniformly on compact sets. Thus, the distinction presents no issue in the limit. 
Assumption \ref{cst'} of stationarity ensures that, for $\overline \mu = (\mu_1 < \mu_2 < \cdots < \mu_k)$, $t\rho \in A$, and $N \in \Z$, 
\be \label{fmustat}
\begin{aligned}
&\quad \; \{\wt h_{\mbf d_N}^N(t,\abullet\,;f_{\overline \mu}^N) -\wt h_{\mbf d_N}^N(t,0;f_{\overline \mu}^N): y \in \R \} \\
&=\Biggl\{\f{1}{\chi N^{1/3}} \Bigl(h_{\mbf d_N}(tN, t\rho N + y \tau N^{2/3};f_{\overline \mu}^N) - \beta \tau N^{2/3}y - h_{\mbf d_N}(tN, t\rho N;f_{\overline \mu}^N)  \Bigr): y \in \R \Biggr\} \\
&\deq \bigl\{\iota_N' f_{\overline \mu}^N(y) :y \in \R  \bigr\}  
\end{aligned}
\ee

Let $\wt H_{\overline \mu} = (\wt H_{\mu_1},\ldots,\wt H_{\mu_k})$ be any subsequential weak limit of $H_{\overline \mu}^N$ in the space $\UC^k$.
If $A = \Z$, then for each $N$, $H_{\overline \mu}^N$ is a vector of continuous functions because it was defined as such by linear interpolation. Then in this case, by Assumption \ref{ctight'}, the marginal functions $\wt H_{\mu_i}$ are almost surely continuous. Thus, the convergence $H_{\overline \mu}^N \Longrightarrow \wt H_{\overline \mu}$ is locally uniform on compact sets in an appropriate coupling, and  $\iota_N' f_{\overline \mu}^N$ also converges to $\wt H_{\overline \mu}$ locally uniformly on compact sets, (recall we can obtain $\iota_N' f_{\overline \mu}^N$ from $\iota_N f_{\overline \mu}^N$).
We aim to show that $\wt H_{\overline \mu} \deq G_{\overline \mu}$. To do this, it suffices to show that (after interpolating in the $y$ variable if $A = \Z$), we have the process-level convergence $\wt h_{\mbf d_N}(t,\abullet,f_{\overline \mu}^N) \Longrightarrow h_{\Ll}(t,\abullet;\wt H_{\overline \mu})$ (note that both sides are $k$-tuples of functions). Then, by taking limits in \eqref{fmustat}, $\wt H_{\overline \mu}$ is a multi-type stationary distribution for the KPZ fixed point. 
The drift assumption \ref{ctight'} and the uniqueness of the multi-type stationary measures in Theorem \ref{thm:invariance_of_SH} completes the proof. By the dynamic programming principle, it suffices to show this for a single fixed time. In particular, we may safely assume that $t \rho \in \Z_+$ so that \eqref{fmustat} holds. 

By Skorokhod representation, we may couple a sequence $H_{\overline \mu}^N$ with $\wt H_{\overline \mu}$ (independently of $\Ll_N$ and $\Ll$) so that the convergence $H_{\overline \mu}^N \to\wt H_{\overline \mu}$ almost surely holds in $\UC^k$. Here, we take the coupling of $\Ll_N \to \Ll$ that converges in the graph topology discussed in the assumption of the theorem. Using \eqref{hdrep}, we accomplish our goal by showing that, for each $\ve > 0$, compact set $K \subseteq \R$, and  $1 \le i \le k$,
\be \label{691}
    \limsup_{N \to \infty} \Pp\Bigl(\,\sup_{y \in K} \Bigl|\,\sup_{x \in \R}\,\{\iota_N' f_{\mu_i}^N (x)  + \Ll_N(x,0;y,t)\} - \sup_{x \in \R}\{\wt H_{\mu_i}(x) + \Ll(x,0;y,t)\}     \Bigr| > \ve\Bigr) = 0.
\ee
We note here that the needed condition that $H_{\overline \mu}^N(x) \le a + b|x|$ for constants $a,b$, independent of $N$ and $x$, is met by assumption since each $H_{\overline \mu}^N$ marginally converges to $\wt H_{\overline \mu}$ in $\UC$.
Let $a > 0$, and take the compact set $K = [-a,a]$. The exit points for the DL satisfy $z_1 := Z_\Ll(t,x_1;f) \le Z_\Ll(t,x_2;f) =: z_2$ for $x_1 \le x_2$. To see this, assume to the contrary that $z_1 > z_2$. Then, the geodesic for $\Ll$ from $(z_2,0)$ to $(x_2,t)$ must cross the geodesic from $(z_1,0)$ to $(x_1,t)$ at a point $p \in \R^2$. Since $z_1$ is an optimal point for $z \mapsto f(x) + \Ll(z,0;y,t)$ and $\Ll(z,0;y,t) = \Ll(z,0;p) + \Ll(p;y,t)$, we have $f(z_1) + \Ll(z_1,0;p) \ge f(z_2) + \Ll(z_2,0;p)$.  Hence, 
\begin{align*}
    f(z_1) + \Ll(z_1,0;x_2,t)  &\ge f(z_1) + \Ll(z_1,0;p) + \Ll(p;x_2,t)  \\
    &\ge f(z_2) + \Ll(z_2,0;p) + \Ll(p;x_2,t) = f(z_2) + \Ll(z_2,0;x_2,t),
\end{align*}
but this is a contradiction because $z_2$ is the rightmost maximizer of $z \mapsto f(z) + \Ll(z,0;x_2,t)$. Now, by \eqref{hdrep}, if $\sup_{y \in K} |Z_{d_N}(t N,t\rho N + y \tau N^{2/3};f_\mu^N)|  \le M\tau N^{2/3}$ and $|Z_\Ll(t,\pm a;\wt H_{\mu})| \le M$, then for all $y \in K$,
\be \label{tight}
\begin{aligned}
\sup_{x \in \R}\{\iota_N' f_{\mu_i}^N (x)  + \Ll_N(x,0;y,t)\} &= \sup_{x \in [-M,M]}\{\iota_N' f_{\mu_i}^N (x) + \Ll_N(x,0;y,t)\} \qquad\text{and} \\
\sup_{x \in \R}\{\wt H_\mu(x)  + \Ll(x,0;y,t)\} &= \sup_{x \in [-M,M]}\{\wt H_\mu(x)  + \Ll(x,0;y,t)\}.
\end{aligned}
\ee

Hence, for any $M,N > 0$, we may bound the probability in \eqref{691} by 
\begin{align}
& \Pp\Bigl(\,\sup_{y \in K} \Bigl|\sup_{x \in [-M,M]}\{\iota_N' f_{\mu_i}^N (x)  + \Ll_N(x,0;y,t)\} - \sup_{x \in [-M,M]}\{\wt H_{\mu_i}^N(x) + \Ll(x,0;y,t)\}     \Bigr| > \ve, \nonumber \\
&\qquad\qquad\qquad\qquad\qquad \sup_{y \in K} |Z_{\mbf d_N}(t N,t\rho N + y \tau N^{2/3};f_{\mu_i}^N)|  \le  M\tau N^{2/3}, |Z_{\Ll}(t,\pm a,H_{\mu_i})| \le M\Bigr) \label{790}\\
&+ \Pp(\,\sup_{y \in K} |Z_{\mbf d_N}(t N,t\rho N + y \tau N^{2/3};f_{\mu_i}^N)|  > M\tau N^{2/3}) \label{791} \\
&+ \Pp(|Z_{\Ll}(t,a,\wt H_{\mu_i})| > M) + \Pp(|Z_{\Ll}(t,-a,\wt H_{\mu_i})| > M). \label{792}
\end{align}
Sending first $N \to \infty$, then $M \to \infty$, the term \eqref{791} vanishes by Assumption \ref{coup'}. The drift assumption \ref{ctight'} of $\wt H_{\mu}$ implies that $\wt H_{\mu}$ grows linearly in $x$. Lemma \ref{lem:Landscape_global_bound} implies that $\Ll(x,s;y,t) \sim -\f{(x-y)^2}{t-s}$ , so \eqref{792} vanishes as $M \to \infty$. Thus, it suffices to show that, for each fixed $M$, the term \eqref{790} vanishes as $N \to \infty$. 

It suffices to show that, for each fixed $M > 0$, on the event where $|Z_{\Ll}(t,\pm a,H_{\mu})| \le M$, for all sufficiently large $N$,
\be \label{eqn:unif_conv}
\sup_{y \in K} \Bigl|\sup_{x \in [-M,M]}\{\iota_N' f_{\mu_i}^N (x)  + \Ll_N(x,0;y,t)\} - \sup_{x \in [-M,M]}\{\wt H_{\mu_i}(x) + \Ll(x,0;y,t)\}     \Bigr| \le \ve.
\ee
In the case $A = \Z$, we have established that $\iota_N' f_{\mu_i}^N$ converges to $\wt H_{\mu_i}$ uniformly on compact sets under the coupling described. The assumption that $(x,y) \mapsto \Ll_N(x,0;y,t)$ converges uniformly on compact sets to $\Ll(x,0;y,t)$ completes the proof. For the rest of the proof, we now assume that $A  =\R$ so that $\iota_N' f_{\mu_i}^N = H_{\mu_i}^N$.

Since each $\wt H_{\mu_i}$ is continuous by Assumption   \ref{ctight'},
we may choose $\delta > 0$ so that for all $x,x' \in [-M,M]$ satisfying $|x - x'| < \delta$, $|\wt H_{\mu_i}(x) - \wt H_{\mu_i}(x')| < \ve$. 
Next, let $N_{\ve,\delta}$ be large enough so that for $N \ge N_{\ve,\delta}$, the following two conditions hold
\begin{enumerate}
\item[(i)] $\{(x,H_{\mu_i}^N(x)):x \in [-M,M]\} \subseteq B_{\ve\wedge \delta}(\hypo(H_{\mu_i}))$ \quad \text{(by the local convergence of hypographs)}
\item[(ii)] for all $x \in [-M,M]$ and $y \in K$, $|\Ll(x,0;y,t) - \Ll_N(x,0;y,t)| \le \ve$ (using the assumption of uniform convergence on compact sets).
\end{enumerate} 
Then, for $N \ge N_{\ve,\delta}$, for each $x \in [-M,M]$, there exists $(x',y')$ with $y' \le \wt H_{\mu_i}(x')$ so that when $|x' - x| < \delta$, $|H_{\mu_i}^N(x) - y'| \le \ve$. Then,
\[
H_{\mu_i}^N(x) \le y' + \ve \le \wt H_{\mu_i}(x') + \ve \le \wt H_{\mu_i}(x) + 2\ve,
\]
and so whenever $N \ge N_{\ve,\delta}$, $y \in [-a,a]$, and $x \in [-M,M]$,
\[
\sup_{x \in [-M,M]}\{H_{\mu_i}^N(x)  + \Ll_N(x,0;y,t)\} \le \sup_{x \in [-M,M]}\{\wt H_{\mu_i}(x) + \Ll(x,0;y,t)\} + 3\ve.
\]
This proves one side of \eqref{eqn:unif_conv}. 

To get the other direction, for each $y$, let $x_y$ be the rightmost maximizer of $\wt H_{\mu_i}(x)  + \Ll(x,0;y,t)$ over $x \in \R$. We note here that the set $\{x_y: y \in [-a,a]\}$ is almost surely finite. To see this, since $\wt H_{\mu_i}$  is a Brownian motion with diffusion coefficient $\sqrt 2$ and drift $2\dir$, 
\cite[Theorem 6]{Rahman-Virag-21} and  \cite[Theorem 5.9]{Busa-Sepp-Sore-22a} imply that, with probability one, each $x_y$ corresponds to the location at time $0$ of a backward infinite geodesic from $(y,t)$. By assumption, $x_y \in [-M,M]$. Then, \cite[Lemma 16]{Bhatia-23} implies that the set of such locations of geodesics from $(y,t)$ over $y \in [-a,a]$ and crossing through the compact set $\{0\} \times [-M,M]$ is almost surely finite.

By the equivalent conditions for convergence in $\UC$ stated in Section \ref{sec:DLKPZ}, for each $x_y$ in this finite set, we may choose $x_y^N$ to be a sequence so that $x_y^N \to x_y$ and $\liminf_{N \to \infty}H_{\mu_i}^N(x_y^N) \ge \wt H_{\mu_i}(x_y)$. Let $N_\ve$ be sufficiently large so that for all $y \in [-a,a]$, $x \in [-M,M]$, and $N \ge N_\ve$,
 $H_{\mu_i}^N(x_y^N) \ge \wt H_{\mu_i}(x_y) - \ve$ and $\Ll(x_y^N,0;y,t) \ge \Ll(x_y,0;y,t) - \ve$. Then, for $N \ge N_\ve$ and all $y \in [-a,a]$,
\begin{align*}
\sup_{x \in [-M,M]}\{ H_{\mu_i}^N(x)  + \Ll_N(x,0;y,t)\} &\ge H_{\mu_i}^N(x_y^N) + \Ll_N(x_y^N,0;y,t) \\
&\ge \wt H_{\mu_i}(x_y) + \Ll(x_y,0;y,t) - 2\ve = \sup_{x \in [-M,M]}\{\wt H_{\mu_i}(x) + \Ll(x,0;y,t)\} - 2\ve,
\end{align*}
completing the other inequality of \eqref{eqn:unif_conv}. 
\end{proof}

The following theorem uses similar ideas as Theorem \ref{thm:LPPSHconv}, but gives a more general and abstract framework that relates convergence to the DL to finite-dimensional convergence of its jointly invariant distributions to the SH. We record it here, as it may be useful in the future for models such as particle systems, stochastic PDEs, and polymers which do not have a direct LPP interpretation. Afterward, we provide a sketch for two models whose proof of convergence to the SH should lie within the reach of this Theorem. While the following is stated in a more general setting, Theorem \ref{thm:LPPSHconv} is not a corollary, as additional precise technical assumptions are needed there. Since the first version of our paper appeared, the theorem below was applied in \cite[Corollary 2.14]{Aggarwal-Corwin-Hegde-2024} to show convergence of the multi-type invariant measures of ASEP to the SH.

\begin{Theorem} \label{thm:gen_SH_conv}
\label{pr:shf}
For  $N\in\N$, we let $X_N$ be a   random object {\rm(}environment{\rm)}.  
For any $\UC$ initial condition  $f:\R \to \R$  such that $f(x) \le a + b|x|$ for constants $a,b > 0$, assume the existence of a height function  $(t,x)\mapsto \wt h^N_{X_N}(t,x;f)$, which is measurable with respect to the $\sigma$-algebra generated by $X_N$.  
	Assume that for some index set $I\subseteq \R$, there exists a sequence of $I$-indexed $\UC$-valued processes    $\{f_\mu^N\}_{\mu\in I}\subseteq C(\R)$, independent of $(\{X_N\}_{N \in \N},\Ll)$, and a family of operations $\iota_N: (\R \cup \{-\infty\})^\R \to (\R \cup \{-\infty\})^\R$ satisfying the following properties:
	\begin{enumerate} [label={\rm(\arabic*)}, ref={(\rm\arabic*)}]
			\item {\rm(Marginal tightness)} \label{ctight} For every $\mu\in I$, the sequence 
		 $
			\{H_\mu^N\}_{N \in \N}:= \{\iota_N f^N_\mu\}_{N\in \N}$ of initial conditions is tight in $\UC$. 
		For each $\mu\in I$, every weak  subsequential limit  $\widetilde{H}_\mu$ of 	$H^N_\mu$
		satisfies the asymptotic slope conditions~\eqref{eqn:drift_assumptions} for $\mu$ with probability one.
	\item {\rm(Joint invariance)}\label{cst} For every $N \in \N$, $k\in\N$, increasing vector $\overline{\mu}=(\mu_1,\dotsc,\mu_k)\in I^k$ and $t>0$, 
	\begin{equation}\label{c}
		H_{\overline{\mu}}^N:=(H_{\mu_1}^N,\dotsc,H_{\mu_k}^N)\sim \wt h^N_{X_N}(t,y;f_{\overline{\mu}}^N)-\wt h^N_{X_N}(t,0;f_{\overline{\mu}}^N). 
	\end{equation}
\item {\rm(Convergence to DL/KPZ fixed point)} \label{coup} There exists a coupling of copies of $\{X_N\}$ and $\Ll$ and an event $\Omega'$ of probability one on which the following holds. Fix $t > 0$ and $\mu \in I$. Assume $\wt H_\mu$ is a weak subsequential limit of $H_\mu^N$, along the subsequence $N_j$. Let $\{f_\mu^{N_j}\}$, $\wt H_\mu$ be any coupling of copies of these processes that is independent of the coupling $\{(X_N)_{N \ge 1},\Ll\}$ and that, with $H_\mu^{N_j} = \iota_{N_j} f_{\mu}^{N_j}$ and  on the event $\Omega'$,  satisfies $H_\mu^{N_j} \to \wt H_\mu$ in the topology on $\UC$. Then, there exists a further subsequence $\{f_\mu^{N_{j(l)}}\}_{l\in\N}$ such that, for any $\ve > 0$ and compact set $K\subseteq\R$ (not depending on the subsequence), 
\[\label{Oconv}
\lim_{l\to\infty} \Pp\big(\sup_{y \in K}|\wt h^{N_{j(l)}}_{X_{N_{j(l)}}}(t,y;f_\mu^{N_{j(l)}}) - h_{\Ll}(t,y;\wt H_\mu)| > \ve\big) = 0.
\]
	\end{enumerate}
	Then for  every   increasing vector $(\mu_1,\dotsc,\mu_k)\in I^k$, we have the weak convergence  \[
(H_{\mu_1}^N,\dotsc,H_{\mu_k}^N)\Rightarrow(G_{\mu_1},\dotsc,G_{\mu_k})
\]on $\UC^k$  as $N\to\infty$.
\end{Theorem}

\begin{proof}
	Fix $k\in\N$ and $\overline{\mu}=(\mu_1,\dotsc,\mu_k)\in\R^k$. As for each $i\in \{1,\dotsc,k\}$ $H^N_{\mu_i}$ is tight, $H^{N}_{\overline{\mu}}=(H^N_{\mu_1},\dotsc,H^N_{\mu_k})$ is tight as well. Let $\wt{H}_{\overline{\mu}}=(\wt{H}_{\mu_1},\dotsc,\wt{H}_{\mu_k})$ be a subsequential weak limit along the subsequence $N_j$. We show that $\wt H_{\overline \mu} \deq (G_{\mu_1},\ldots,G_{\mu_k})$. 
	
	From Condition \ref{ctight}, properties~\eqref{eqn:drift_assumptions} hold for each $(\wt H_{\mu_i},\mu_i)$.
	Hence, by Theorem \ref{thm:invariance_of_SH}, it is sufficient to show that $\wt H_{\overline \mu}$ is invariant under the KPZ fixed point, for then the uniqueness of Theorem \ref{thm:invariance_of_SH}
 implies that $\wt H_{\overline \mu}$ must be distributed as the finite-dimensional marginals of $G$. Specifically, we show that, for each $t > 0$,
	\be \label{Hinvar}
	h_{\Ll}(t,\aabullet;\wt H_{\overline \mu}) -h_{\Ll}(t,0;\wt H_{\overline \mu})  \deq \wt H_{\overline \mu}.
	\ee
	 As an intermediate step, we show that for each $t > 0$, there exists a subsequence $N_{j(l)}$ so that 
	\be \label{dist_conv}
	\wt h^{N_{j(l)}}_{X_{N_{j(l)}}}(t,\aabullet; f^{N_{j(l)}}_{\overline \mu}) \Longrightarrow h_{\Ll}(t,\aabullet;\wt H_{\overline \mu}).
	\ee
	  By Skorokhod representation, there exists a coupling of copies of $f_{\overline \mu}^{N_j}$ and $\wt H_{\overline \mu}$ such that the convergence $H_{\overline \mu}^{N_j} = \iota_N f_{\overline \mu}^{N_j} \to \wt H_{\overline \mu}$ holds on $\UC^k$. By condition \ref{coup} and a union bound, there exists a coupling of copies of $X_N$ and $\Ll$, which we take to be independent of the coupling of $\{f_{\overline \mu}^{N_j}\}$ and $\wt H_{\overline \mu}$ so that, under this coupling, there exists a further subsequence $N_{j(l)}$ such that, for every $\ve > 0$ and a compact set $K\subseteq \R$,
	\[
	\lim_{l\to\infty} \Pp\big(\sup_{y \in K, 1 \le i \le k }|\wt h^{N_{j(l)}}_{X_{N_{j(l)}}}(t,y;f_{\mu_i}^{N_{j(l)}}) - h_{\Ll}(t,y;\wt H_{\mu_i})| > \ve\big) = 0.
	\]
	Hence~\eqref{dist_conv} is proved. Combining~\eqref{dist_conv} with Condition \ref{cst} and the assumed weak convergence of $H_{\overline \mu}^{N_j}$ to $\wt H_{\overline \mu}$,~\eqref{Hinvar} follows from
	\be \label{Hinvarlim}
	h_{\Ll}(t,\aabullet;\wt H_{\overline \mu}) -h_{\Ll}(t,0;\wt H_{\overline \mu})   \deq \lim_{l \to \infty} \wt h^{N_{j(l)}}_{X_{N_{j(l)}}}(t,\aabullet;f_{\overline \mu}^{N_{j(l)}}) -\wt h^{N_{j(l)}}_{X_{N_{j(l)}}}(t,0;f_{\overline \mu}^{N_{j(l)}}) \deq \lim_{l \to \infty} H_{\overline \mu}^{N_{j(l)}} \deq \wt H_{\overline \mu},
	\ee
	where the limits above are weak limits in the space $\UC^k$. 
\end{proof}

We now present two examples of processes that could potentially be shown to converge to the SH using Theorem \ref{thm:gen_SH_conv}. We do not provide complete details, but instead focus on LPP models in the present paper (treated in the following section):
\begin{enumerate}
\item Discrete polymer models: The existence of Busemann functions for discrete positive-temperature models is known in wide generality \cite{Janjigian-Rassoul-2020b,Groathouse-Janjigian-Rassoul-2023}. In the case of the inverse-gamma polymer originally introduced in \cite{Seppalainen-2012}, it is known \cite{geor-rass-sepp-yilm-15} that the Busemann functions for each edge along a down-right path (in particular, along a horizontal line) are independent, and so the marginal tightness of Condition \ref{ctight} holds. If one could show convergence of the inverse-gamma polymer to the DL, along with Condition \ref{coup} of Theorem \ref{thm:gen_SH_conv}, convergence to the SH would follow. Such convergence is needed for the use of our Theorem, but \cite{Aggarwal-Huang-2023} shows convergence of the inverse-gamma line ensemble to the Airy line ensemble--the necessary first step. The joint distribution of Busemann functions for the inverse-gamma polymer was recently studied in \cite{Bates-Fan-Seppalainen}. The form of this measure is also built from queuing mappings, so one could conceivably prove convergence to the SH directly, using similar techniques as in \cite{Barraquand-Corwin-22}. The proof in \cite{Barraquand-Corwin-22} is somewhat technical, however, using methods from \cite{Auffinger-Baik-Corwin-2012}, so the method we have outlined would allow for a cleaner proof.  
\item Particle systems: In \cite{Busa-Sepp-Sore-22b}, we showed that the ergodic invariant measures of multi-type  TASEP constructed through basic coupling (that is, all types use common Poisson clocks) converge to SH. This proof relied on the explicit structure of the measures, first studied for an arbitrary number of species in \cite{Ferrari-Martin-2007}. Since then, \cite{Aggarwal-Corwin-Hegde-2024} has generalized this to ASEP by using Theorem \ref{thm:gen_SH_conv}. The paper \cite{Dauvergne-Virag-21} studies an alternate coupling of TASEP built from exponential LPP. For this coupling, \cite[Theorem 1.20]{Dauvergne-Virag-21} gives Condition \ref{coup} of Theorem \ref{pr:shf} with $\iota_N f(x) = N^{-1/3} f(2N^{2/3} x)$. We know the marginally invariant measures are i.i.d. Bernoulli, so tightness under the scaling holds. Thus, any jointly invariant measure in the LPP coupling satisfying Condition \ref{cst} must converge to the SH (matching drifts as needed).    
\end{enumerate}

\section{Stationary horizon limits in solvable models} \label{sec:solvable}
\subsection{Notation and conventions} \label{sec:gen_solv}
We turn our attention to solvable models. To prove Theorem \ref{thm:all_conv}, we verify the assumptions of Theorem \ref{thm:LPPSHconv} for each of the six models considered. For the convergence of $\Ll_N \to \Ll$, we use \cite{Dauvergne-Virag-21}, which gives us the desired convergence. In each subsection that follows, we will remind the reader of the choices of parameters. We note that in all cases except for Poisson last-passage percolation, the convention of \cite{Dauvergne-Virag-21} is that the vertical coordinates increase as we move downward. We flip this convention so  that the vertical coordinate increases as we move upward. By translation invariance and relabelling, this change needs no further justification. In the case of the SJ model, we reformulate the model as a last-passage model instead of a first-passage model. The precise details are explained in Section \ref{sec:SJ}.

It is important to note that in \cite{Dauvergne-Virag-21}, in the cases of the SJ and Poisson lines models, the convergence of the four-parameter field $\Ll_N \to \Ll$ was shown in a weaker hypograph topology (as opposed to uniform convergence on compact sets, as was shown for the others). However, when fixing the times $s < t$, the convergence of the process as a function of $(x,y) \in \R$ is shown to hold in the graph topology. Because the limiting object $\Ll$ is continuous, this implies uniform convergence on compact sets, as required by the Theorem. For more on this distinction, see \cite[Theorems 1.16, 10.3, Lemmas 7.2, 7.3, Remark 7.4, and pages 82-83]{Dauvergne-Virag-21}.

For parameters $\beta,\chi,\tau > 0$ (depending on the specific model), we consider the operator $\iota_N:\UC \to \UC$ defined by
\be \label{iotadef}
\iota_N f(x) = \f{f(x\tau N^{2/3}) - \beta \tau N^{2/3}x}{\chi N^{1/3}}.
\ee
We now make the following observation. If $f_N$ is an interpolation of an i.i.d. random walk with finite variance and $f_N(1)$ has mean $\beta + \f{2\mu \chi}{\tau}N^{-1/3}$, then  $\iota_N f$ converges to a Brownian motion with drift $2\mu$. From the choice of the parameters in each model, it can also be checked that the limiting diffusivity is $\sqrt 2$.   We will use this fact and parameterize all Busemann processes in this section by their mean. This may be different from parameterizations used in other works, but we will make clear our definitions used in the subsequent sections. In general, we will consider the Busemann process indexed by $\mu$, where the one-unit horizontal increments have mean
\be \label{meanparam}
\beta_N(\mu) = \beta + \f{2\mu \chi}{\tau}N^{-1/3},
\ee
and the value of $\beta,\chi,\tau$ (depending on $\rho$) will vary from model to model. In particular, with this choice of $\beta_N(\mu)$, Assumption \ref{ctight'} of Theorem \ref{thm:LPPSHconv} will be satisfied for all the cases considered and for any choice of $\mu$. As discussed in Section \ref{sec:DLKPZ}, there is a technical requirement for convergence in $\UC$ that there exist $a,b > 0$ so that $f_n(x) \le a + b|x|$ for all $n \ge 1$ and $x \in \R$. For each model we consider, the function is either the interpolation of a random walk or can be coupled with a random walk in a way that the difference goes to $0$ on compact sets (as is the case for the Poisson process discussed in Lemma \ref{lem:Poicoup}). By passing to a subsequence, the following lemma, which appeared in an earlier version of \cite{Busa-Sepp-Sore-22b} and is included in Appendix \ref{sec:BMRW}, resolves this issue for all cases considered. 
	\begin{lemma}\label{lem:sseq}
		Let  $S^N(x)$ be a sequence of scaled and
   continuously interpolated
   two-sided i.i.d.\ random walks,  
		converging, almost surely on some probability space, in the topology of uniform convergence on compact sets to a Brownian motion $B$ with drift $\mu$ and diffusivity $\sqrt 2$. Specifically, $S^N$ is the linear interpolation of 
  \[
  S^N(x) = \begin{cases}
                \f{\sum_{i = 1}^{ Nx } X_i^N  - c_1Nx}{c_2 \sqrt N} & x > 0 \\
                0 &x = 0 \\
                -  \f{\sum_{i = Nx + 1}^{ 0 } X_i^N  - c_1Nx}{c_2 \sqrt N} &x < 0
      \end{cases}
  \]
  where $\{X_i^N\}_{i \in \Z}$ are i.i.d. for each $N \in \N$, and $c_1,c_2  \in \R$ ae constants. 
  Assume further that there exists $\ve > 0$ so that $\mathbb E[e^{tX_i^N}] < \infty$ for all $t \in (-\ve,\ve)$ and $N \in \N$.
	Then, there exists a deterministic subsequence $N_j$ and  a finite random constant $M > 0$ such that with probability one,  
	\[
	S^{N_j}(x) \le (3 + |\mu|)|x| + M
	\]
	for all $x \in \R$ and sufficiently large $j$. 
	\end{lemma}

The space-time stationarity of Assumption \ref{cst'} will be discussed in detail for each model. We refer the reader to \cite{Janjigian-Rassoul-2020b,Georgiou-Rassoul-Seppalainen-17b,Groathouse-Janjigian-Rassoul-2023} for more on the joint shift-invariance in a general setting.  

The remaining nontrivial part of the convergence, therefore, lies in showing Assumption \ref{coup'} of Theorem \ref{thm:LPPSHconv}. For each of the cases considered, we note that, if $K = [a,b]$, a simple paths-crossing and monotonicity argument (as in the proof of Theorem \ref{thm:LPPSHconv}) implies that
\begin{align*}
&\quad \;\Pp\bigl(\sup_{y \in K} |Z_{\mbf d}(tN,  t\rho N + y\tau N^{2/3};f_\mu^N)| > MN^{2/3}\bigr)  \\
&\le \Pp\bigl( Z_{\mbf d}(tN,  t\rho N + a\tau N^{2/3};f_\mu^N) < -MN^{2/3}\bigr) +\Pp\bigl( Z_{\mbf d}(tN,  t\rho N + b\tau N^{2/3};f_\mu^N) > MN^{2/3}\bigr).
\end{align*}
Hence, it suffices to prove that for each fixed $y \in \R$,
\be \label{exitpty}
\lim_{M \to \infty}\limsup_{N \to \infty}\Pp( |Z_{d}(tN,  t\rho N + y\tau N^{2/3};f_\mu^N)| > MN^{2/3}) = 0,
\ee
where we have removed the supremum over $y \in K$ inside the probability expression. Showing these exit point bounds is the most substantial part of the proofs. For the Poisson LPP, Poisson lines, and SJ models, these have not yet been developed in the generality that we need. The proofs in these cases are somewhat technical and are postponed until Section \ref{sec:Details}.
To distinguish the last-passage time $\mbf d$ in the various models, we use the following conventions:

\begin{center}
\begin{tabular}{|c|c|}
     \hline Model & $\mathbf d$ \\ \hline
     Poisson LPP &$\Hh$  \\
     Poisson lines &$\Uu$ \\
     SJ model &$\Tt$ \\
     Lattice LPP & $\Dd$ \\
     Brownian LPP & $L$ \\ \hline
\end{tabular}
\end{center}

\subsection{Poisson LPP} \label{sec:PoissonLPP}
Let $\mbf X$ be a rate one Poisson process in the plane $\R^2$. For points $\mbf x \le \mbf y \in \R^2$, we let $\Hh(\mbf x;\mbf y)$ be the maximal number of points of $\mbf X$ that lie along an up-right path in the plane from $\mbf x$ to $\mbf y$, excluding a point at $\mbf x$ if one exists. This model is known as Poisson last-passage percolation, or the Hammersley process. The model in this form was first studied in \cite{Aldous-Diaconis-1995} and was motivated by the previous work of Hammersley \cite{Hammersley-1972} on the longest increasing subsequence problem. In the notation of Theorem \ref{thm:LPPSHconv}, $A = B = \R$. The existence of Busemann functions for this model was established  in \cite[Theorem 2.2]{Cator-Pimentel-2012}. We parameterize them as follows: for $\beta > 0$,
\be \label{eqn:HamBuse}
\B^\beta_\Hh(\mbf x,\mbf y) = \lim_{t \to -\infty} \Hh((-t,-t\beta^2),\mbf y) - \Hh((-t,-t\beta^2),\mbf x).
\ee
To describe the distribution of Busemann functions along a horizontal line, we first introduce some notation. Let $\mathcal N$ be the set of all nonnegative, locally finite measures on $\R$. For $\nu \in \mathcal N$, define, for $x \in \R$, 
\be \label{nufundef}
\nu(x) = \begin{cases}
\nu(0,x] & x \ge 0 \\
-\nu(x,0] &x < 0.
\end{cases}
\ee
It is straightforward to verify that $\nu(0) = 0$ and for $x < y$, 
\be \label{nuinc}
\nu(x,y] = \nu(y) - \nu(x).
\ee
We observe here that for any $\nu \in \mathcal N$, its associated function defined by \eqref{nufundef} is non-decreasing and right-continuous, and therefore it is also upper semi-continuous. To ensure $\nu \in \UC$ and that the process started from initial data (defined in Section \ref{sec:Ham_proof}) is well-defined, we impose the additional conditions
\be \label{Hnucond}
\liminf_{x \to -\infty} \f{\nu(x)}{x} > 0 \qquad\text{and}\qquad \limsup_{x \to +\infty} \f{\nu(x)}{x} < \infty.
\ee
For $\Aa \in \R$, let $\nu^\Aa$ denote a Poisson point process on $\R$ of intensity $e^{\Aa}$ (This convention allows us to utilize the EJS-Rains technique to obtain Lemma \ref{lem:PoiEJS}) . Corollary 5.5 of \cite{Cator-Pimentel-2012} states that, for $\beta > 0$,
\[
\{\B^\beta_\Hh((0,0),(x,0)): x \in \R\} \deq \{\nu^{\log \beta}(x):x \in \R\}.
\]
Here, we make the remark, that in \cite{Cator-Pimentel-2012},  the Busemann functions are parameterized by the angle $\alpha \in (\pi,3\pi/2)$. In our setting, $\alpha = \pi + \arctan \beta^2$. In all examples, $\nu$ will be a point process, in which case we may use the notation $x \in \nu$ to signify $\nu\{x\} = 1$.

We observe that the Busemann function is not continuous in $x$ because $\nu^{\log \beta}$ is not. In the lemma below, it is shown that the difference between the original function and the linearly interpolated approximating function, under diffusive scaling, converges to $0$ uniformly on compact sets. The following is a fairly straightforward corollary of Donsker's theorem, and the full details are provided in Appendix \ref{sec:BMRW}.  A similar statement was stated without proof in \cite[page 1275]{Cator-Groeneboom-06}, although the statement given below is slightly stronger, because it states that we may couple continuous functions together on the same probability space as $\nu_N$.
\begin{lemma} \label{lem:Poicoup}
Let $\nu_N(x)$ be the function associated to a rate $\mu_N$ Poisson point process on $\R$, as in \eqref{nufundef}. Assume that for some parameters $\beta,\chi > 0$ and $\mu \in \R$, we have that $\f{\mu_N - \beta }{\chi}\sqrt N \to \mu$. Then, on the same probability space where $\nu_N$ is defined, there exists a sequence of continuous functions $\{f_N\}$ so that $\f{f_N(xN) - \beta x N}{\chi \sqrt N}$ converges, in distribution, to a two-sided Brownian motion with diffusivity $\sigma = \f{\sqrt \beta}{\chi}$ and drift $\mu$, and so that, for each compact set $K$,
\[
\lim_{N \to \infty} \sup_{x \in K}\Bigl|\f{f_N(x) - \nu_N(x)}{\sqrt N}\Bigr| = 0.
\]
\end{lemma}

\subsubsection{Convergence of the Busemann process to SH}
The parameters for convergence of Poisson LPP to the DL in the sense of \eqref{dNtoL} are given as follows via \cite{Dauvergne-Virag-21}:
\be \label{Hamm_param}
\chi^3 = \sqrt \rho,\;\; \alpha = 2\sqrt \rho, \;\; \beta = \f{1}{\sqrt \rho}, \;\; \f{\chi}{\tau^2} = \f{1}{4\rho^{3/2}},
\ee
uniquely determined by the condition $\tau > 0$. We now state the result for convergence to the SH.
\begin{theorem} \label{thm:Hamm_SH}
    Fix $\rho > 0$, and define $\beta,\tau,\chi$ as in \eqref{Hamm_param}, $\iota_N$ as in \eqref{iotadef}, and $\beta_N(\mu)$ as in \eqref{meanparam}. For $\mu \in \R$, let
    \[
    f_\mu^N(x) = \B_\Hh^{\beta_N(\mu)}((0,0),(x,0)), \quad\text{and}\quad H_{\mu}^N = \iota_N f_\mu^N.
    \]
    Then, for $\mu_1 < \cdots < \mu_k$, we have the weak convergence on the space $\UC^k$.
    \[
    (H_{\mu_1}^N,\ldots,H_{\mu_k}^N) \Longrightarrow (G_{\mu_1},\ldots,G_{\mu_k}).
    \]
    Moreover, the processes can be coupled together so that convergence occurs almost surely uniformly on compact sets. 
\end{theorem}

\begin{proof}
We verify the assumptions of Theorem \ref{thm:LPPSHconv}. The marginal convergence to Brownian motion follows by Lemma \ref{lem:Poicoup}, handling Assumption \ref{ctight'} of Theorem \ref{thm:LPPSHconv}. Note that the construction of Lemma \ref{lem:Poicoup} can be done on the same probability space where the Poisson process is defined, so this handles the uniform convergence on compact sets. 

 From the invariance of Poisson LPP under space-time shifts and definition of the Busemann functions \eqref{eqn:HamBuse}, it follows that $\B^\beta_\Hh$ is stationary under space-time shifts and is additive. Namely, we have that for any $x,t \in \R$,
\be \label{Ham_stat}
\{\B_\Hh^\beta((x,t),(x+y,t)):y \in \R\} \deq \{\B_\Hh^\beta((0,0),(y,0)): y \in \R\},
\ee
where the equality in distribution holds as processes in $\UC$.
Furthermore, \cite[Theorem 3.2]{Cator-Pimentel-2012} states that, for $s < t$, and $y \in \R$
\[
\B^\beta_\Hh((0,s),(y,t)) = \sup_{-\infty < z \le x + y}\{\B^\beta_\Hh((0,s),(z,s)) + \Hh(z,s;y,t)\}.
\]
Hence, for any $t > 0$,
\begin{align*}
&\quad \; \B_{\Hh}^\beta((x,t),(x + y,t)) = \B_{\Hh}^\beta((0,0),(x+y,t))  - \B_{\Hh}^\beta((0,0),(x,t)) \\
&=  \sup_{-\infty < z \le x + y}\{\B^\beta_\Hh((0,0),(z,0)) + \Hh(z,0;x + y,t)\} - \sup_{-\infty < z \le x}\{\B^\beta_\Hh((0,0),(z,0)) + \Hh(z,0;x,t)\} \\
&= h_\Hh(t,x+y ; \B^\beta_\Hh((0,0),(\aabullet,0))) -h_\Hh(t,x ; \B^\beta_\Hh((0,0),(\aabullet,0))). 
\end{align*}
and this combined with \eqref{Ham_stat} gives us the space-time stationarity of Assumption \ref{cst'}.  Theorem \ref{thm:Hamexitpt} in the following section (along with the discussion around \eqref{exitpty}) verifies  Assumption \ref{coup'}. 
\end{proof}

\subsection{Poisson lines model} \label{sec:PoissonLines}
We consider next a semi-discrete version of Poisson LPP. It was first studied in \cite[Section 3]{Seppalainen-1998b}. Let $\{F_i\}_{i \in \Z}$ be an i.i.d. collection of rate one Poisson point processes on $\R$. For $x < y$, let $F_i(x,y]$ be the number of Poisson points in the interval $(x,y]$. For $x \le y$ and $m \le n$, set 
\[
\Uu(x,m;y,n) = \sup\Bigl\{\sum_{i = m}^n F_i(x_{i-1},x_i]: x = x_{m - 1} \le x_m \le \cdots \le x_{n - 1} \le x_n = y \Bigr\}.
\]
This is a semi-discrete version of  the Poisson LPP model in the following way: If $F_i$ has a point at $x$, we put a point in the plane at $(x,i)$. On the full probability event where, for $i \neq j$, the sets of points in $F_i$ and $F_j$ are disjoint, the value $\Uu(x,m;y,n)$ is the maximal number of points that can be collected in an up-right path from $(x,m)$ to $(y,n)$, not counting a potential point at $(x,m)$.
Just as for Poisson LPP,  our initial data $\nu$ will be a point process with associated function $\R \to \R$ defined in \eqref{nufundef}.  For initial data $\nu \in \mathcal N$, associated to time level $-1$, we set 
\[
h_\Uu(n,y;\nu) = \sup_{-\infty < x \le y}\{\nu(x) + \Uu(x,0;y,n)\}.
\]
In Lemma \ref{lem:SJ_joint_exist}, we prove the existence of the sequence of jointly invariant measures $f_{\overline \mu}^N$ with Poisson process marginals. The convergence in Theorem \ref{thm:lines_conv} does not rely on the uniqueness of such a measure in the prelimit, although one could conceivably prove uniqueness using standard techniques from \cite{Damron_Hanson2012, Ahlberg_Hoffman,Georgiou-Rassoul-Seppalainen-17b,Janjigian-Rassoul-2020b,Groathouse-Janjigian-Rassoul-2023}, should the need arise. 
For a fixed choice of direction $\rho > 0$, we define the quantities
\be \label{PoiLines_param}
\chi^3 = \sqrt \rho(1 + \sqrt \rho)^2,\quad \alpha = \rho + 2\sqrt \rho,\quad \beta = 1 + \f{1}{\sqrt \rho},\quad \f{\chi}{\tau^2} = \f{1}{4\rho^{3/2}}.
\ee
\begin{theorem} \label{thm:lines_conv}
Fix $\rho > 0$, and define $\beta,\tau,\chi$ as in \eqref{PoiLines_param}, $\iota_N$ as in \eqref{iotadef}, and $\beta_N(\mu)$ as in \eqref{meanparam}. For $N \ge 1$ and real numbers $\mu_1 < \cdots < \mu_k$, let $f_{\overline \mu}^N :=(f_{\mu_1}^N,\ldots,f_{\mu_k}^N)$ be any random element of $\UC^k$ so that $f_{\mu_i}^N$ is the function associated to a Poisson point process of intensity $\beta_N(\mu_k)$ {\rm(}as in \eqref{nufundef}{\rm)}, and so that for all $x \in \R$ and $n \ge 0$,
\[
f_{\overline \mu}^N \deq \{h_{\Uu}(n,x + y;f_{\overline \mu}^N) - h_{\Uu}(n,x;f_{\overline \mu}^N): y \in \R\}.
\]
For $N \ge 1$ and $i \in \{1,\ldots,k\}$, define $H_{\mu_i}^N = \iota_N f_{\mu_i}^N$. Then, the following weak convergence holds in $\UC^k$, and the processes may be coupled together so that the convergence holds almost surely on compact sets:
\[
(H_{\mu_1}^N,\ldots,H_{\mu_k}^N) \Longrightarrow (G_{\mu_1},\ldots,G_{\mu_k}).
\]
\end{theorem}

\begin{proof}[Proof of Theorem \ref{thm:lines_conv}]
The convergence of the model to the DL is from \cite{Dauvergne-Virag-21}. Using the same reasoning as for the Poisson LPP case to get uniform convergence on compact sets, the only remaining detail is to verify \eqref{exitpty}. This is handled in Theorem \ref{thm:PoiL_exit}.
\end{proof}

\subsection{Sepp\"al\"ainen-Johansson model} \label{sec:SJ}
The Sepp\"al\"ainen-Johansson (SJ) model was introduced in \cite{Seppalainen-1998} and further studied in \cite{Johansson-2001}. It was shown to converge to the DL in \cite{Dauvergne-Virag-21}. We describe this as follows. On the integer lattice $\Z^2$, we assign random weights to each edge. Vertical edges have weight $0$, while horizontal edges have i.i.d. Bernoulli weights with parameter $1-p \in (0,1)$. Let $t_e^f$ be the weight of edge $e$. For integers $k \le m$ and $j \le n$, we define
\[
\Tt_{1-p}^f(k,j;m,n) = \min_{\pi \in \Pi_{\mbf (k,j),(m,n)}} \sum_{e \in \pi} t_e^f,
\]
where $\Pi_{(k,j), (m,n)}$ is the set of up-right paths $\{\mbf x_k\}_{k = 0}^{n}$ that satisfy $\mbf x_0 = (k,j),\mbf x_{n} = (m,n)$, and $\mbf x_k - \mbf x_{k - 1} \in \{e_1,e_2\}$.
As defined, this is a model of first-passage percolation. To fit our setting, we reformulate it in the last passage setting. For a horizontal edge $e$, set $t_e = 1-t_e^f$  so that the edges $t_e$ are i.i.d.\ Bernoulli with success parameter $p$. Then, define
\[
\Tt_p(k,j;m,n) := (m-k) - \Tt_{1-p}^f(k,j;m,n) = \max_{\pi \in \Pi_{(k,j),(m,n)}} \sum_{e \in \pi: \, e \,\text{horizontal}} t_e.
\]
For this model, in the setting of Theorem \ref{thm:LPPSHconv}, $A = B = \Z$. In \cite{Dauvergne-Virag-21}, it was shown that
\[
\f{\Tt_p^f(s\rho N + x\tau' N^{2/3} ,sN; t\rho N + y\tau' N^{2/3},tN) - \alpha' N(t - s) -  \beta' \tau' N^{2/3}(y - x)}{\chi' N^{1/3}}
\]
converges to $\Ll(x,s;y,t)$, where, for a choice of $\rho > 0$ and $p \in (0,1)$, we set $\lambda' = \f{p}{1 - p}$ and choose the parameters
\[
\chi'^3 = \f{-\sqrt{\lambda'} (\sqrt{\lambda' \rho} - 1)^2(\sqrt \rho + \sqrt {\lambda'})^2}{\sqrt \rho (\lambda' + 1)^3},\quad \alpha' = \f{(\sqrt{\rho \lambda'} - 1)^2}{\lambda' + 1},\quad \beta' = \f{\lambda' - \sqrt{\lambda'/\rho}}{\lambda' + 1},\quad \f{\chi'}{\tau'^2} = \f{-\sqrt \lambda}{4(\lambda' + 1) \rho^{3/2}}.
\]
which uniquely define the parameters under the condition $\tau  > 0$. 
To obtain the convergence
\[
\f{d_p(s\rho N + x\tau N^{2/3} ,sN; t\rho N + y\tau N^{2/3},tN) - \alpha N(t - s) -  \beta \tau N^{2/3}(y - x)}{\chi N^{1/3}} \to \Ll(x,s;y,t),
\]
we therefore replace $p$ with $1-p$, set $\alpha = \rho - \alpha'$, and $\beta = \tau - \beta'$, $\tau' = \tau$, and $\chi' = -\chi$. In other words, we set $\lambda = \f{1-p}{p}$, and define the parameters by the condition $\tau > 0$, and
\be \label{SJ_param}
\chi^3 =  \f{\sqrt \lambda (\sqrt{\lambda \rho} - 1)^2(\sqrt \rho + \sqrt \lambda)^2}{\sqrt \rho (\lambda + 1)^3},\quad \alpha = \rho - \f{(\sqrt{\rho \lambda} - 1)^2}{\lambda + 1},\quad \beta = \tau - \f{\lambda - \sqrt{\lambda/\rho}}{\lambda + 1},\quad \f{\chi}{\tau^2} = \f{\sqrt \lambda}{4(\lambda + 1)\rho^{3/2}}.
\ee
We note here that the convergence to $\Ll$ holds only when $\rho > \lambda^{-1} = \f{p}{1-p}$. Otherwise, the shape function for the model is not strictly convex in that direction \cite{Seppalainen-1998}. In the sequel, we take $p$ as fixed and simply write $d$ in place of $d_p$.
Considering initial data $f:\Z \to \R$ on level $-1$ whose linear interpolation lies in $\UC$, define, for $n \ge 0$ and $m \in \Z$,
\begin{align*}
h_\Tt(n,m;f) &= \sup_{-\infty < k \le m}\{f(x) + \Tt(k,0;m,n)\}, \quad
Z_\Tt(n,m;f) = \max \argmax_{-\infty < k \le m}\{f(x) + \Tt(k,0;m,n)\}.   
\end{align*}
 
Our next theorem is of a similar type as Theorem \ref{thm:lines_conv}. Jointly invariant measures for the SJ model with i.i.d. Bernoulli marginals exist, and this is seen from the same proof (word-for-word) as the argument in Appendix \ref{sec:PoiL_exist} for the Poisson lines model, as the exact analogues of all needed inputs are provided in this section. In fact, the proof is simpler because the SJ model is fully discrete.   
\begin{theorem} \label{thm:SJ_conv}
Fix $\rho > \f{p}{1-p}$, and define $\beta,\tau,\chi$ as in \eqref{PoiLines_param}, $\iota_N$ as in \eqref{iotadef}, and $\beta_N(\mu)$ as in \eqref{meanparam}. For $N \ge 1$ and real numbers $\mu_1 < \cdots < \mu_k$,  and let $f_{\overline \mu}^N :=(f_{\mu_1}^N,\ldots,f_{\mu_k}^N)$ be any random element of $\UC^k$ such that each $f_{\mu_i}^N$ is the linear interpolation of an i.i.d. ${\rm Ber \rm}(\beta_N(\mu_k))$ random walk and so that for all $m \in \R$ and $n \ge 0$,
\[
f_{\overline \mu}^N \deq \{h_{\Uu}(n,m + y;f_{\overline \mu}^N) - h_{\Uu}(n,m;f_{\overline \mu}^N): y \in \R\}.
\]
For $N \ge 1$ and $i \in \{1,\ldots,k\}$, define $H_{\mu_i}^N = \iota_N f_{\mu_i}^N$. Then, the following weak convergence holds in $C(\R,\R^k)$: 
\[
(H_{\mu_1}^N,\ldots,H_{\mu_k}^N) \Longrightarrow (G_{\mu_1},\ldots,G_{\mu_k}).
\]
\end{theorem}
\begin{proof}
Just as for Theorem \ref{thm:lines_conv}, The only needed input is the exit point bounds, which are handled in Theorem \ref{thm:SJ_exit_pt}. 
\end{proof}

\subsection{Exponential and geometric LPP} \label{sec:latice_LPP}
LPP on the lattice $\Z^2$ is defined as follows. Let $\{Y_{\mbf x}\}_{\mbf x \in \Z^2}$ be   i.i.d. nonnegative random variables on the  vertices of the planar integer lattice.  For $\mbf x \le \mbf y \in \Z^2$, define the last-passage time 
\be\label{d100} 
\Dd(\mbf x;\mbf y) = \sup_{\mbf x_\centerdot \in \Pi_{\mbf x,\mbf y}} \textstyle\sum_{k = 0}^{|\mbf y - \mbf x|_1} Y_{\mbf x_k}, 
\ee
where $\Pi_{\mbf x, \mbf y}$ is the set of up-right paths $\{\mbf x_k\}_{k = 0}^{n}$ that satisfy $\mbf x_0 = \mbf x,\mbf x_{n} = \mbf y$, and $\mbf x_k - \mbf x_{k - 1} \in \{\mbf e_1,\mbf e_2\}$. A maximizing path is called a geodesic. The cases where the weights $Y_{\mbf x}$ are $\Exp(1)$ or when they have geometric distribution with mean $\gamma > 0$ (equivalently, $Y_{\mbf x} \sim {\rm Geom}(v)$ with $v = \f{1}{1 + \gamma}$).

The existence of Busemann functions in LPP is known for a general class of weights \cite{Georgiou-Rassoul-Seppalainen-17b}. In the exponential case, we index the mean of the Busemann process in terms of a real parameter $\beta > 1$:  For $\mbf x,\mbf y \in \Z^2$,
\[
\B^\beta_\Dd(\mbf x,\mbf y) = \lim_{n \rightarrow \infty} \Dd\Bigl(-n,- \lfloor (\beta-1)^2 n \rfloor;\mbf y\Bigr) - \Dd\Bigl(-n,- \lfloor (\beta-1)^2 n \rfloor;\mbf x\Bigr). 
\]
 For geometric random variables with mean $\gamma > 0$ and supported on $\{0,1,2,\ldots\}$, we define $\overline \gamma = \gamma(\gamma + 1)$. Then, we parameterize the Busemann functions for $\beta > \gamma$ as
 \[
\B^\beta_\Dd(\mbf x,\mbf y) = \lim_{n \to \infty} \Dd\Bigl(- \lfloor \overline \gamma^2 n  \rfloor ,- \lfloor (\beta -\gamma)^2n \rfloor ;\mbf y\Bigr) - \Dd\Bigl(- \lfloor \overline \gamma^2 n  \rfloor ,- \lfloor (\beta -\gamma)^2n \rfloor ; \mbf x\Bigr).
 \]
  In the exponential (resp. geometric) case, $\{B^\rho_{ie_1,(i+1)e_1}\}_{i \in \Z}$ is an i.i.d. sequence of exponential (geometric) random variables with mean $\beta$ \cite[Equation 3.9 and Section 7.1]{Georgiou-Rassoul-Seppalainen-17b} (see also \cite{Sepp_lecture_notes}) We note that in \cite{Georgiou-Rassoul-Seppalainen-17b}, Busemann functions are defined for endpoints going to $\infty$ in the northeast directions, while our starting points travel southwest. The change is made by a simple reflection. See, for example, \cite[Lemma 4.3(iii)]{Timo_Coalescence}.

In both exponential and geometric LPP, the convergence of the rescaled passage times to the DL \eqref{dNtoL} was proven in \cite{Dauvergne-Virag-21}. In the exponential case, for a choice of $\rho > 0$,
\be \label{exp_param}
\chi^3 = \f{(\sqrt \rho + 1)^4}{\sqrt \rho},\quad \alpha = (\sqrt \rho + 1)^2,\quad \beta = 1 +\f{1}{\sqrt \rho},\quad \f{\chi}{\tau^2} = \f{1}{4\rho^{3/2}}.
\ee
In the geometric case, we have 
\be \label{geom_param}
\chi^3 = \f{\overline \gamma(\overline \gamma(1 + \rho) + (2\gamma+1)\sqrt \rho  )^2}{\sqrt \rho},\quad \alpha = \gamma(\rho + 1) + 2\overline \gamma \sqrt \rho,\quad \beta = \gamma + \f{\overline \gamma}{\sqrt \rho},\quad \f{\chi}{\tau^2} = \f{1}{4\rho^{3/2}}.
\ee
These uniquely define the parameters with the additional condition $\tau > 0$. 
\begin{theorem} \label{thm:exp_geom}
    In exponential or geometric last-passage percolation, fix $\rho > 0$, and for $\mu \in \R$ and $\beta,\tau,\chi$ defined in \eqref{exp_param} for exponential LPP and \eqref{geom_param} for geometric LPP, let $\iota_N$ be defined as in \eqref{iotadef}, and let $\beta_N(\mu)$ be defined by \eqref{meanparam}. Let
    \[
    f_\mu^N \text{ be the linear interpolation of the discrete function } \{\B_{\Dd}^{\beta_N(\mu)}(\mbf 0, ie_1)\}_{i \in \Z}.
    \]
    For $\mu,x \in \R$, define $H_{\mu}^N = \iota_N f_\mu^N$.    Then, for each $\mu_1 < \cdots < \mu_k$, we have the weak convergence on the space $C(\R,\R^k)$.
    \[
    (H_{\mu_1}^N,\ldots,H_{\mu_k}^N) \Longrightarrow (G_{\mu_1},\ldots,G_{\mu_k}).
    \]
\end{theorem}
\begin{proof}
We verify each of the assumptions of the Theorem \ref{thm:LPPSHconv}. Assumption \ref{ctight'} follows from the general discussion in Section \ref{sec:gen_solv}. 

For the space-time stationarity of Assumption \ref{cst'}, the dynamic programming principle holds for the Busemann functions (For example, by \cite[Lemma 3.3]{Fan-Seppalainen-20} and \cite[Theorem A.2]{Groathouse-Janjigian-Rassoul-21}). That is, for any $\beta > 1$, with probability one, for $\ell < n$,
\[
\B^\beta_\Dd((r,\ell),(m,n)) = \sup_{-\infty < k \le m} \{\B^\beta_\Dd((r,\ell),(k,\ell)) + \Dd(k,\ell + 1;m,n)\}.
\]
The joint space-time stationarity of Assumption \ref{cst'} now follows by the same argument as in the proof of Theorem \ref{thm:Hamm_SH}. 

For the tightness of the exit point bounds in Assumption \ref{ctight'}, we have established that we need only show \eqref{exitpty} for a fixed $y \in \R$. in the exponential case, this was shown in \cite{Emrah-Janjigian-Seppalainen-20}, Theorem 2.5 (See also~\cite{Bhatia-2020}, Theorem 2.5,~\cite{BasuSarkarSly_Coalescence}, Theorem 3,~\cite{Martin-Sly-Zhang-21}, Lemma 2.8 and~\cite{Seppalainen-Shen-2020}, Corollary 3.6 and Remark 2.5b). Each of these, however, deals with exit point bounds from the stationary initial condition in the quadrant, so there is some nontrivial work to be done to extend it to the stationary model in the half-plane. This is precisely what is done in \cite[Lemma C.5 (arXiv version)]{Busa-Sepp-Sore-22a}.  For the geometric case, the exit point bounds from the quadrant are handled using the EJS-Rains formula for geometric LPP in \cite[Theorem B.1]{Groathouse-Janjigian-Rassoul-21}. The extension to the half-plane case follows the same as for the exponential case. In particular, \cite[Lemma C.3 (arXiv version)]{Busa-Sepp-Sore-22a} shows that one can couple the half-plane exponential stationary model with a model in the quadrant, where the almost surely unique geodesics agree in the quadrant. In the geometric case, geodesics are not unique, but one can replace this condition with leftmost or rightmost geodesics. The proof of the exit point bounds we need then follows exactly as for the exponential case in the proof of Lemma C.5 in the arXiv version of \cite{Busa-Sepp-Sore-22a}. The only needed input is a queuing result for geometric arrival times instead of exponential arrival times, and \cite[Lemma A.2]{Groathouse-Janjigian-Rassoul-21} exactly provides this needed input. Since we parameterized the Busemann process by the mean $\beta$, we only need to verify that, when the mean is perturbed from a value $\beta$ on the order $N^{-1/3}$, the associated direction is perturbed on the order of $N^{-1/3}$. This follows from the relations for exponential/geometric LPP in \eqref{exp_param} and \eqref{geom_param}, namely
\[
\rho = \f{1}{(\beta-1)^2} \text{ (exponential)},\quad \text{and}\qquad \rho = \f{\overline \gamma^2}{(\beta - \gamma)^2} \text{ (geometric)}.
\]
and a Taylor expansion, similarly as is handled in the proof of Theorem \ref{thm:Hamexitpt}.
\end{proof}

\subsection{Brownian LPP}
We now define Brownian last-passage percolation (BLPP). Let $\mathbf B = \{B_i\}_{i \in \Z}$ be a field of independent, two-sided standard Brownian motions. For $x \le y$ and $m \le n$, define
\be \label{BLPPdef}
\BL(x,m;y,n) = \sup\Bigl\{\sum_{i = m}^n B_i(x_i) - B_i(x_{i-1}): x = x_{m - 1} \le x_m \le \cdots \le x_{n - 1} \le x_n = y \Bigr\}.
\ee
 In the notation of Theorem \ref{thm:LPPSHconv}, $A = \R$ and $B = \Z$. The Busemann function parameterized by the mean $\beta > 0$ is defined for $\mbf x,\mbf y \in \R \times \Z$ by 
\[
\B^\beta_\BL(\mbf x,\mbf y) = \lim_{n \to \infty} \BL\Bigl(\Bigl(-\f{n}{\beta^2},-n\Bigr),\mbf x\Bigr) - \BL\Bigl(\Bigl(-\f{n}{\beta^2},-n\Bigr),\mbf y\Bigr).
\]
These Busemann limits were first shown to exist in \cite{blpp_utah}, and the process in the direction parameter was studied by the second and third author in \cite{Seppalainen-Sorensen-21a,Seppalainen-Sorensen-21b}. In fact, the following is true. Once the Busemann process is properly extended to a right-continuous process in the $\beta$ parameter, we have
\[
\Bigl\{\B^\beta_\BL((0,0),(2\,\abullet ,0)): \beta > 0 \Bigr\} \deq \{G_\beta: \beta > 0\}.
\]
This distributional equality was proved in \cite[Theorem 5.4]{Seppalainen-Sorensen-21b}, noting that the parameterization used there, originally from \cite{Busani-2021} is different. See \cite[page 27]{Sorensen-thesis} for a discussion of the different parameterizations of the SH used in the literature.  When centered around a given direction in the KPZ scaling, we recover the full SH, as opposed to the SH restricted to positive parameters.  This is the statement of Theorem \ref{BLPP_conv} below, which was shown previously in \cite[Theorem 5.4]{Seppalainen-Sorensen-21b}, under a different parameterization. See also \cite[Theorem 2.3.2(ii) and Equation (3.4.6)]{Sorensen-thesis}. While not needed here because of the scaling invariance, the analogous exit-point result using the EJS-Rains technique was proved previously as \cite[Theorem 3.4.1]{Sorensen-thesis}. For a choice of $\rho > 0$, the scaling parameters from \cite{Dauvergne-Virag-21} in this case are
\be \label{BLPP_param}
\chi^3 = \rho^{3/2},\;\; \alpha = 2\sqrt \rho, \;\; \beta = \f{1}{\sqrt \rho}, \;\; \chi/\tau^2 = \f{1}{4\rho^{3/2}},
\ee
which uniquely defines the parameters under the assumption $\tau > 0$.
\begin{theorem} \label{BLPP_conv}
    Fix $\rho > 0$, define $\beta,\tau,\chi$ as in \eqref{BLPP_param}, and define $\beta_N(\mu)$ as in \eqref{meanparam}. For $\mu \in \R$, let
    \[
    f_\mu^N(x) = \B_{\BL}^{\beta_N(\mu)}((0,0),(x,0)), \quad\text{and}\quad H_{\mu}^N = \iota_N f_\mu^N.
    \]
    Then, for $\mu_1 < \cdots < \mu_k$ and sufficiently large $N$ so that $\beta_N(\mu_1) > 0$, we have the following equality in distribution on the space $C(\R,\R^k)$:
    \[
    (H_{\mu_1}^N,\ldots,H_{\mu_k}^N) \deq (G_{\mu_1},\ldots,G_{\mu_k}).
    \]
\end{theorem}

\section{Details of the exit point bounds for the Poisson LPP, Poisson lines, and SJ models} \label{sec:Details}

In this section, we shall often make reference to certain queuing mappings. Indeed, LPP is intimately tied to queuing theory. As a simple example, in lattice LPP described in Section \ref{sec:latice_LPP}, we assign i.i.d. weights $\{\omega_{(m,n)}\}_{(m,n) \in \Z^2}$ to each vertex of $\Z^2$. Passage times $D$ from $0$ evolve via the local rule
\[
D((0,0);(m,n)) = D((0,0);(m-1,n)) \vee D((0,0);(m,n-1)) + \omega_{(m,n)}.
\]
For a series of queues in tandem, where $\omega_{(m,n)}$ denotes the service time of customer $m$ and queue $n$, we can then interpret $D(m,n)$ as the time when customer $m$ departs from queue $n$. See \cite{glynn1991,Baccelli-2000,Draief-2005,brownian_queues,Georgiou-Rassoul-Seppalainen-17b,Fan-Seppalainen-20,Seppalainen-Sorensen-21a} and the references therein for more on the connection between LPP and queues.

\subsection{Poisson LPP} 

\subsubsection{Hammersley process as a particle system} \label{sec:Ham_proof}
Given an initial measure $\nu \in \mathcal N$ satisfying \eqref{Hnucond}, we define the \textit{Hammersley interacting fluid process} started from the initial measure $M_0^\nu:= \nu$ as follows. Take a Poisson process $\mbf X$ of intensity one in $\R^2$. For a Borel set $A \subseteq \R$, let $M_{t-}^\nu(A) = \lim_{s \nearrow t} M_s^\nu(A)$. Whenever $(x_0,t) \in \mbf X$, $M_t^\nu(\{x_0\}) = M_{t-}^\nu(\{x_0\})+ 1$, and for $x > x_0$, $M_t^\nu(x_0,x] = (M_t^\nu(x_0,x] - 1) \vee 0$.  . This process was first introduced in \cite{Aldous-Diaconis-1995}. There, (see also Theorem 4.1 in \cite{Cator-Pimentel-2012}), it is shown that we may construct this process via a variational formula: Using the same Poisson point process $\mbf X$ that we used to construct the passage times $d$, for $\nu \in \mathcal N$ satisfying \eqref{Hnucond}, $t > 0$, and $y \in \R$, define
\be \label{Hamm_meas_evolve}
h_\Hh(t,y; \nu) = \sup_{-\infty < x \le y} \{\nu(x) + \Hh(x,0;y,t)\}, \quad\text{and for $x < y$,}\quad M_t^\nu(x,y] = h_\Hh(t,y; \nu) - h_\Hh(t,x;\nu).
\ee
In all applications we consider here, $\nu$ is a point process. In this case, we consider the evolution of $M^\nu$ as the movement of particles. The points of $\nu$ denote the initial configuration of particles. When we meet a Poisson point $(x_0,t) \in \mbf X$ for $t > 0$, the particle closest to the right of $x_0$ in the configuration at time $t-$ moves to $x_0$. These dynamics are well-defined under assumption \eqref{Hnucond}.  We consider trajectories of particles as up-left paths, where the left steps take place when the particles change their position. The following lemma follows from definition and the almost sure condition that no two Poisson points will lie along the same horizontal or vertical line. 
\begin{lemma} \label{lem:dist_traj}
Let $\nu$ be a point process satisfying \eqref{Hnucond} almost surely. Then, with probability one, for any two distinct particles of  $\nu$, their trajectories under the Hammersley process are disjoint paths. If $x < y$ are points of $\nu$, then the trajectory of $y$ lies above and to the right of the trajectory of $x$.
\end{lemma}

We prove the following intermediate lemma.
\begin{lemma} \label{lem:hdxto0}
There exists an event of full probability on which, for all $y \in \R$ and $t > 0$, 
\be \label{dxto0}
\lim_{x \to -\infty}\f{\Hh(x,0;y,t)}{x} = 0.
\ee
On this event, whenever $t > 0$, and $\nu \in \mathcal N$ satisfies \eqref{Hnucond},
\be \label{hdxto0}
\liminf_{y \to -\infty} \f{h_\Hh(t,y;\nu)}{y} \ge \liminf_{y \to -\infty} \f{\nu(y)}{y}.
\ee
\end{lemma}
\begin{proof}
The fact that \eqref{dxto0} holds almost surely for each fixed $y,t$ is a quick consequence of the shape theorem for the Hammersley process, along with scaling invariance. It is shown, for example, in \cite[Equation 12]{Aldous-Diaconis-1995}. Hence, \eqref{dxto0} holds on a full probability event for all $y \in \Q$ and $t \in \Q_{>0}$. For arbitrary $y \in \R$ and $t  > 0$, we may find rational pairs $(r_1,q_1)$ and $(r_2,q_2)$ so that $r_1 < y < r_2$ and $q_1 < t < q_2$, and so
\[
\Hh(x,0;r_1,q_1) \le \Hh(x,0;y,t) \le \Hh(x,0;r_2,q_2),
\]
so \eqref{dxto0} holds on this event as well for all $(t,y) \in \R^2$.
Now, assume that $\nu \in \mathcal N$ satisfies \eqref{Hnucond}, and let 
\[
c := \liminf_{x \to -\infty} \f{\nu(x)}{x} > 0.
\]
Let $\ve < c/2$ and $t > 0$. There exists $K > 0$ so that when $x < -K$,
\[
\Hh(x,0;0,t) \le \ve x,\quad\text{and}\quad \nu(x) \le (c - \ve)x. 
\]
Then,  for $y < -K$,
\begin{align*}
h_\Hh(t,y;\nu) = \sup_{-\infty < x \le y}\{\nu(x) + \Hh(x,0;y,t)\} &\le \sup_{-\infty < x \le y}\{\nu(x) + \Hh(x,0;0,t)\} \\
&\le \sup_{-\infty < x \le y}\{(c - 2\ve)x   \} = (c-2\ve)y.
\end{align*}
Hence, for each $\ve > 0$,
\[
\liminf_{y \to -\infty} \f{h_\Hh(t,y;\nu)}{y} \ge c - 2\ve. \qedhere
\]
\end{proof}

If $\nu \in \mathcal N$ is a point process satisfying \eqref{Hnucond} almost surely, we define the point process $\eta^\nu_x$ as the locations of particle trajectories crossing the vertical ray $\{x\} \times (0,\infty)$ (we do not include the particle trajectory of the point $(x,0)$ if $x \in \nu$). For $t > 0$, let $\eta^\nu_x(t)$ be the number of points of $\eta^\nu_x$  in $ \{x\} \times (0,t]$. We now prove the following lemma.
\begin{lemma} \label{lem:Hamquad_process}
 With probability one, simultaneously  for all $y \in \R$, $t> 0$, and $\nu \in \mathcal N$ satisfying \eqref{Hnucond},
\be \label{Lhquad}
h_\Hh(t,y;\nu) = \nu(y) + \eta^\nu_y(t).
\ee
\end{lemma}
\begin{proof}
We work on the intersection of the full probability event of Lemma \ref{lem:hdxto0} with the full-probability event where  $\mbf X$ is locally finite, no two points of $\mbf X$ lie on the same horizontal or vertical line, and there exist no points of $\mbf X$ on the horizontal axis,

    Let $x \le y$, and consider a maximal path from $(x,0)$ to $(y,t)$ for $d$. By Lemma \ref{lem:dist_traj}, every point along the maximal path lies on the trajectory of a distinct point for the point process $\nu$. Since trajectories are up-left paths, these trajectories must have originated from a point of $\nu$ in $(x,y]$ (along the horizontal boundary) or a point of $\eta^\nu_x$ in $\{x\} \times [0,t]$. Hence, 
    \be \label{upbdHam}
    \nu(x) + \Hh(x,0;y,t) \le \nu(x) + (\nu(y) - \nu(x)) + \eta^\nu_y(t) = \nu(y) + \eta^\nu_y(t).
    \ee
    Taking the supremum over $x \le y$ yields $h_\Hh(y,t;\nu) \le \nu(y) + \eta^\nu_y(t)$.
    We turn to proving the opposite inequality. We first make the observation that $\eta^\nu_y(t)$ must be finite. If not, by the monotonicity of Lemma \ref{lem:dist_traj}, every particle trajectory of a point in $\nu$ over $(y,\infty)$ must cross the line segment $\{y\} \times [0,t]$. But then, when the process meets points  of $\mbf X$ in $(y,\infty) \times (t,\infty)$, there is no closest particle from the right to map to. This does not happen under the condition \eqref{Hnucond}, as shown in \cite{Aldous-Diaconis-1995}. 

    Now, if $\eta^\nu_y(t) = 0$, then $h_\Hh(t,y;\nu) \ge \nu(y) + \Hh(y,0;y,t) = \nu(y) + \eta^\nu_y(t)$. Hence, we will now assume $\eta^\nu_y(t) > 0$.  Consider the following construction of a down-left path. The construction is illustrated in Figure \ref{fig:Hammersley1}. 
    Starting from the point $(y_0,t_0) = (y,t)$, move downwards until reaching the topmost point $(y_0,t_{-1})$ of $\eta^\nu_y \cap (\{y\} \times [0,t])$ (this could be $(y,t)$ itself). If $(y_0,t_{-1}) \in \mbf X$ (in which case this is the only point of $\mbf X$ along that vertical line ray a.s.), then continue moving downward to the next point (if no such point exists, then $\Hh(y,0;y,t) = 1$ and then follow a similar argument as the first case). Otherwise, follow the trajectory of the particle $(y_0,t_{-1})$ leftwards until reaching a point of $\mbf X$, which we call $(y_{-1},t_{-1})$. Move downward along a vertical line until reaching either the horizontal axis or until reaching the next particle trajectory. Call this point $(y_{-1},t_{-2})$. If we reach the horizontal axis, we terminate the procedure. Otherwise, assuming that we have determined the point $(y_{-i},t_{-(i+1)})$ for $i \ge 1$ with $t_{-(i+1)} > 0$, we note that almost surely (since no two points of $\mbf X$ lie on the same vertical line), we may move leftward along that point's trajectory until reaching another point $(y_{-(i+1)},t_{-(i+1)})$. Then, move downward again until reaching a point on the horizontal axis or a point on the next particle trajectory. We argue two things. First, we show that if the path ever terminates at the point $(y_{-i},0)$, then $\nu(y_{-i}) + \Hh(y_{-i},0;y,t) = \nu(y) + \eta^\nu_y(t)$, and second, we show that the path must terminate. 
    \begin{figure}
        \centering
        \includegraphics[height = 3in]{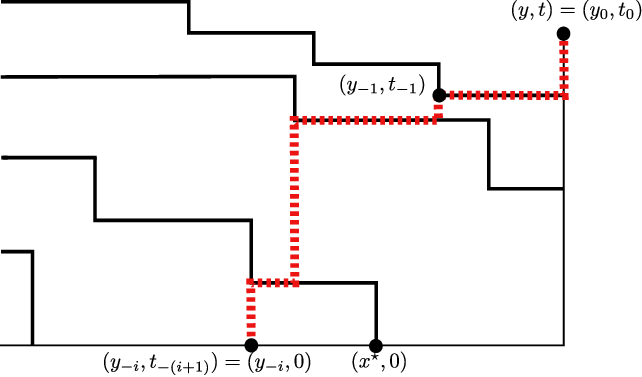}
        \caption{\small An illustration of the construction of the down-left path. The horizontal axis and the ray $\{y\} \times (0,\infty)$ are in black/light. The particle trajectories are in black/medium thickness. There are two particles crossing the line $\{y\} \times (0,\infty)$and three particles emanating from the horizontal axis  in the window shown. The path we construct is red/dashed. In this case, the procedure terminates at the point $(y_{-i},t_{-(i+1)})$, where $t_{-(i+1)} = 0$. The point $(x^\star,0)$ is the particle along the horizontal axis that generated the last trajectory travelled by the red path. There can be no other particles $(x,0)$ for $x \in (y_{-i},x^\star)$; otherwise, the red path would follow another particle trajectory before terminating.}
        \label{fig:Hammersley1}
    \end{figure}
    
    By the monotonicity in Lemma \ref{lem:dist_traj}, the down-left path we constructed moves from the trajectory of the topmost particle in $\eta^\nu_y \cap (\{y\} \times [0,t])$, then moves to the trajectory of the next topmost particle, and so on until there are no more such particle trajectories. Then, if the path continues, it starts at the trajectory of the rightmost point in $\nu \cap (-\infty,y]$, then to the next rightmost, and so on. If the path terminates at the point $(y_{-i},0)$, the point of $\nu$ closest to the right of $(y_{-i},0)$ must be the point from which the trajectory containing $(y_{-i},t_{-i})$ originated. Otherwise, the path must have met another particle trajectory before reaching the horizontal axis. See Figure \ref{fig:Hammersley1}. By reversing the order of the down-left path constructed, we have thus found an up-right path from $(y_{-i},0)$ to $(y,t)$ that collects $\eta^\nu_y(t) + \nu(y_{-i},y]$ points. Then, using \eqref{nuinc}, we obtain
    \[
    h_\Hh(t,y;\nu) \ge \nu(y_{-i}) + \Hh(y_{-i},0;y,t) \ge \nu(y) + \eta^\nu_y(t).
    \]
    Now, we complete the proof by arguing that the path must terminate. Suppose, by way of contradiction, that it does not. The sequence $y_{-i}$ must go to $-\infty$ as $i \to \infty$ almost surely; otherwise, $\mbf X$ must contain infinitely many points in the compact window $[\inf y_{-i},y] \times [0,t]$.  We observe next that 
    \be \label{dHambd}
    \Hh(y_{-i},t_{-(i+1)};y,t) \ge M_{t_{-(i+1)}}^\nu(y_{-i},y]
    \ee
    because the trajectories of points of $M^\nu_{t_{-(i+1)}}(y_{-i},y]$ must have crossed the line segment $\{y\} \times [0,t]$ or originated from the horizontal ray $\{0\} \times (-\infty,y]$ and remain in the interval $(y_{-i},y]$ at time $t_{-(i+1)}$. This latter set is exactly the set of trajectories containing points of $\mbf X$ in the path from $(y_{-i},t_{-(i+1)})$ to $(y,t)$ that we constructed. Furthermore, since $t \ge t_{-(i+1)}$ and since particles move up and to the left as time increases,  
    \be \label{MtHambd}
    M_t^\nu(y_{-i},y] \le M^\nu_{t_{-(i+1)}}(y_{-i},y] + \eta^\nu_y(t).
    \ee
        Lastly, since $t_{-(i+1)} > 0$, $\Hh(y_{-i},0;y,t) \ge \Hh(y_{-i},t_{-(i+1)};y,t)$. 
    Combining this with \eqref{dHambd} and \eqref{MtHambd}, we obtain, for all $i \ge 1$, 
    \be \label{hdineq}
     h_\Hh(t;y_{-i};\nu) - h_\Hh(t,y;\nu) + \Hh(y_{-i},0;y,t)  = -M_t^\nu (y_{-i},y] + \Hh(y_{-i},0;y,t)  \ge - \eta^\nu_y(t).
    \ee
     Lemma \ref{lem:hdxto0} implies that
    \[
    \limsup_{x \to -\infty} h_\Hh(t,x;\nu) + \Hh(x,0;y,t) = -\infty.
    \]
    Since $y_{-i} \to -\infty$, this contradicts \eqref{hdineq}.
\end{proof}

\subsubsection{Coupling the Poisson processes and the EJS-Rains identity for Poisson LPP} \label{sec:Poiscoup}
We now define a coupling of Poisson point processes as follows. We recall the convention from Section \ref{sec:PoissonLPP} that $\nu^\Aa$ is a Poisson point process of intensity $e^\Aa$. Fix some large parameter $\Lambda > 0$. On a suitable probability space, define a Poisson point process $\nu^{\Lambda}$ on $\R$ of intensity $e^{\Lambda}$, and let $\{U_i\}_{i \in \Z}$ be and i.i.d. sequence of uniform random variables on $[0,1]$. We may split $\nu^\Lambda$ into two independent Poisson point processes by considering the negative and positive points of $\nu^\Lambda$ separately. Enumerate the positive points as $x_0,x_1,\ldots$ and enumerate the negative points as $x_{-1},x_{-2},\ldots$.  For two parameters $\rho,\lambda  \in (-\infty,\Lambda]$, we define the process $\nu^{\rho,\lambda}$ on this probability space as follows. For $i \ge 0$, if $U_i \le \f{e^{\lambda}}{e^{\Lambda}}$, we keep the point $x_i$, otherwise, we discard it. For $i < 0$, if $U_i \le \f{e^{\rho}}{e^{\Lambda}}$, we keep the point $x_i$ and otherwise discard it. Then, $\nu^{\rho,\lambda} |_{(-\infty,0)}$ and $\nu^{\rho,\lambda} |_{[0,\infty)}$ are independent Poisson processes of intensity $e^{\rho}$ and $e^{\lambda}$, respectively. For $\rho = \lambda$, we write $\nu^\lambda = \nu^{\lambda,\lambda}$. Then, in the notation of \eqref{nufundef}, $\nu^{\rho,\lambda}(x) = \nu^\rho(x)$ for $x < 0$ and $\nu^{\rho,\lambda}(x) = \nu^\lambda(x)$ for $x \ge 0$. Sending $\Lambda \to \infty$ and using Kolmogorov's extension theorem, we can couple the processes $\nu^{\rho,\lambda}$ together in this monotone coupling for all $\rho,\lambda \in \R$. We take this coupling to be independent of the Poisson point process $\mbf X$ that defines $d$, and let $\Pp$ be the probability measure on this space. We now define
\be \label{hrlev}
h_{\Hh}^{\Aa,\Bb}(t,y) = h_\Hh(t,y;\nu^{\Aa,\Bb}) = \sup_{-\infty < x < 0}\{\nu^\Aa(x) + \Hh(x,0;y,t)\} \vee \sup_{0 \le x \le y}\{\nu^\Bb(x) + \Hh(x,0;y,t)\},
\ee
and 
\be \label{Zrldef}
Z_\Hh^{\Aa,\Bb}(t,y) = \sup \argmax_{x \in \R}\{\nu^{\Aa,\Bb}(x) + \Hh(x,0;y,t)\}.
\ee
We set $h_{\Hh}^{\Bb}(t,y) = h_{\Hh}^{\Bb,\Bb}(t,y)$ and $Z_\Hh^{\Bb}(t,y) = Z_\Hh^{\Bb,\Bb}(t,y)$.
Recall that we defined $\eta_x^\nu$ to be the locations of particle trajectories of particles from $\nu$  crossing the ray $\{x\} \times (0,\infty)$. For shorthand notation, for the processes $\nu^{\Aa,\Bb}$ we defined, we set $\eta_x^{\Aa,\Bb} = \eta_x^{\nu^{\Aa,\Bb}}$, and $\eta_x^\Bb = \eta_x^{\Bb,\Bb}$.  We use the following lemma, noting that our notation flips the roles of $y$ and $t$ from that in \cite{Aldous-Diaconis-1995}. As defined, $\eta_x^{\Aa,\Bb}$ is  a point process on the set $\{x\} \times (0,\infty)$. In the following lemma, we interchangeably use $\eta_x^{\Aa,\Bb}$ to denote the projection of this point process onto $(0,\infty)$.
\begin{lemma} \cite[Lemma 8]{Aldous-Diaconis-1995} \label{lem:InvPois}
For $\Bb \in \R$, $\eta_x^{\Bb}$ is a Poisson point process of intensity $e^{-\Bb}$ on $\R_{>0}$. In particular, for $t > 0$, $\eta_x^\Bb(t)$ has the Poisson distribution with mean $te^{-
\Bb}$.     
\end{lemma}
We now prove a version of the Cameron-Martin-Girsanov theorem for Poisson processes.
\begin{lemma} \label{lem:HamRND}
    For $\Bb \in \R$, let $\Pp_\Bb$ denote the measure of a Poisson process $\nu$ on $[0,y]$ with intensity $e^{\Bb}$. For $\Aa,\Bb \in \R$, $\Pp_{\Aa}$ and $\Pp_{\Bb}$ are mutually absolutely continuous, with
    \[
    \f{d\Pp_{\Aa}}{d\Pp_{\Bb}} = \exp\Bigl(e^{\Bb} y - e^{\Aa} y + (\Aa - \Bb)(\nu[0,y])\Bigr).
    \]
\end{lemma}
\begin{proof}
For a measurable set $\mathcal A$ of the configuration space, define
\[
\wh \Pp_\Aa(\mathcal A)  = \E_\Bb\Bigl[ \exp\Bigl(e^{\Bb} y - e^{\Aa} y + (\Aa - \Bb)(\nu[0,y])\Bigr)\ind_{\mathcal A}   \Bigr].
\]
Let $A_1,\ldots,A_k$ be disjoint Borel subsets of $[0,y]$. We show that, under $\wh \Pp_\Aa$, $\nu(A_1),\ldots,\nu(A_k)$ are independent Poisson random variables with means $e^{\Aa}|A_1|,\ldots,e^{\Aa}|A_k|$, where $|\abullet|$ denotes Lebesgue measure. We show this via moment generating functions. First, observe that, for a Borel set $A \subset \R$ and $w \in \R$, 
\[
\E_\Bb[\exp(w \nu(A))] = \exp\Bigl(e^\Bb |A| (e^w - 1)\Bigr).
\]
Then,
\begin{align*}
    &\quad \; \wh \E_\Aa\Bigl(\exp(\sum_{i = 1}^k t_i \nu(A_i))\Bigr) = \E_\Bb\Bigl[\exp\Bigl(e^{\Bb} y - e^{\Aa} y + (\Aa - \Bb)\nu[0,y] + \sum_{i = 1}^k t_i \nu(A_i)\Bigr)\Bigr] \\
    &=\E_\Bb \Bigl[\exp\Bigl(e^\Bb y - e^\Aa y + \sum_{i = 1}^k(\Aa - \Bb + t_i)\nu(A_i) + (\Aa -  \Bb) \nu([0,y]\setminus \cup A_i)   \Bigr)\Bigr] \\
    &= \exp\Bigl(e^\Bb y - e^\Aa y + \sum_{i = 1}^k e^{\Bb}|A_i|(e^{\Aa -  \Bb + t_i} - 1) + e^{\Bb}(y - \sum |A_i|)(e^{ \Aa - \Bb} - 1)     \Bigr) \\
    &= \exp\Bigl(\sum_{i = 1}^k e^\Aa |A_i|(e^{t_i} - 1)\Bigr) = \prod_{i =1}^k \E_\Aa(e^{t_i \nu(A_i)}). \qedhere
\end{align*}
\end{proof}
Now, we make the following definitions for $\Aa,\Bb,y \in \R$ and $t > 0$. The computation of the expectation follows from Lemmas \ref{lem:Hamquad_process} and \ref{lem:InvPois}, and the integral and maximum follow from routine calculations. 
\begin{align*}
M_\Hh^\Aa(t,y) &:= \E[h_{\Hh}^\Aa(t,y)] =  e^\Aa y + e^{-\Aa} t,\qquad R_\Hh^{\Aa,\Bb}(t,y) := \int_{\Bb}^\Aa \quad \; M_\Hh^w(t,y)\,dw = (e^{\Aa} - e^{\Bb}) y + (e^{-\Bb} - e^{-\Aa}) t, \\
\zeta_\Hh(t,y) &:= \arg \inf_{\Aa \in \R} M_\Hh^\Aa(t,y) = \f{1}{2}\log\Bigl(\f{t}{y}\Bigr),\quad \gamma_\Hh(t,y) := \inf_{\Aa \in \R} M_\Hh^\Aa(t,y) = e^{\zeta}y + e^{-\zeta}t = 2\sqrt{yt}.
\end{align*}
We now prove the Poisson LPP version of the moment generating function identity in~\cite{Emrah-Janjigian-Seppalainen-20,Rains-2000}.
\begin{lemma} \label{lem:PoiEJS}
Let $\Aa,\Bb \in \R$ and $t,y > 0$. Then,
\[
\E\Bigl[\exp\Bigl((\Aa -  \Bb)h_{\Hh}^{\Aa,\Bb}(t,y)\Bigr)\Bigr] = \exp(R_\Hh^{\Aa,\Bb}(t,y)).
\]
\end{lemma}
\begin{proof}
    By Lemma \ref{lem:Hamquad_process}, we have 
    \[
    h_{\Hh}^{\Aa,\Bb}(t,y) = \nu^{\Aa,\Bb}(y) + \eta_y^{\Aa,\Bb}(t) = \nu^\Bb(y) + \eta_y^{\Aa,\Bb}(t), 
    \]
    where the second equality follows because $y > 0$. Writing $\eta_y^{\Aa,\Bb}(t) = h_{\Hh}^{\Aa,\Bb}(t,y) - \nu^\Bb(y)$, we see that the random variable $\eta_y^{\Aa,\Bb}(t)$ is a function of three mutually independent processes: $\{\nu^\Aa(x):x < 0\}$, $\{\nu^\Bb(x): x \in [0,y]\}$, and $\mbf X$. Lemma \ref{lem:InvPois} states that $\eta_y^\Aa(t)$ has the Poisson distribution with mean $e^{-\Aa}t$, whose moment generating function is readily computed. 
    We use the Radon-Nikodym derivative of Lemma \ref{lem:HamRND} to transform the law of the  process  $\{\nu^\Bb(x):x \in [0,y]\}$ to $\{\nu^\Aa(x):x \in [0,y]\}$ as follows:
    \begin{align*}
    &\quad \; \E\Bigl[\exp\Bigl((\Aa - \Bb)h_{\Hh}^{\Aa,\Bb}(t,y)\Bigr)\Bigr] 
    = \E\Bigl[\exp\Bigl((\Aa -  \Bb)(\nu^\Bb(y) + \eta_y^{\Aa,\Bb}(t)  \Bigr)\Bigr] \\
    &= \exp\Bigl(e^{\Aa} y - e^{\Bb} y\Bigr) \E\Bigl[\exp\Bigl((\Aa - \Bb) \eta_y^\Aa(t)\Bigr)\Bigr] 
    = \exp\Bigl((e^{\Aa}  - e^{\Bb}) y + t e^{-\Aa} (e^{\Aa - \Bb} - 1)\Bigr) =  \exp(R_\Hh^{\Aa,\Bb}(t,y)). \qedhere
    \end{align*}
\end{proof}
For $t,y > 0$, $\Aa \in \R$, we often use the following shorthand notation, noting that these quantities indeed depend on $t,y$: $R^{\Aa,\Bb} = R_\Hh^{\Aa,\Bb}(t,y), M^\Aa = M_\Hh^\Aa(t,y)$, $\gamma = \gamma_\Hh(t,y)$, and $\zeta = \zeta_\Hh(t,y)$. 

We now prove the following lemma, which is essentially Taylor expansion. However, we point out that this precise calculation involves the specific choice of $\zeta = \zeta_H(t,y)$ and is only valid for that choice. 
\begin{lemma} \label{lem:HamRupbd}
 For each $\ve \in (0,1)$, there exists a constant $C = C(\ve) > 0$ so that for all $t,y > 0$, $\ve < \sqrt{\f{t}{y}} < \ve^{-1}$ and $e^{\Aa}, e^\Bb  \in (\ve,\ve^{-1})$,
\[
\Bigl|R^{\Aa,\Bb} - \gamma(\Aa - \Bb) - \f{\gamma e^{2\zeta}}{6}((\Aa - \zeta)^3 - (\Bb - \zeta)^3)\Bigr| \le C(t+y)\Bigl((\Aa - \zeta)^4 + (\Bb - \zeta)^4\Bigr)
\]
\end{lemma}
\begin{proof}
Due to the definition $R^{\Aa,\Bb} = \int_\Bb^\Aa M^w\,dw$, it suffices to show that
\[
\Bigl|M^\Aa - \gamma -\f{\gamma e^{2 \zeta}}{2}(\Aa - \zeta)^2\Bigr| \le C(t+y)|\Aa - \zeta|^3. 
\]
First observe that 
\begin{align*} 
M^\Aa - \gamma &= (e^\Aa - e^\zeta)y + (e^{-\Aa} - e^{-\zeta})t = (e^\Aa - e^\zeta)\Bigl(y - \f{t}{e^\Aa e^\zeta}\Bigr) \\
&=(e^\Aa - e^\zeta) \Bigl(\f{ye^\Aa - t e^{-\zeta}}{e^\Aa}\Bigr) = (e^\Aa - e^\zeta) \f{y}{ e^\Aa}(e^\Aa - e^\zeta),
\end{align*}
where in the last line, we have used $t/y e^{-\zeta} = e^{\zeta}$. Then, 
\be \label{Mrhoexp}
    M^\Aa - \gamma - \f{\gamma}{2}(e^\Aa - e^\zeta)^2 = (e^\Aa - e^\zeta)^2\Bigl(\f{y}{e^\Aa} - \f{\gamma}{2}\Bigr) = (e^\Aa - e^\zeta)^2 \f{\sqrt{yt}}{e^\Aa} \Bigl(\sqrt{\f{y}{t}} - e^\Aa\Bigr) = -\f{\gamma}{2e^\Aa}(e^\Aa - e^\zeta)^3.
\ee
Observe that there exists a constant $c = c(\ve) > 0$ so that, for $t,y > 0$ satisfying $\ve^2 < \f{t}{y} < \ve^{-2}$,   
\be \label{Hamgambd}
c(t+y) \le \gamma(t,y) \le t + y.
\ee
 We also note that when $\ve < \f{t}{y} < \ve^{-1}$ and $\ve < e^{\Aa} < \ve^{-1}$, $\Aa$ and $\zeta$ are bounded. Hence, there exists a constant $C > 0$ changing from line to line so that
 \begin{align*}
    &\quad |M^\Aa - \gamma -\f{\gamma e^{2\zeta}}{2}(\Aa - \zeta)^2| \le \Bigl|M^\Aa - \gamma -\f{\gamma }{2}(e^\Aa - e^\zeta)^2\Bigr| + \Bigl|\f{\gamma }{2}(e^\Aa - e^\zeta)^2 - \f{\gamma e^{2\zeta}}{2}(\Aa - \zeta)^2\Bigr| \\
    &\overset{\eqref{Mrhoexp}}{\le} C(t+y)\Bigl(|e^\Aa - e^\zeta|^3 + |(e^\Aa - e^\zeta)^2 - e^{2\zeta}(\Aa - \zeta)^2|\Bigr) \le C(t+y) |\Aa - \zeta|^3,
\end{align*}
where the last inequality follows by a Taylor approximation of $e^x$ around $x = \zeta$ and the assumption that $\Aa$ and $\zeta$ are bounded. 
\end{proof}

\begin{lemma} \label{lem:Hamrlmont}
Let $\Aa' \le \Aa$ and $\Bb' \le \Bb$. Then, for all $t > 0$ and $y \in \R$, $Z_\Hh^{\Aa,\Bb}(t,y) \le Z_\Hh^{\Aa',\Bb'}(t,y)$. 
\end{lemma}
\begin{proof}
It follows from \eqref{nuinc} and the definition of the coupling of $\nu^{\Aa,\Bb}$ that
\[
\nu^{\Aa,\Bb}(y) -\nu^{\Aa,\Bb}(x) = \nu^{\Aa,\Bb}(x,y] \le \nu^{\Aa',\Bb'}(x,y] = \nu^{\Aa',\Bb'}(y) - \nu^{\Aa',\Bb'}(x).
\]
Let $z_1 = Z_\Hh^{\Aa,\Bb}(t,y)$ and $z_2 = Z_\Hh^{\Aa',\Bb'}(t,y)$.
Let $z_2 < x \le y$. By definition of $z_2$,
\begin{align*}
&\nu^{\Aa,\Bb}(x) + \Hh(x,0;y,t) - (\nu^{\Aa,\Bb}(z_2) + \Hh(z_2,0;y,t)) \\
&\le \nu^{\Aa',\Bb'}(x) + \Hh(x,0;y,t) - (\nu^{\Aa',\Bb'}(z_2) + \Hh(z_2,0;y,t))  < 0.
\end{align*}
Thus, we must have $z_1 \le z_2$. 
\end{proof}

\begin{lemma} \label{lem:hHameq}
Let $t > 0, y  > 0$, and $\Aa,\Bb,w \in \R$. If $Z_\Hh^{\Aa,w}(t,y) \wedge Z_\Hh^{\Aa,\Bb}(t,y) \le 0$, then $h_{\Hh}^{\Aa,w}(t,y) = h_{\Hh}^{\Aa,\Bb}(t,y)$. If $Z_\Hh^{w,\Bb}(t,y) \vee Z_\Hh^{\Aa,\Bb}(t,y) > 0$, then $h_{\Hh}^{w,\Bb}(t,y) = h_{\Hh}^{\Aa,\Bb}(t,y)$. 
\end{lemma}
\begin{proof}
    If $Z_\Hh^{\Aa,w}(t,y) \wedge Z_\Hh^{\Aa,\Bb}(t,y) \le 0$, then 
    \[
    h_{\Hh}^{\Aa,w}(t,y) = \sup_{-\infty < x \le 0}\{\nu^\Aa(x) + \Hh(x,0;y,t)\} = h_{\Hh}^{\Aa,\Bb}(t,y).
    \]
    A similar argument applies to the second statement. 
\end{proof}

\begin{lemma} \label{lem:Hamexitpt}
For each $\ve \in (0,1)$, there exists a constant $C = C(\ve) > 0$ so that, for $\zeta = \zeta(n,y) = \f{1}{2}\log\Bigl(\f{t}{y}\Bigr)$, whenever $\ve < e^\Aa < e^\zeta < \ve^{-1}$,
\be \label{HamZbd1}
\Pp(Z_\Hh^\Aa(t,y) > 0) \le \exp(-C(t+y)(\zeta - \Aa)^3),
\ee
and whenever $\ve < e^\zeta < e^\Aa < \ve^{-1}$,
\be \label{HamZbd2}
\Pp(Z_\Hh^\Aa(t,y) \le 0) \le \exp(-C(t+y)(\Aa - \zeta)^3).
\ee
\end{lemma}
\begin{proof}
    We prove \eqref{HamZbd1}, with \eqref{HamZbd2} following a symmetric argument. Let $\ve < e^{\Aa} < e^{\zeta} < \ve^{-1}$, and let $\Bb = (\zeta - \Aa)/4$ so that $\Aa = \zeta - 4\Bb$. Lemma \ref{lem:Hamrlmont} implies that, on the event $\{Z_\Hh^\Aa > 0\}$  (dropping the $(t,y)$ argument for simplicity), we also have 
\[
Z_\Hh^{\zeta,\zeta - 2\Bb} = Z_\Hh^{\Aa + 4\Bb,\Aa + 2\Bb} > 0, \qquad \text{and} \qquad Z_\Hh^{\zeta - 4\Bb,\zeta - 2\Bb} = Z_\Hh^{\Aa,\Aa + 2\Bb} > 0.
\]
Therefore, on this event,  Lemma \ref{lem:hHameq} implies that $h_{\Hh}^{\zeta,\zeta - 2\Bb} = h_{\Hh}^{\zeta - 4\Bb,\zeta - 2\Bb}$ (again dropping the $(t,y)$ dependence for ease of notation).
Thus, using H\"older's inequality and $R^{\Aa,\Bb} = -R^{\Bb,\Aa}$,
\begin{align*}
\Pp(Z_\Hh^\Aa > 0) &= \E[\exp(\Bb h_{\Hh}^{\zeta, \zeta - 2\Bb}- \Bb h_{\Hh}^{\zeta - 4\Bb,\zeta - 2\Bb}) \ind(Z_\Hh^\Aa > 0)]\\
&\le \E[\exp(\Bb h_{\Hh}^{\zeta,\zeta - 2\Bb} - \Bb h_{\Hh}^{\zeta - 4\Bb,\zeta - 2\Bb})] \\
&\le  \E[\exp(2\Bb h_{\Hh}^{\zeta,\zeta - 2\Bb} )]^{1/2} \E[\exp(-2\Bb h_{\Hh}^{\zeta - 4\Bb,\zeta - 2\Bb})]^{1/2} \\
&= \exp\Bigl(\f{1}{2}R^{\zeta,\zeta - 2\Bb} + \f{1}{2}R^{\zeta - 4\Bb,\zeta - 2\Bb}\Bigr) = \exp\Bigl(\f{1}{2}R^{\zeta,\zeta - 2\Bb} - \f{1}{2}R^{\zeta - 2\Bb,\zeta - 4\Bb}\Bigr).
\end{align*}
Using Lemma \ref{lem:PoiEJS} and \eqref{Hamgambd}, we have,  
\begin{align*}
&\quad\, R^{\zeta,\zeta - 2\Bb} - R^{\zeta - 2\Bb,\zeta - 4\Bb} \\ &\qquad\qquad \le \gamma(2\Bb) + \f{\gamma e^{2\zeta}}{6}(-(-2\Bb))^3  - \Bigl[\gamma(2\Bb) + \f{\gamma e^{2\zeta}}{6}((-2\Bb)^3 - (-4\Bb)^3)\Bigr] + C(t+y)\Bb^4 \le -C(t+y)\Bb^3,
\end{align*}
where the constant $C$ changes in each step.
\end{proof}

The following theorem is the culmination of the results of this section. It is analogous to \cite[Corollary 4.4]{Cator-Groeneboom-06}. The result is in a  slightly different setting because we deal with the stationary model in the upper-half plane instead of in the quadrant. The results of Lemma \ref{lem:Hamquad_process} are the starting point for connecting the  half-plane model to the quadrant model. In particular, the points of $\eta_x^b$ may be viewed as ``sinks" on the left boundary (see \cite{Cator-Groeneboom-2005}). We do not make use of that connection here, as we work directly in the full-space model. Theorem \ref{thm:Hamexitpt} allows us to perturb both the direction and the initial point on the order $N^{-1/3}$. The methods of \cite{Cator-Groeneboom-06} likely can be adapted for this purpose, but our result also has the advantage of the $e^{-CM^3}$ bound instead of the $M^{-3}$ bound present in the result of \cite{Cator-Groeneboom-06}.
\begin{theorem} \label{thm:Hamexitpt}
Let $\rho > 0, \mu \in \R$, and define $\chi,\beta,\tau$ by \eqref{Hamm_param} \rm{(}in particular, $\beta = \rho^{-1/2}$\rm{)}. As in \eqref{meanparam}, let $\beta_N(\mu) = \beta + \f{2\mu \chi}{\tau}N^{-1/3}$. Then, for each $y \in \R$ and $t > 0$, there exists a constant $C = C(y,\mu,\rho,t) > 0$ so that, for all sufficiently large positive $M$,
\[
\limsup_{N \to \infty} \Pp(|Z_\Hh^{\log(\beta_N(\mu))}(tN,t\rho N + y\tau N^{2/3})| > MN^{2/3}) \le e^{-CM^3}.
\]
\end{theorem}
\begin{remark}
    By our choice of parameterization, the $\log(\beta_N(\mu))$ in the superscript indicates that the initial data has intensity $\beta_N(\mu)$. 
\end{remark}
\begin{proof}
    We prove that 
    \[
    \limsup_{N \to \infty} \Pp(Z_\Hh^{\log(\beta_N(\mu))}(tN,t\rho N + y\tau N^{2/3}) > MN^{2/3}) \le e^{-CM^3},
    \]
    and the proof of the lower bound follows a symmetric argument. We first observe that, by shift invariance,
    \[
    \Pp(Z_\Hh^{\log(\beta_N(\mu))}(tN,t\rho N + y\tau N^{2/3}) > MN^{2/3}) = \Pp(Z_\Hh^{\log(\beta_N(\mu))}(tN,t\rho N + y\tau N^{2/3} - MN^{2/3}) > 0).
    \]
   We observe that 
    \begin{align*}
    &\quad \; \exp\Bigl(\zeta(tN,t\rho N + y\tau N^{2/3} - MN^{2/3})\Bigr) - \beta = \sqrt{\f{tN}{t\rho N + y\tau N^{2/3} - MN^{2/3}}} - \beta \\
    &= \f{1}{\sqrt{\rho + (y\tau - M)N^{-1/3}/t}} - \f{1}{\sqrt \rho} = \f{(M - y\tau)N^{-1/3}/t}{\sqrt \rho \sqrt{\rho + (y\tau-M)N^{-1/3}/t} (\sqrt \rho +\sqrt{\rho + (y\tau-M)N^{-1/3}/t} )} \\
    &\ge \f{(M - y\tau)N^{-1/3}/t}{\sqrt \rho \sqrt{\rho + y\tau N^{-1/3}/t} (\sqrt \rho +\sqrt{\rho + y\tau N^{-1/3}/t} )} \ge C(M - y\tau)N^{-1/3},
    \end{align*}
    where $C = C(y,\mu,\rho,t) > 0$ is a constant independent of $M$.
    Since $\beta_N(\mu) - \beta = \f{2\mu \chi}{\tau}N^{-1/3}$, we see that we may choose $M$ sufficiently large, independent of $N$ so that 
     $\exp(\zeta(tN,t\rho N + y\tau N^{2/3} - MN^{2/3}))  > \beta_N(\mu)$. Fix such a large $M$, and let $\ve  \in (0,1)$. Let $N$ be large enough (depending on $M,y,\rho,\mu$) so that 
     \[
    \ve < \beta_N(\mu) <  \exp(\zeta(tN,t\rho N + y\tau N^{2/3} - MN^{2/3})) < \ve^{-1}.
     \]
     A Taylor approximation shows that there exist constants $c_1 = c_1(y,\mu,\rho,t), C_1 = C_1(y,\mu,\rho,t)$ and $C_2(M) = C_2(M,y,\mu,\rho,t)$ so that
     \[
     \zeta(tN,t\rho N + y\tau N^{2/3}) - \log(\beta_N(\mu)) \ge C_1 (M - y\tau - c_1)N^{-1/3}  - C_2(M)N^{-2/3}.
     \]
     We emphasize here that $C_2$ may depend on $M$, but $c_1$ and $C_1$ do not. Thus, when $M \ge (y\tau - c_1)$, and changing $C_1$ to $C_1/2$, we have
     \[
     \zeta(tN,t\rho N + y\tau N^{2/3}) - \log(\beta_N(\mu)) \ge C_1 M N^{-1/3} - C_2(M) N^{-2/3}.
     \]
     Thus, by Lemma \ref{lem:Hamexitpt}, $M$ sufficiently large and for a constant $C = C(y,\mu,\rho,t) > 0$,
     \begin{align*}
     &\quad \; \limsup_{N \to \infty} \Pp(Z_\Hh^{\log(\beta_N(\mu))}(tN,t\rho N + y\tau N^{2/3} - MN^{2/3}) > 0) \\
     & \le \limsup_{N \to \infty} \exp\bigl(-C(tN + \rho t N 
 + y\tau N^{2/3} - MN^{2/3})(C_1 M N^{-1/3} - C_2(M) N^{-2/3})^3\Bigr) \\
 &= \exp(-C C_1 (t + \rho t)M^3). \qedhere  
     \end{align*}
\end{proof}

\subsection{Poisson lines model}
\subsubsection{Queues}
The stationary initial data for the Poisson lines model is tied to the classical version of Burke's theorem, originally proved in \cite{Burke1956}. We formulate the queuing setup as follows. Let $A$ and $S$ be locally finite point processes. For $s < t$, we will say that $A(s,t]$ corresponds to the number of arrivals in the time interval $(s,t]$ and that $S(s,t]$ corresponds to the number of services available in the same time interval. The queue follows the first-in-first out principle. We define the following queuing operators:
\be \label{eqn:PoiLqueue}
\begin{aligned}
Q(A,S)(t) &= \sup_{-\infty < s \le t}\{A(s,t] - S(s,t]\}, \\
D(A,S)(s,t] &= A(s,t]  + Q(A,S)(s) - Q(A,S)(t), \\
R(A,S)(s,t] &= S(s,t] + Q(A,S)(t) - Q(A,S)(s).
\end{aligned}
\ee
Some authors define $Q(A,S) =\sup_{-\infty < s \le t}\{A(s,t] - S(s,t]\} \vee 0$, but we remove the $+$ distinction simply by the convention $A(t,t] = S(t,t] = 0$. Here, $Q(A,S)(t)$ denotes the length of the queue at time $t$, and $D(A,S)(s,t]$ denotes the number of departures from the queue in the interval $(s,t]$. $R(A,S)(s,t]$ is the number of arrivals plus the number of unused services in the interval $(s,t]$. In this setting, the queue has been running since time $-\infty$, so we require that $\limsup_{s \to -\infty} A(s,0] - S(s,0] = -\infty$ to make this well-defined. To see why $t \mapsto Q(t) := Q(A,S)(t)$ describes the evolving length of the queue, observe that, for any $\ve > 0$,
\[
Q(A,S)(t) = (Q(t-\ve) + A(t-\ve,t] - S(t-\ve,t]) \vee \sup_{t - \ve < s \le t}\{A(s,t] - S(s,t]\}.
\]
Since $A$ and $S$ are point processes, we may choose a (random) $\ve > 0$ small enough so that $A(t-\ve,t] + S(t-\ve,t] = 1$. Case by case, one sees that the queue length increases by one when $A(t-\ve,t] = 1$, and if $Q(t-\ve) \ge 1$, the queue length decreases when $S(t-\ve,t] = 1$. If $Q(t-\ve) = 0$ and $S(t-\ve,t] = 1$, then the queue length remains at $0$ at time $t$. In particular, if $A$ and $S$ are independent Poisson point processes with rates $0 < \lambda < \mu < \infty$, respectively, then the evolution of the queue is a Markov process that transitions from $j$ to $j+1$ at rate $\lambda$ and for $j \ge 1$, it transitions from state $j$ to $j - 1$ at rate $\mu$. The following is a strengthening of Burke's theorem \cite{Burke1956}. Except for the independence in \eqref{PoiDRQ_indep}, the theorem  comes from \cite[Theorem 3]{rep_non_colliding} and \cite[page 11]{Kelly-2011}.
\begin{theorem} \label{thm:Burke_Poi}
    Let $0 < \lambda < \mu < \infty$, and let $A$ and $S$ be independent Poisson point processes on $R$ of intensity $\lambda$ and $\mu$, respectively. Then, $D(A,S)$ and $R(A,S)$ are independent Poisson point processes of intensity $\lambda$ and $\mu$, respectively. For each $t \in \R$, $Q(A,S)(t) \sim {\rm Geom}\bigl(1-\f{\lambda}{\mu}\bigr)$, and the processes
    \be \label{PoiDRQ_indep}
    \{D(A,S)(s,t],R(A,S)(s,t]:-\infty < s \le t\},\quad\text{and}\quad \{Q(A,S)(s):s \ge t\}
    \ee
    are independent. 
\end{theorem}
\begin{proof}
The fact that $D(A,S)$ and $R(A,S)$ are independent Poisson processes with the given intensities is exactly \cite[Theorem 3]{rep_non_colliding}. The distribution of $Q(A,S)(t)$ is \cite[page 11]{Kelly-2011}. We prove the independence of the processes in \eqref{PoiDRQ_indep}. The independence of $D(A,S)(s,t]$ and $Q(t)$ was originally established by Burke \cite{Burke1956}. The extension in the present theorem follows by writing $Q(t)$ for $t \in \R$ as an explicit function of $\{D(A,S)(t,u],R(A,S)(t,u]:u \ge t\}$. 

We represent $A$ and $S$ as functions $\R \to \R$ by setting $A(0) = S(0) = 0$, and defining $A(t) - A(s) = A(s,t]$ and $S(t) - S(s) = S(s,t]$ for $s < t$. Similarly, we can define $D(A,S)$ and $R(A,S)$ as functions. In this sense, we have 
\begin{align*}
D(A,S)(t)  &= A(t) + \sup_{-\infty < s \le 0}\{- A(s) + S(s)\} - \sup_{-\infty < s \le t}\{A(t) - A(s) + S(s) - S(t)\} \\
&= S(t) + \sup_{-\infty < s \le 0}\{S(s) - A(s)\} - \sup_{-\infty < s \le t}\{S(s) - A(s)\}.
\end{align*}
Similarly,
\[
R(A,S)(t) = A(t) + \sup_{-\infty < s \le t}\{S(s) - A(s)\} - \sup_{-\infty < s \le 0}\{S(s)-A(s)\}.
\]
This allows us to extend the operators $D$ and $R$ to pairs of functions $S,A:\R \to \R$ satisfying $S(0) = A(0) = 0$. We similarly extend $Q$ by
\[
Q(A,S)(t) = \sup_{-\infty < s \le t}\{A(t) - A(s) - S(t) + S(s)\}.
\]
For a function $f:\R \to \R$, define the reflected function $\Rf f(x) = f(-x)$. We first show how the desired result follows from 
\be \label{PoiLrfD}
D(\Rf R(A,S),\Rf D(A,S)) = \Rf S.
\ee
The left-hand side of \eqref{PoiLrfD} is
\begin{align*}
&\quad \;\Rf R(A,S)(t) + Q(\Rf R(A,S),\Rf D(A,S))(0) - Q(\Rf R(A,S),\Rf D(A,S))(t) \\
&=  \Rf S(t) + \Rf Q(A,S)(t) - \Rf Q(A,S)(0) +  Q(\Rf R(A,S),\Rf D(A,S))(0) - Q(\Rf R(A,S),\Rf D(A,S))(t). 
\end{align*}
Hence, rearranging \eqref{PoiLrfD} yields, for all $t \in \R$,
\[
\Rf Q(A,S)(t) - Q(\Rf R(A,S),\Rf D(A,S))(t) = \Rf Q(A,S)(0) - Q(\Rf R(A,S),\Rf D(A,S))(0). 
\]
In other words, the function $t \mapsto \Rf Q(A,S)(t) - Q(\Rf R(A,S),\Rf D(A,S))(t)$ is constant. Both $\Rf Q(A,S)(t)$ and $Q(\Rf R(A,S),\Rf D(A,S))(t)$ described queue length from equilibrium initial data and so, almost surely, there exist times when the queue is $0$, i.e., they both achieve the value $0$ for some value $t$. Hence, $\Rf Q(A,S)(t) = Q(\Rf R(A,S),\Rf D(A,S))(t)$ for all $t \in \R$. Hence, by definition,
\begin{align*}
Q(A,S)(t) &= \Rf Q(\Rf R(A,S),\Rf D(A,S))(t) \\
&= \sup_{-\infty < s \le -t}\{\Rf R(A,S)(-t) - \Rf R(A,S)(s) - \Rf D(A,S)(-t) + \Rf D(A,S)(s)  \} \\
&= \sup_{-\infty < s \le -t}\{R(A,S)(t) - R(A,S)(-s) - D(A,S)(t) +  D(A,S)(-s)\} \\
&= \sup_{t \le s < \infty}\{ D(A,S)(t,s] - R(A,S)(t,s]\},
\end{align*}
and this completes the proof, conditional on \eqref{PoiLrfD}.

We turn to proving \eqref{PoiLrfD}. This was proven in \cite[Lemma D.2]{Seppalainen-Sorensen-21a} and \cite[Lemma 2.3.8]{Sorensen-thesis} in the case where the functions $S$ and $A$ are continuous. This relation has deep connections to the RSK correspondence. Connections between the RSK correspondence and these Pitman-like transforms were first made in \cite{Noumi-Yamada-2004,Biane-Bougerol-OConnell-2005} and have been studied further in \cite{Dauvegne-Nica-Virag-2021,Bates-Fan-Seppalainen,Directed_Landscape,Dauvergne-Virag-2024}. We show how to modify the argument given in \cite{Seppalainen-Sorensen-21a,Sorensen-thesis} for our purposes, where $S$ and $A$ are not continuous. We mention that in these works, the roles of $D$ and $R$ are flipped, as well as the ordering of the arguments. We note here that all operations in \eqref{PoiLrfD} are well-defined almost surely: Since $A$ and $S$ are independent Poisson processes with intensities $\lambda < \mu$, $\lim_{s \to -\infty}[S(s) - A(s)] = -\infty$, and $\sup_{-\infty < s \le t}[S(s) - A(s)] < \infty$ for all $t$. Since we know the distributions of $R(A,S)$ and $D(A,S)$, the same reasoning implies that $D(\Rf R(A,S),\Rf D(A,S))$ is well-defined. 

By unpacking the definitions just as in \cite[Lemma 2.3.8]{Sorensen-thesis}, to prove \eqref{PoiLrfD}, it suffices to show that, for all $t \in \R$, setting $f(s) = S(s) - A(s)$ and $F(t) = \sup_{-\infty < s \le t} f(s)$,
\be \label{PoiPitman}
F(t) = \inf_{t \le s < \infty}\{2 F(s) - f(s)\}.
\ee
For continuous functions $f$ satisfying $\lim_{s \to \pm \infty} f(s) = \pm \infty$, this is Pitman's $2M - X$ theorem \cite[Equation (1.4)]{Pitman1975} (see also \cite[Equation (13)]{brownian_queues} and \cite[Lemma A.14]{Sorensen-thesis}). We show that \eqref{PoiPitman} holds almost surely for our choice of $f = S - A$, simultaneously for all choices of $t$. The full probability event is the event where the point processes $A$ and $S$ do not share any points and where $\lim_{s \to \pm \infty} S(s) - A(s) = \pm \infty$. 

First, observe that, for all $s \ge t$, $F(s) \ge F(t) \wedge f(s)$, so
\[
2 F(s) - f(s) \ge F(t) + f(s) - f(s) = F(t).
\]
We seek to find a value of $s \ge t$ for which equality holds. Observe that $f$ is integer-valued, so $F$ is as well. Furthermore, since $S$ and $A$ do not share any common points, $f$ only changes via jumps of size $\pm 1$. We also note that $f$ is right-continuous. Observe that $f(t) \le F(t)$ and that $\lim_{s \to \infty} f(s) = +\infty$. In this case, a version of the intermediate value theorem holds: there exists $s \ge t$ so that $f(s) = F(t)$ for some $s \ge t$ because $F(t)$ is an integer, and $f$ only changes in jumps of size $\pm 1$. Set $s^\star = \inf\{s \ge t: f(s) = F(t)\}$. Because $f$ is right-continuous, $f(s^\star) = F(t) = F(s^\star)$. Hence,
$
2F(s^\star) - f(s^\star) = 2F(t) - F(t) = F(t)
$,
as desired. 
\end{proof}

We recall the notation that for initial data $\nu \in \mathcal N$, associated to time level $-1$, we set 
\[
h_\Uu(n,y;\nu) = \sup_{-\infty < x \le y}\{\nu(x) + \Uu(x,0;y,n)\}.
\] For $\nu \in \mathcal N$, define $\overline h_\Uu(n,\abullet;\nu)$ to be the function $y \mapsto h_\Uu(n,y;\nu) - h_\Uu(n,0;\nu)$. The following shows that the Poisson lines model with initial data evolves via the queuing mappings.
\begin{lemma} \label{lem:PoiLevolve}
 Let $\nu \in \mathcal N$ satisfy
 \[
 \liminf_{x \to -\infty} \f{\nu(x)}{x} =:c > 1. 
 \]
 Then, for each $n \ge 0$,
 \be \label{eqn:PoiLliminf}
 \liminf_{y \to -\infty} \f{h_\Uu(n,y;\nu)}{y} \ge c,
 \ee
 and, for $y \in \R$ the following equalities hold, each side being well-defined and finite. 
 \be \label{35}
 h_\Uu(n,y;\nu) - h_\Uu(n-1,y;\nu) = Q(F_n,\overline h_\Uu(n-1,\abullet;\nu))(y).
 \ee
 while
 \be \label{36}
 \overline h_\Uu(n,y;\nu) = R(F_n,\overline h_\Uu(n-1,\abullet;\nu))(y).
 \ee
 In the case $n = 0$, we define $h_\Uu(-1,\abullet;\nu) = \nu$.
 In particular, if, for $\Aa > 0$, $\nu = \nu^\Aa$, independent of $\{F_i\}_{i \ge 0}$, then $h_\Uu(n,y;\nu^\Aa) - \nu^\Aa(y)$ is distributed as the sum of $n+1$ i.i.d. ${\rm Geom}(1-e^{-\Aa})$ random variables.
\end{lemma}
\begin{proof}
    We apply the dynamic programming principle: for $n \ge 0$,
    \begin{align*}
    h_\Uu(n,y;\nu) &= \sup_{-\infty < x \le y}\{\nu(x) + \Uu(x,0;y,n)\} \\
    &= \sup_{-\infty < x \le z \le y}\{\nu(x) + \Uu(x,0;z,n-1) + F_n(z,y] \} \\
    &= \sup_{-\infty < z \le y}\{h_\Uu(n-1,z;\nu) + F_n(z,y]\} = \sup_{-\infty < z \le y}\{h_\Uu(n-1,z;\nu) + F_n(y) - F_n(z)\}.
    \end{align*}
    Since $F_n$ is a Poisson point process of intensity $1$, we know that $\lim_{z \to -\infty} \f{F_n(z)}{z} = 1$. Hence, to show \eqref{eqn:PoiLliminf}, it suffices to show that 
    \[
    \liminf_{y \to -\infty}\f{1}{y} \sup_{-\infty < z \le y}\{h_\Uu(n-1,z;\nu) - F_n(z)\} \ge c - 1.
    \]
    Assume that \eqref{eqn:PoiLliminf} holds for $n - 1$. Let $\ve  \in (0,c-1)$, and let $K$ be large so that for all $z < -K$, $h_\Uu(n-1,z;\nu) -F_n(z) \le (c - 1 - \ve)z$. Then, for $y < -K$,
    \[
    \sup_{-\infty < z \le y}\{h_\Uu(n-1,z;\nu) - F_n(z)\} \le \sup_{-\infty < z \le y}\{(c-1-\ve)z\} = (c-1 - \ve)y.
    \]
    Dividing by $y$, sending $y \to -\infty$, then $\ve \searrow 0$ completes the proof of \eqref{eqn:PoiLliminf} by induction. We see further that this makes all the queuing operations well-defined because $h_\Uu(n-1;z,\nu) - F_n(z) \to -\infty$ as $z \to -\infty$. We see now that $h_\Uu(n,y;\nu) - h_\Uu(n-1;y,\nu)$ is equal to
    \[
  \sup_{-\infty < z \le y}\{h_\Uu(n-1,z;\nu) - h_\Uu(n-1,y;\nu) + F_n(y) - F_n(z)\} = Q(F_n,\overline h_\Uu(n-1,\abullet;f))(y).
    \]
    Then,
    \begin{align*}
    &\quad \; R(F_n,\overline h_\Uu(n-1,\abullet;\nu))(y) = \overline h_\Uu(n-1,y;\nu) + Q(F_n,\overline h_\Uu(n-1,\abullet;f))(y) -  Q(F_n,\overline h_\Uu(n-1,\abullet;f))(0)\\
    &= h_\Uu(n-1,y;\nu) - h_\Uu(n-1,0;\nu) + h_\Uu(n,y;\nu) - h_\Uu(n-1;y,\nu) - h_\Uu(n;0,\nu) + h_\Uu(n-1;0,\nu)\\
    &= \overline h_\Uu(n;y,f).
    \end{align*}
    We finish by proving that $h_\Uu(n,y;\nu^\Aa) - \nu^\Aa(y)$ is distributed as the sum of $n + 1$ i.i.d. Geometric random variables. We do this by induction on $n$. Recall that $\nu^\Aa$  is a Poisson point process of intensity $e^\Aa$. Then, Equations \eqref{35}--\eqref{36} and Theorem \ref{thm:Burke_Poi} imply that $\overline h_\Uu(0,\abullet;\nu^\Aa) = D(F_0,\nu^\Aa)\sim \nu^\Aa$, $h_\Uu(0,y;\nu^\Aa) - \nu^\Aa(y) = Q(F_0,\nu^\Aa)(y) \sim {\rm Geom}(1-e^{-\Aa})$, and the processes
    \[
    \{h_\Uu(0,y;\nu^\Aa) - h_\Uu(0,x;\nu^\Aa): x \le y\} \quad\text{and}\quad h_\Uu(0,y;\nu^\Aa) - \nu^\Aa(y)
    \]
    are independent. By way of induction,  assume for some $n \ge 0$ that $\{h_\Uu(j,y;\nu^\Aa) - h_\Uu(j-1,y;\nu^\Aa)\}_{0 \le j \le n}$ is a collection of i.i.d. ${\rm Geom}(1 - e^{-\Aa})$ random variables, independent of the process
    \be \label{hd123}
    \{h_\Uu(n,y;\nu^\Aa) - h_\Uu(n,x;\nu^\Aa): x \le y\},
    \ee
    and $\overline h_\Uu(n,\abullet;\nu^\Aa) \sim \nu^{\Aa}$. By construction, the Poisson point process $F_{n+1}$ is independent of this collection. Then, using \eqref{35},\eqref{36} and Theorem \ref{thm:Burke_Poi},
    \begin{align*}
    &\overline h_\Uu(n+1,\abullet;\nu^\Aa) = R(F_{n+1},\overline h_\Uu(n,\abullet;\nu^\Aa)) \sim \nu^\Aa, \\
    &h_\Uu(n+1,y;\nu^\Aa) - h_\Uu(n,y;\nu^\Aa) =   Q(F_{n+1},\overline h_\Uu(n,\abullet;\nu^\Aa))(y) \sim {\rm Geom}(1 - e^{-\Aa}).
    \end{align*}
    Furthermore, by definition of the mappings $R$ and $Q$, we see that 
    \[
    h_\Uu(n+1,y;\nu^\Aa) - h_\Uu(n,y;\nu^\Aa),\quad \text{and}\quad \{h_\Uu(n+1,y;\nu^\Aa) - h_\Uu(n+1,x;\nu^\Aa): x \le y\}
    \]
    are functions of $F_{n+1}$ and the process in \eqref{hd123}. This completes the induction.
\end{proof}

\subsubsection{Coupling and the EJS-Rains identity for the Poisson lines model}
Define the coupling $\{\nu^{\Aa,\Bb}:\Aa,\Bb \in (0,\infty) \}$ of processes just as in Section \ref{sec:Poiscoup}. Take this coupling to be independent of $\{F_i\}_{i \in \Z}$, under probability measure $\Pp$. Here, we note that we restrict to positive $\Aa,\Bb$ to make the following well-defined via Lemma \ref{lem:PoiLevolve}. 
\begin{align*}
h_\Uu^{\Aa,\Bb}(n,y) &= h_\Uu(n,y;\nu^{\Aa,\Bb}), \quad\text{and}\quad 
Z_\Uu^{\Aa,\Bb}(n,y) = \sup \argmax_{-\infty < x \le y} \{\nu^{\Aa,\Bb}(x) + \Uu(x,0;y,n)\}.
\end{align*}
We use a single superscript when $\Aa = \Bb$. Lemma \ref{lem:PoiLevolve} implies that, for $\Aa > 0$,
\[
\E[h_\Uu^{\Aa}(n,y)] = \E[\nu^\Aa(y) + (h_\Uu^{\Aa}(n,y) - \nu^\Aa(y))] = e^\Aa y + \f{n+1}{e^\Aa - 1}.
\]
We now make the following definitions
\begin{align*}
    M_\Uu^{\Aa}(n,y) &:= e^\Aa y + \f{n+1}{e^\Aa - 1},\quad R_\Uu^{\Aa,\Bb}(n,y) := \int_\Bb^\Aa M_\Uu^{w}(n,y)\,dw = (e^\Aa - e^\Bb)y + (n+1)\log\f{1 - e^{-\Aa}}{1 - e^{-\Bb}} , \\
    \zeta_\Uu(n,y) &:= \arg \inf_{\Aa \in \R} M_\Uu^\Aa(n,y) = \log\Bigl(1 + \sqrt{\f{n+1}{y}}\Bigr),\quad \gamma_\Uu(n,y) := y + 2\sqrt{(n+1)y}.
\end{align*}

The following follows a nearly identical proof as in Lemma \ref{lem:PoiEJS}. The only difference is that $h_\Uu^{\Aa}(n,y) - \nu^\Aa(y)$ is distributed as the sum of $n+1$ i.i.d ${\rm Geom}(1-e^{-\Aa})$ random variables instead of a Poisson random variable. We observe here that if $\Bb \le 0$, $h_\Uu^{\Aa,\Bb}(n,y)$ is still finite by Lemma \ref{lem:PoiLevolve}, but $\Bb >0$ is essential for the moment generating function in the following lemma.
\begin{lemma} \label{lem:LinesEJS}
Let $\Aa,\Bb \in (0,\infty)$. Then, 
\[
\E\Bigl[\exp\Bigl((\Aa-\Bb)h_\Uu^{\Aa,\Bb}(n,y)\Bigr)\Bigr] = \exp(R_\Uu^{\Aa,\Bb}(n,y)).
\]
\end{lemma}

We now prove the Poisson lines analogue of Lemma \ref{lem:HamRupbd}. The shorthands $R^{\Aa,\Bb},M^{\Aa},\gamma$, and $\zeta$ are used without the arguments $n,y$. 
\begin{lemma} \label{lem:PoiL_Rab}
For $\ve \in (0,1)$, there exists a constant $C = C(\ve) > 0$ so that for $n,y > 0$, $\ve < \f{n+1}{y} < \ve^{-1}$, and $e^\Aa - 1,e^\Bb - 1 \in (\ve,\ve^{-1})$,
\[
\Bigl|R^{\Aa,\Bb} - \gamma(\Aa - \Bb) - \f{\gamma e^{2\zeta}}{3(e^\zeta - 1) + 6(e^\zeta - 1)^2}((\Aa - \zeta)^3 - (\Bb - \zeta)^3) \Bigr| \le C(n + y)\Bigr((\Aa - \zeta)^4 + (\Bb - \zeta)^4\Bigr).
\]
\end{lemma}
\begin{proof}
First, observe that $e^\zeta - 1 = \sqrt{\f{n+1}{y}}$. Then,
\begin{align*}
    &\quad \; M^\Aa - \gamma = e^\Aa y + \f{n+1}{e^\Aa - 1} - e^\zeta - \f{n+1}{e^\zeta - 1} 
    =(e^\Aa - e^\zeta)y + (n+1)\Bigl(\f{1}{e^\Aa - 1} - \f{1}{e^\zeta - 1}\Bigr) \\
    &= (e^\Aa - e^\zeta)y\Bigl(1 - \f{\sqrt{y/(n+1)}}{y(e^\Aa - 1)}\Bigr) = (e^\Aa - e^\zeta)y\Bigl(\f{e^\Aa - 1 - \sqrt{(n+1)/y}}{e^\Aa - 1}\Bigr) = (e^\Aa - e^\zeta)^2 \f{y}{e^\Aa - 1}.
\end{align*}
Next, let $w = (e^\zeta - 1) + 2(e^\zeta - 1)^2 = \sqrt{\f{n+1}{y}} + \f{2(n+1)}{y}$, and observe that 
\[
wy + \gamma = \sqrt{(n+1)y} + 2(n+1) + \gamma =  (2 \sqrt{(n+1)y} + y)\Bigl(\sqrt{\f{n+1}{y}} + 1\Bigr)  = \gamma e^\zeta.
\]
Therefore,
\begin{align*}
    M^\Aa - \gamma - \f{\gamma}{w}(e^\Aa - e^\zeta)^2 = (e^\Aa - e^\zeta)^2\Bigl(\f{y}{e^\Aa - 1} - \f{\gamma}{w}\Bigr) = -\f{\gamma}{w(e^\Aa - 1)}(e^\Aa - e^\zeta)^3. 
\end{align*}
The proof now follows just as the proof of Lemma \ref{lem:HamRupbd}, noting that there exist constants $c,C > 0$ so that 
\[
c (n + y) \le \gamma_{\Uu}(n,y) \le C(n+y)
\]
the assumptions imply that $e^\Aa - 1$, $e^\Bb - 1$, and $w$ are all bounded away from $0$ and $\infty$.
\end{proof}

\begin{theorem} \label{thm:PoiL_exit}
Let $\rho > 0$, $\mu \in \R$, and  define $\chi,\beta,\tau$ by \eqref{PoiLines_param}. As in \eqref{meanparam}, let $\beta_N(\mu) = \beta + \f{2\mu \chi}{\tau}N^{-1/3}$. Then, for each $y \in \R$ and $t > 0$, there exists a constant $C= C(y,\mu,\rho,t) > 0$ so that, for all sufficiently large positive $M$,
\[
\limsup_{N \to \infty} \Pp(|Z_\Uu^{\log(\beta_N(\mu))}(\lfloor tN \rfloor,\rho N + y\tau N^{2/3})| > MN^{2/3}) \le e^{-CM^3}.
\]
\end{theorem}
\begin{proof}
The proof now follows the same as the proof of Theorem \ref{thm:Hamexitpt}. The necessary adjustments specific to the Poisson lines model are given in Lemmas \ref{lem:LinesEJS} and \ref{lem:PoiL_Rab}.
\end{proof}

\subsection{Sepp\"al\"ainen-Johansson model}
We recall the notation from Section \ref{sec:SJ}: Considering initial data $f:\Z \to \R$ on level $-1$ whose linear interpolation lies in $\UC$, define, for $n \ge 0$ and $m \in \Z$,
\begin{align*}
h_\Tt(n,m;f) &= \sup_{-\infty < k \le m}\{f(x) + \Tt(k,0;m,n)\}, \quad
Z_\Tt(n,m;f) = \max \argmax_{-\infty < k \le m}\{f(x) + \Tt(k,0;m,n)\}.   
\end{align*}
\subsubsection{Queues and couplings}
We first give the following condition on $f$ for $h_\Tt(n,m;f)$ to be well-defined.
\begin{lemma} \label{SJhwd}
With probability one, whenever
\[
\liminf_{k \to -\infty} \f{f(k)}{k} = c > p,
\]
then $h_\Tt(n,m;f) < \infty$ for all $n \ge 0$ and $m \in \Z$. Furthermore, for each $n \ge 0$,
\[
\liminf_{m \to -\infty} \f{h_\Tt(n,m;f)}{m} \ge c.
\]
\end{lemma}
\begin{proof}
$d$ satisfies the dynamic programming principle: for $n > 0$, 
\[
\Tt(k,0;m,n) = \sup_{k \le j \le m}\{\Tt(k,0;j,n-1) + \Tt(j,n;m,n)\}
\]
For $k \in \Z$, let 
\[
t^n(k)  = \begin{cases}
\sum_{i = 1}^k t_{(i-1,n),(i,n)} &k > 0 \\
-\sum_{k = k+1}^0 t_{(i-1,n),(i,n)} &k \le 0
\end{cases}
\]
Thus,
\begin{align*}
    h_\Tt(n,m;f) &= \sup_{-\infty < k \le j \le m}[f(k) + \Tt(k,0;j,n-1) + \Tt(j,n;m,n)] \\
    &= \sup_{-\infty < j \le m}[h_\Tt(n-1,j;f) + t^n(m) - t^n(j)].
    \end{align*}
If we remove the middle inequality above, the statement for $n = -1$ reverts to the definition of $h_\Tt$ by defining $h_\Tt(n-1,j;f) = f$. Since the horizontal edges $t$ are i.i.d. ${\rm Ber}(p)$, we have $\lim_{m \to -\infty} \f{t^n(m)}{m} = p$, almost surely. All parts of the lemma will thus be satisfied if we inductively show that 
\be \label{SJhdlim}
\liminf_{m \to -\infty} \f{1}{m}\sup_{-\infty < j \le m}[h_\Tt(n-1,j;f) - t^n(j)] \ge c - p.
\ee
Assume, by induction, that the statement holds for $n - 1 \ge -1$ (the base case $n - 1 = -1$ is the assumption on $f$). Let $\ve \in (0,c-p)$ Then, there exists $M > 0$ so that when $j < -M$, $h_\Tt(n-1,j;f) - t^n(j) \le (c - p - \ve)j$. Then, when $m < -M$
\[
\sup_{-\infty < j \le m}[h_\Tt(n-1,j;f) - t^n(j)] \le \sup_{-\infty < j \le m}[(c-p - \ve)j] = (c-p-\ve)m.
\]
Dividing by $m$, sending $m \to -\infty$, then sending $\ve \searrow 0$ completes the proof. 
\end{proof}

In this section, we show that the marginally invariant measures for the SJ model are i.i.d. Bernoulli with intensity strictly greater than $p$. For this, we borrow a Burke property proved in \cite{Busa-Sepp-Sore-22b} for discrete-time queues. Let $\Qs_1:=\{0,1\}^\Z$ be the space of configurations of particles on $\Z$ with the following interpretation:  a configuration  $\bq{x}=\{x_j\}_{j\in\Z}\in \Qs_1$  has a particle at time $j\in\Z$ if $x_j=1$, otherwise,   $\bq{x}$ has a hole at time $j\in\Z$.  
	Let $\arrv,\srvv\in\Qs_1$. In the queuing context, $\arrv$ represents arrivals of customers to a queue, and  $\srvv$ represents the available  services in  the queue. For $i\le j\in \Z$ let $a[i,j]$  be the number of customers that arrive in the interval $[i,j]$, that is, $a[i,j] = \sum_{k = i}^j a_k$. Similarly, let $s[i,j]$ be the number of services available during the time interval $[i,j]$. The queue length at time $i$ is then given by 
	\begin{equation}\label{Q}
		[Q(\arrv,\srvv)]_i=\sup_{j: j\leq i}\big(a[j,i]-s[j,i]\tspb\big)^+ = \sup_{j: j \le i + 1}\big(a[j,i]-s[j,i]\tspb\big),
	\end{equation}
 where $x[i+1,i]$ is defined to be $0$.
In our applications, $\arrv$ and $\srvv$ are always such that 
queue lengths are finite. The departures from the queue come from the mapping $\depv=D(\arrv,\srvv):\Qs_1\times\Qs_1 \rightarrow \Qs_1$,  given by 
 \[
 [D(\arrv,\srvv)]_i=[Q(\arrv,\srvv)]_{i-1}-[Q(\arrv,\srvv)]_i+a_i.
 \]
We also define the map $R:\Qs_1\times\Qs_1\rightarrow \Qs_1$ as 
\[
[R(\arrv,\srvv)]_i =[Q(\arrv,\srvv)]_i - [Q(\arrv,\srvv)]_{i-1} + s_i.
\]

\begin{lemma}[\cite{Hsu-Burke-1976}, see also Theorem 4.1 of \cite{koni-ocon-roch-02} and Lemma B.1 of \cite{Busa-Sepp-Sore-22b}] \label{lem:SJBurke}
Let $0 < p < u < 1$, and assume that $\arrv$ and $\srvv$ are independent i.i.d. Bernoulli sequences with success probability $p$ and $u$, respectively.  Let $\bq{q} = Q(\arrv,\srvv)$, $\depv=D(\arrv,\srvv)$, $\bq{r}=R(\arrv,\srvv)$. Then  for any $i_0\in\Z$, the random variables 
		\begin{equation}\label{eq11}
			\{d_j\}_{j\leq{i_0}}, \,\,\{r_j\}_{j\leq{i_0}} \,\, \text{ and } \,\, q_{i_0}
		\end{equation} 
		are mutually independent with marginal distributions $d_j\sim{\rm Ber}(p)$, $r_j\sim {\rm Ber}(u)$, and $q_{i_0}\sim{\rm Geom}(v)$ with $v=\frac{u-p}{(1-p)u}$. 
\end{lemma}

In light of Lemma \ref{lem:SJBurke}, the stationary initial conditions for this model are described by i.i.d. Bernoulli measures. This is made more precise in the sequel. First, we couple these measures as follows, noting that the coupling is not jointly stationary but will help us to derive exit point bounds from the stationary initial conditions. Let $\{U_i\}_{i \in \Z}$ be an i.i.d. sequence of uniform random variables on $[0,1]$. For $\Aa \in \R$, define the coupled sequences $\{I^\alpha_{(i,-1)}: \Aa \in \R, i \in \Z\}$ by
\[
I^\alpha_{(i,-1)}  = \begin{cases}
1 &U_i \le \f{e^\Aa}{1 + e^\Aa} \\
0 &\text{otherwise}.
\end{cases}
\]
Define the discrete function
The choice of parameter $\Aa$ is a convention that allows the EJS-Rains formula to work for this case. For $\Aa,\Bb \in \R$, define the discrete function
\[
f^{\Aa,\Bb}(k) = \begin{cases}
\sum_{i = 1}^k I^\gamma_{(i,-1)} &k \ge 1 \\
-\sum_{i = k+1}^0 I^\alpha_{(i,-1)} &k \le 0.
\end{cases}
\]
We use the convention that the empty sum is $0$ so that $f^{\Aa,\Bb}(0) = 0$. Recall the parameter $\lambda = \f{1-p}{p}$.
For $\Aa > -\log \lambda$ and $\Bb \in \R$, define
\[
h_\Tt^{\Aa,\Bb}(n,m) = h_\Tt(n,m;f^{\Aa,\Bb}),\quad \text{and}\quad Z_\Tt^{\Aa,\Bb}(n,m) = Z_\Tt(n,m;f^{\Aa,\Bb}).
\]
For $\Aa = \Bb$, we set $h_\Tt^\Aa(n,m) = h_\Tt^{\Aa,\Aa}(n,m)$ and $Z_\Tt^{\Aa}(n,m) = Z_\Tt^{\Aa,\Aa}$. 
For $n \in \Z$, we make use of the shorthand notation $t_{i}^n = t_{(i-1,n),(i,n)}$, so that $t^n = \{t^n_i\}_{i \in \Z}$ is an i.i.d. sequence of ${\rm Ber}(p)$ random variables.  For $\Aa > -\log\lambda$, $n \ge 0$ and $m \in \Z$, set
\[
I^{\Aa,\Bb}_{(m,n)} = h_\Tt^{\Aa,\Bb}(n,m) - h_\Tt^{\Aa,\Bb}(n,m-1),\qquad\text{and}\qquad J^{\Aa,\Bb}_{(m,n)} = h_\Tt^{\Aa,\Bb}(n,m) - h_\Tt^{\Aa,\Bb}(n-1,m),
\]
Note the flip in the coordinates in the definition. We think of $m$ as the horizontal component and $n$ as the vertical, but $n$ becomes time in the definition $h_\Tt$ and comes first.  Next, let $I^{\Aa,\Bb,n} = \{I^{\Aa,\Bb,n}_{m}\}_{m \in \Z} := \{I^{\Aa,\Bb}_{(m,n)}\}_{m \in \Z}$, and $J^{\Aa,\Bb,m} = \{J^{\Aa,\Bb,m}_n\}_{n \ge 0} := \{J^{\Aa,\Bb}_{(m,n)}\}_{n \ge 0}$.

Similarly as before, we use the conventions $f^{\Aa,\Aa} = f^{\Aa}$, $I^{\Aa,\Aa} = I^\Aa$, and $J^{\Aa,\Aa} = J^\Aa$.

We note that the condition $\Aa > -\log \lambda$ guarantees that these sequences are all finite almost surely because $\{I^\alpha_{(i,-1)}\}_{i \le 0}$ is an i.i.d. sequence with   mean $\f{e^{\Aa}}{1 + e^\Aa} > \f{\lambda^{-1}}{1 + \lambda^{-1}} = p$. Lemma \ref{SJhwd} implies the sequences are finite. 

\begin{lemma} \label{lem:SJQueue}
    Let $\Aa > -\log \lambda$ and $\Bb \in \R$. For $n \ge 0$ and $m \in \Z$, 
    \be \label{SJ-IJqueue}
    J^{\Aa,\Bb}_{(m,n)} = [Q(t^n,I^{\Aa,\Bb,n-1})]_m,\qquad\text{and}\qquad I^{\Aa,\Bb}_{(m,n)} = [R(t^n,I^{\Aa,\Bb,n-1})]_{m}.
    \ee
    Consequently, for each $n \ge -1$, $I^{\Aa,n}$ is an i.i.d sequence of ${\rm Ber}(u)$ random variables, while for each $m \in \Z$, $J^{u,m}$ is an i.i.d. sequence of ${\rm Geom}(v)$ random variables, where $u = \f{e^\Aa}{1 + e^\Aa}$ and $v = \f{u-p}{(1-p)u}$.
\end{lemma}
\begin{proof} 
Applying the dynamic programming principle to the definition of $h_\Tt^{\Aa,\Bb}$, we have 
\begin{align*}
    h_\Tt^{\Aa,\Bb}(n,m) &= \sup_{-\infty < k \le j \le m}[f^{\Aa,\Bb}(k) + \Tt(k,0;j,n-1) + \Tt(j,n;m,n)] \\
    &= \sup_{-\infty < j \le m}[h_\Tt^{\Aa,\Bb}(n-1,j) + t^n[j+1,m]].
\end{align*}
Hence,
\begin{align}
J^{\Aa,\Bb}_{(m,n)} &= h_\Tt^{\Aa,\Bb}(n,m) - h_\Tt^{\Aa,\Bb}(n-1,m) \nonumber \\
&= \sup_{-\infty < j \le m}[h_\Tt^{\Aa,\Bb}(n-1,j) - h_\Tt^{\Aa,\Bb}(n-1,m)  + t^n[j+1,m]] \nonumber \\
&= \sup_{-\infty < j \le m}[t^n[j+1,m] - I^{\Aa,\Bb,n-1}[j+1,m]] \nonumber \\
&= \sup_{-\infty < j \le m + 1}[t^n[j,m] - I^{\Aa,\Bb,n-1}[j,m]] = [Q(t^n,I^{\Aa,\Bb,n-1})]_m. \label{QtnIun}
\end{align}
Then,
\begin{align*}
&\quad \,I^{\Aa,\Bb}_{(m,n)} = h_\Tt^{\Aa,\Bb}(n,m) - h_\Tt^{\Aa,\Bb}(n,m-1) \\
&= h_\Tt^{\Aa,\Bb}(n,m) - h_\Tt^{\Aa,\Bb}(n-1,m) + h_\Tt^{\Aa,\Bb}(n-1,m) - h_\Tt^{\Aa,\Bb}(n-1,m-1) + h_\Tt^{\Aa,\Bb}(n-1,m-1) - h_\Tt^{\Aa,\Bb}(n,m-1) \\
&= [Q(t^n,I^{\Aa,\Bb,n-1})]_m + I^{\Aa,\Bb,n-1}_m - [Q(t^n,I^{\Aa,\Bb,n-1})]_{m-1} = [R(t^n,I^{\Aa,\Bb,n-1})]_{m}.
\end{align*}
The ``consequently" part follows by Lemma \ref{lem:SJBurke} and induction, following the same technique as Lemma \ref{lem:PoiLevolve}. 
\end{proof}

\subsubsection{The EJS-Rains identity and exit point bounds}
We first state the following. Its proof is a straightforward verification using moment generating functions, just as in Lemma \ref{lem:HamRND}.
\begin{lemma} \label{CM-Bernoulli}
On the space $\{0,1\}^m$, for $u \in (0,1)$, let $\Pp_u$ be the measure corresponding to i.i.d. {\rm Ber($u$)} random variables $\{X_i\}_{1 \le i \le m}$. For $w,u \in (0,1)$, the measures $\Pp_u$ and $\Pp_w$ are mutually absolutely continuous, with 
\[
\f{d \Pp_w}{d \Pp_u} = \exp\Bigl(\bigl(\log(w/(1-w)) - \log(u/(1-u))\bigr)\sum_{i = 1}^m X_i\Bigr).
\]
\end{lemma}

For $m,n \ge 0$ and $\Aa,\Bb \in \R$ we make the following definitions.
\begin{align*}
M_\Tt^\Aa(n,m) &:= m\f{e^\Aa}{e^\Aa + 1} + (n+1) \f{p}{(1-p)e^\Aa - p},\\
\qquad R_\Tt^{\Aa,\Bb}(n,m) &:= \int_{\Bb}^\Aa M_\Tt^w(n,m)\,d w = m\log\Bigl(\f{1+e^\Aa}{1+e^\Bb}\Bigr) +(n+1)\Bigl[\log\Bigl(\f{(1-p)e^\Aa - p)}{(1-p)e^\Bb - p)}\Bigr) + (\Bb- \Aa) \Bigr], \\
\zeta_\Tt(n,m) &:=\arg \inf_{\Aa  > -\log \lambda} M_\Tt^\Aa(n,m) =
\begin{cases}
 \log\Biggl(\f{p + \sqrt{\f{n+1}{m}p(1-p)}}{(1-p) + \sqrt{\f{n+1}{m}p(1-p)}}\Biggr) &\f{m}{n+1} > \f{p}{1-p} \\
 \infty &\f{m}{n+1} \le \f{p}{1-p},
    \end{cases}\\
\gamma_\Tt(n,m) &: = \begin{cases}
m\Biggl(p + \sqrt{\f{n+1}{m}p(1-p)}\Biggr) + (n+1)\f{p\Bigl(1 - p - \sqrt{\f{n+1}{m}p(1-p)}\Bigr)}{\sqrt{\f{n+1}{m}p(1-p)}} & \f{m}{n+1} > \f{p}{1-p} \\
m & \f{m}{n+1} \le \f{p}{1-p},
    \end{cases}
\end{align*}
and $\gamma_\Tt(n,m)$ is also equal to $\inf_{\Aa  > -\log \lambda} M_\Tt^\Aa(n,m)$. To compute the last two quantities, it is helpful to minimize $mu + (n+1)\f{p(1-u)}{u-p}$ over $u \in (p,1]$, then change variables to $\Aa = \log(\f{u}{1-u})$. In taking the limit to the DL, the restriction $\f{m}{n + 1} > \f{p}{1-p}$ corresponds exactly to the condition $\rho  > \lambda^{-1}$.

\begin{lemma} \label{lem:SJ_EJS}
For $\Aa,\Bb \in \bigl(-\log \lambda,\infty\bigr)$,
\[
\E\Bigl[\exp\Bigl((\Aa - \Bb) h_\Tt^{\Aa,\Bb}(n,m)\Bigr)\Bigr] = \exp(R_\Tt^{\Aa,\Bb}(n,m))
\]
\end{lemma}
\begin{proof}
Define $\mathcal E^{\Aa,\Bb}(m,n) = h_\Tt^{\Aa,\Bb}(m,n) - f^{\Bb}(m) =  \sum_{j = 0}^n J_{(m,j)}^{\Aa,\Bb}$, and $\mathcal E^{\Aa} = \mathcal E^{\Aa,\Aa}$. By Lemma \ref{lem:SJQueue}, $\mathcal E^{\Aa,\Bb}(m,n)$ is distributed as the sum of $n + 1$ i.i.d. ${\rm Geom}(v_\Aa)$ random variables, where 
\[
v_{\Aa} = \f{u -p}{(1-p)u},\quad u = \f{e^{\Aa}}{1 + e^{\Aa}}.
\]
We use the change of measure in Lemma \ref{CM-Bernoulli} to transform the i.i.d. weights on the positive horizontal edges with parameter $\Bb$ (Bernoulli parameter $\f{e^{\Bb}}{1 + e^{\Bb}}$) to i.i.d. weights with parameter $\Aa$ (Bernoulli parameter $\f{e^{\Aa}}{1 + e^{\Aa}})$:
\begin{align*}
    &\quad \; \E\Bigl[\exp\Bigl((\Aa - \Bb) h_\Tt^{\Aa,\Bb}(n,n)\Bigr)\Bigr] = \E\Bigl[\exp\Bigl((\Aa - \Bb)(\mathcal E^{\Aa,\Bb}(m,n)+ f^{\Aa}(m) )\Bigr)\Bigr]  \\
    &= \Bigl(\f{1 + e^\Aa}{1 + e^\Bb}\Bigr)^m \E\Bigl(\exp\Bigl((\Aa - \Bb) \mathcal E^{\Aa}(m,n) \Bigr)\Bigr) \\
    &= \Bigl(\f{1 + e^\Aa}{1 + e^\Bb}\Bigr)^m \Bigl(\f{v_\Aa}{1 - e^{\Aa-\Bb}(1-v_\Aa)}\Bigr)^{n+1}  = \Bigl(\f{1 + e^\Aa}{1 + e^\Bb}\Bigr)^m\Bigl(\f{e^\Bb((1-p)e^\Aa - p)}{e^\Aa((1-p)e^\Bb - p)}\Bigr)^{n+1}.
\end{align*}
the requirement $\Bb > -\log \lambda$ is exactly what is needed for the moment generating function in the second line to be finite. 
\end{proof}

\begin{lemma} \label{lem:SJRab}
For $\ve \in (0,1-p)$, there exists a constant $C = C(\ve) > 0$ so that for $\f{e^\Aa}{e^\Aa + 1},\f{e^\Bb}{e^\Bb + 1} \in (p +\ve, 1-\ve)$ and $p + \ve < \f{n+1}{m} < \ve^{-1}$,
\[
\Bigl|R^{\Aa,\Bb} - \gamma - \f{\gamma}{3 w(1+ e^\zeta)^4}((\Aa - \zeta)^3 - (\Bb - \zeta)^3) \Bigr| \le C(n + m) ((\Aa - \zeta)^4 + (\Bb - \zeta)^4),
\]
where $z = z(n,m) = \f{e^\zeta}{1 + e^\zeta}$ and $w = w(n,m) = \f{\gamma(z - p)}{m}$. 
\end{lemma}
\begin{proof}
It is straightforward to check the existence of constants $C_1,C_2$ (depending on $\ve > 0$) so that $C_1(n+m) \le \gamma(n,m) \le C_2(n+m)$ whenever $p + \ve < \f{n+1}{m} < \ve^{-1}$. 
 For $\Aa > -\log \lambda$, and $\zeta = \zeta_\Tt(n,m)$, define $u = \f{e^\Aa}{1 + e^\Aa}$ and $z = \f{e^\zeta}{1 + e^\zeta}$. Observe that $z$ is the minimizer of $mu + (n+1)\f{p(1-u)}{u-p}$ over $u \in (p,1)$, so $z = p + \sqrt{\f{(n+1) p(1-p)}{m}}$. Next, observe that 
\begin{align*}
M^\Aa - \gamma &= mu + (n+1)\f{p(1-u)}{u-p} - m z - (n+1) \f{p(1-u)}{z - p} \\
&= (u - z)\Bigl(m - (n+1) \f{p(1-p)}{(u - p)(z - p)}\Bigr) \\
&= \f{(u - z)m}{(u - p)(z - p)} \Bigl((u-p)(z-p) - \f{n+1}{m}(1-p)p\Bigr) \\
&= \f{(u - z)m}{(u - p)(z - p)} \Bigl((u-p)(z-p) - (z - p)^2\Bigr) 
= \f{m}{u - p}(u - z)^2.
\end{align*}
Now, set $w = \f{\gamma(z - p)}{m}$. Then,
\[
M^\Aa - \gamma - \f{\gamma}{w}(u - z)^2 = - \f{\gamma}{(u-p) w}(u - z)^3.
\]
 Observe that the conditions on $n,m$ state that $\f{n+1}{m} = \f{(z - p)^2}{p(1-p)}$ is bounded away from $0$ and $\infty$. Combined with the inequality $C_1(n+m) \le \gamma \le C_2(n+m)$, we see that $w$ is bounded. The remainder of the proof now follows by integrating $M^\Aa$, just as in the proof of Lemma \ref{lem:HamRupbd}.
\end{proof}

\begin{theorem} \label{thm:SJ_exit_pt}
Let $\rho > -\log \lambda = \f{p}{1-p}$, $\mu \in \R$, and  define $\chi,\beta,\tau$ by \eqref{SJ_param}. As in \eqref{meanparam}, let $\beta_N(\mu) = \beta + \f{2\mu \chi}{\tau}N^{-1/3}$. Then, for each $y \in \R$ and $t > 0$, there exists a constant $C= C(y,\mu,\rho,t) > 0$ so that, for all sufficiently large positive $M$,
\[
\limsup_{N \to \infty} \Pp(|Z_\Tt^{\log(\beta_N(\mu)) - \log(1-\beta_N(\mu))}(\lfloor tN \rfloor, \lfloor \rho N + y\tau N^{2/3} \rfloor)| > MN^{2/3}) \le e^{-CM^3}.
\]
\end{theorem}
\begin{remark}
    Note that the parameter choice of $\log(\beta_N(\mu)) - \log(1-\beta_N(\mu))$ corresponds to initial data for i.i.d. ${\rm Ber \rm}(\beta_N(\mu))$ random variables. 
\end{remark}
\begin{proof}
The proof follows the same for Poisson LPP in Theorem \ref{thm:Hamexitpt}. The needed inputs specific to the SJ model are Lemmas \ref{lem:SJ_EJS} and \ref{lem:SJRab}. Note that the condition $\rho > -\log \lambda$ implies that, for sufficiently large $N$,
\[
\f{\lfloor \rho N + y\tau N^{2/3} \rfloor}{\lfloor tN \rfloor + 1 } > \f{p}{1-p}. \qedhere
\]
\end{proof}

\appendix
\section{Random walk and Brownian motion} \label{sec:BMRW}
\begin{lemma}\label{lem:Arw}
	Fix $N\geq 1$. Let $S_n=\sum_{i=1}^{n} X_i$ and $S^N_x=N^{-1/2}S_{\lfloor xN \rfloor}$ be a random walk with i.i.d.\ mean zero steps  $\{X_i\}_{i=1}^\infty$   with  variance $\sigma^2$ and  $\E(e^{\theta X_1})<\infty$ for sufficiently small $\theta > 0$. 
	Then there exist constants $C,c>0$ independent of $N$, such that  for $L\in\N$
	\begin{equation}\label{Arw13} 
		\Pp\big(\sup_{0\leq x\leq \infty}[S^N_x-x-L]>0\big)\leq Ce^{-c\sqrt{L}}.
	\end{equation} 
\end{lemma}
\begin{proof}
	Using Markov's inequality, there exists a constant $C(\sigma)>0$ so that,  for every $n\in\N$ and $t>0$, 
	\begin{equation}\label{app108} 
		\Pp\big(S_n>t\big)=\Pp(e^{n^{-1/2} S_n}>e^{n^{-1/2} t})\leq  \big(1+\tfrac12\sigma^2n^{-1}+o(n^{-1})\big)^ne^{-n^{-1/2}t}\leq Ce^{-\frac{t}{\sqrt{n}}}.
	\end{equation}

	By applying Etemadi's inequality followed by \eqref{app108} to both $S_i$ and $-S_i$, we have  
	\begin{equation}\label{Arw3}
		\begin{aligned}
			&\Pp\big(\max_{0\leq y \leq x}S^N_y>t\big)=\Pp\big(\max_{1\leq i \leq \lfloor xN \rfloor }S_i>N^{1/2}t\big)\\
			&\le \Pp\big(\max_{1\leq i \leq \lfloor xN \rfloor}|S_i|>N^{1/2}t\big)
			\le 3 \max_{1\leq i \leq \lfloor xN \rfloor} \Pp\big(|S_i|>N^{1/2}t/3\big)
			\leq Ce^{-\frac{t}{3\sqrt{x}}},  \qquad \forall t>0.
		\end{aligned}
	\end{equation}
	Let $z_0=0$, and for $i\in\N$ define
	\begin{equation*}
		z_i=L\sum_{j=1}^i2^j = L(2^{i+1}-2).
	\end{equation*}
	Also define the interval $I_i=(z_{i-1},z_i]$. Note that the minimum of $x\mapsto x+L$ on the interval $I_i$ is 
	\begin{equation*}
		M_i:=L(2^i-1), \qquad \forall i\in\N.
	\end{equation*}
	Let $M_0=0$, and define the random variable
	\begin{equation*}
		T=\inf\{i\in\N:\text{$S^N_{z_{i-1}}\leq M_{i-1}$,\,\, $\max_{x\in I_i}S^N_{x}>M_i$}\}.
	\end{equation*}
	In words,  $I_T$ is  the first interval on which the random walk $S^N$ exceeds the minimum of the function $x\mapsto x+L$ on that interval. Then,
	\begin{align*} 
		\{\sup_{0\leq x\leq \infty}[S^N_x-x-L]>0\}&\subseteq \bigcup_{j=1}^\infty\{T=j\}\subseteq \bigcup_{j=1}^\infty\{\text{$S^N_{z_{j-1}}\leq M_{j-1}$,\,\, $\max_{y\in I_j}S^N_{y}>M_j$}\}
		\\
		&\subseteq \bigcup_{j=1}^\infty \bigl\{ \max_{0\le y\le z_j-z_{j-1}} (S^N_{z_{j-1}+y}-S^N_{z_{j-1}}) > M_j-M_{j-1}\bigr\}. 
	\end{align*} 
	A union bound  and \eqref{Arw3} with $x=z_j-z_{j-1}=2^jL$ and $t=M_j-M_{j-1}=2^{j-1}L$ gives 
	\[
		\Pp\big(\sup_{0\leq x\leq \infty}[S^N_x-x-L]>0\big)\leq \sum_{j=1}^\infty Ce^{-\frac{2^{j-1}L}{3(2^jL)^{1/2}}}\leq Ce^{-cL^{1/2}}. \qedhere
		\]
	\end{proof}

\begin{proof}[Proof of Lemma \ref{lem:sseq}]
		Recall that we want to show that for a scaled random walk $S^N$ converging to Brownian motion with drift $\mu$,  there exists a deterministic subsequence $N_j$ and  a finite random constant $M > 0$ such that with probability one,  
	$
	S^{N_j}(x) \le (3 + |\mu|)|x| + M
	$
	for all $x \in \R$ and sufficiently large $j$. It suffices to show the result for $\mu=0$.  
		By Skorokhod representation, we can assume that, almost surely,  $S^N\rightarrow B$ on compact intervals. For $M>0$, define the event 
	\begin{equation}\label{fd7}
		\mathcal{B}_{M}=\Big\{\sup_{x\in \R}\tspb [\tspb B(x)-|x|\tspb] -M<0\Big\}.
	\end{equation}
Note that the quantity $\sup_{x\in \R}\tspb [\tspb B(x)-|x|\tspb]$ is almost surely finite, so
\begin{equation}\label{fd1}
	\frac{\ve_M}2:=1-\Pp\big(\mathcal{B}_{M}\big)\stackrel{M\rightarrow\infty}{\longrightarrow}0.
\end{equation}
We now construct a subsequence  $\bq{N}^1=\{N^1_l\}_{l\in\N}$ in the following manner. Since $S^N\rightarrow B$ uniformly on compacts, there exists $N^1_1$ so that
\begin{equation*}
	\Pp\Big(\{\sup_{|x|\leq 1}  [S^N(x)-2|x|\tspb]-1 <0 \,\text{ for all $N\geq N^1_1$}\}\cap\mathcal{B}_1\Big)> 1-\frac{\ve_1}4.
\end{equation*}
The general step is as follows.  For $l\in\N$, we   find $N^1_{l+1}>N^1_{l}$ so that 
\begin{equation}\label{fd2}
	\Pp\Big(\bigl\{\sup_{|x|\leq l+1} [S^N(x)-2|x|\tspb]-1<0 \text{ for all $N\geq N^1_{l+1}$}\bigr\}\cap\mathcal{B}_1\Big)>1- \ve_12^{-l-2}
\end{equation}
Define 	$A^1_l:=\{\sup_{|x|\leq l} S^N(x)-2|x|-1<0 \text{ for all $N\geq N^1_l$}\}\cap\mathcal{B}_1$, and $\mathcal{A}^1:=\bigcap_{l\in\N}A^1_l$.
From the definition of $A^1_l$, \eqref{fd1} and \eqref{fd2} we conclude that
$
	\Pp\big(\mathcal{A}^1\big) \ge 1 - \f{\ve_1}2 >1-\ve_1.
$
Next, we claim that on the event $\mathcal{A}^1$, for all sufficiently large $\ell$,
\begin{equation}\label{fd4}
	S^{N^1_l}(x)<3|x|+1 \qquad \forall x\in\R.
\end{equation}
We show \eqref{fd4} for $x\geq 0$. For $\ell \in \N$, let 
$
	E_l:=\{\exists\, x\geq 0: \ S^{N_l^1}(x)\geq 3x+1\}\cap \mathcal{A}^1
$
Observe that, because  $E_l \subseteq \mathcal{A}^1 \subseteq A^1_l$, 
\begin{equation}\label{fd5}
	E_l\subseteq \big\{ S^{N^1_l}(l)<2l+1,\; \sup_{x \ge l}[S^{N^1_l}(x)-3x-1]>0\big\}.
\end{equation}
Furthermore, by  Lemma \ref{lem:Arw} there exist constants $C,c>0$ so that
\begin{equation*}
	\Pp\big(S^{N^1_l}(l)<2l+1, \,\, \sup_{x \ge l}S^{N^1_l}(x)-3x-1>0\big)\leq \Pp\big(S^{N^1_l}(0)=0,\,\, \sup_{x \ge 0}S^{N^1_l}(x)-3x-l>0\big)\leq Ce^{-c\sqrt l}.
\end{equation*}
Combined with \eqref{fd5}, we conclude that $\sum_l\Pp(E_l)<\infty$, so \eqref{fd4} holds for $x \in [0,\infty)$ by Borel-Cantelli. The extension to $x \le 0$ follows by symmetry. For $M\geq 2$, construct a new subsequence $\bq{N}^M=\{N^M_l\}_{l\in\N}$ from $\bq{N}^{M-1}$ as follows. For $l\in\N$, let $N^M_l\in\bq{N}^{M-1}$ be such that $N^M_l\geq N^{M-1}_l$, and 
\begin{equation}\label{fd6}
	\Pp\Big(\big\{\sup_{|x|\leq l M} S^N(x)-2|x|-M<0 \text{ for all $N\geq N^M_l$}\big\}\cap\mathcal{B}_M\Big)>1- \ve_M2^{-l-2}.
\end{equation}
We define $A^M_l=\{\sup_{|x|\leq l M} S^N(x)-2|x|-M<0 \text{ for all $N\geq N^M_l$}\}\cap\mathcal{B}_M$  and $\mathcal{A}^M:=\bigcap A^M_l$. As before we conclude  by a Borel-Cantelli argument that, on  the event $\mathcal A^M$, for sufficiently large $l$,
\begin{equation}\label{fd10}
	S^{N_l^M}(x)<3|x|+M \qquad \forall x\in\R.
\end{equation}
From the construction of $\mathcal{A}^M$ we see that
\begin{equation}\label{fd8}
	\mathcal{A}^M\subseteq \mathcal{A}^{M+1}, \qquad \forall M\in\N.
\end{equation} 
 Indeed, as by definition $N^{M}_l\leq N^{M+1}_l$, for every  $l\in\N$, $A^M_l\subseteq A^{M+1}_l$. Moreover, from \eqref{fd6}
\begin{equation}\label{fd9}
	\Pp(\mathcal{A}^M)>1- \ve_M.
\end{equation}
 Consider now the diagonal sequence $\bq{m}=\{m_i\}_{i=1}^\infty$ defined as $m_i:=N^i_i$.
From \eqref{fd8} the limit  $\mathcal{A}:=\lim_{M\rightarrow \infty}\mathcal{A}^M$ is nonempty. From \eqref{fd9} and \eqref{fd1},  $\Pp(\mathcal{A})=1$. Let $\omega\in \mathcal{A}$. There exists $M(\omega)\in\N$ such that $\omega\in \mathcal{A}^M$. From \eqref{fd10}, for large enough $l$,
\begin{equation*}
	S^{N^M_l}(x;\omega)<3|x|+M(\omega) \qquad \forall x\in\R.
\end{equation*}
On the other hand, the sequence $\bq{m}$ was constructed such that $\{m_i\}_{i={M+1}}^\infty\subseteq \bq{N}^M$, which implies that for large enough $i$,
\begin{equation*}
	S^{m_i}(x;\omega)<3|x|+M(\omega) \qquad \forall x\in\R. \qedhere
\end{equation*} 

\end{proof}
We now prove Lemma \ref{lem:Poicoup}.
\begin{proof}[Proof of Lemma \ref{lem:Poicoup}]
We restrict to the set $x \ge 0$, the extension to $x < 0$ holding by a symmetric argument. For $x \ge 0$, we may write
\[
\nu_N(x) = \sum_{i = 1}^{\lfloor x \rfloor} \nu_N(i-1,i] + Y^N_{\lfloor x \rfloor + 1}(x).
\]
Note the random variables $\{\nu_N(i-1,i]\}_{i \ge 1}$ are i.i.d. Poisson random variables with mean $\mu_N$. We will use $X_i^N$ to denote $\nu_N(i-1,i]$. The process $\{Y^N_{i+1}(x + i):x \in [0,1]\}$ is independent of $\{X_j\}_{j \le i}$, and given $X_{i+1}^N$, it can be described in law by
\[
Y^N_{i + 1}(x + i) \deq\sum_{j = 1}^{X_{i + 1}^N} \ind(U^N_{i,j} \le x), \quad x \in [0,1].
\]
where $\{U_{i,j}\}_{i,j \ge 1}$ is an infinite array of i.i.d. uniform random variables on $[0,1]$, independent of $X_{i + 1}^N$. Consider the random walk determined by the variables $X_i^N$. Write its linearly interpolated version as
\[
f_N(x) =  \sum_{i = 1}^{\lfloor x \rfloor} X_i^N + Z_{\lfloor x \rfloor + 1}^N(x).
\]
By the assumptions on the sequence $\mu_N$, it follows that 
\[
x \mapsto \f{f_N(Nx) - \beta N x}{\chi \sqrt N}
\]
converges in the sense of uniform convergence on compact sets (on $\R_{\ge 0}$) to a Brownian motion with diffusivity $\sigma$ and drift $\mu$.

We now show that  $\f{f_N(Nx) - \nu_N(Nx)}{\sqrt N}$ converges to $0$, almost surely uniformly on compact sets. In particular, we show that, for any $M > 0$ and $\ve > 0$,
\[
\Pp\Bigl(\sup_{x \in [0,M]} \Bigl|\f{Y^N_{\lfloor Nx \rfloor + 1}(Nx) - Z^N_{\lfloor Nx \rfloor + 1}(Nx)}{\sqrt N}
\Bigr| > \ve \text{ for infinitely many } N\Bigr) = 0,
\]
and then the desired result follows by continuity of measure. Observe that 
\begin{align*}
    \Pp\Bigl(\sup_{x \in [0,M]} \Bigl|Y^N_{\lfloor Nx \rfloor +1}(Nx) - Z^N_{\lfloor Nx \rfloor +1}(Nx)\Bigr| > \ve \sqrt N \Bigr) = \Pp\Bigl(\sup_{0 \le i \le MN} \sup_{x \in [i,i + 1]}|Y^N_{i + 1}(x) - Z^N_{i +1}(x)| > \ve \sqrt N  \Bigr).
\end{align*}
Both $Y^N_{i +1}(x)$ and $Z^N_{i+1}(x)$ are functions that increase from $0$ to $X^N_{i + 1}$ between $x = i$ and $x = i + 1$. Their maximum difference on $[i,i+1]$ is $X^N_{i + 1}$. Hence, the above is bounded by 
\[
\Pp(\sup_{0 \le i \le MN} X_{i + 1}^N \ge \ve \sqrt N)  = 1 - (1-\Pp(X_{i +1}^N \ge \ve \sqrt N))^{MN + 1}.
\]
Note that $X_{i + 1}^N$ has a finite moment generating function that is bounded in $N$. Hence, there exists a constant $c > 0$ so that this probability is bounded by 
\[
1 - (1 - c e^{-\ve N})^{MN+1}  = 1 - \exp((MN+1)\log(1 - ce^{-\ve\sqrt N})) = 1 - e^{-c(MN+1)e^{-\ve \sqrt N} + o(1)}.
\]
This quantity summable in $N$, so the proof is complete by the Borel-Cantelli lemma.
\end{proof}

\section{Existence of jointly invariant measures for the Poisson lines model} \label{sec:PoiL_exist} In this subsection, we prove the existence of a weak form of Busemann functions that are jointly invariant for the Poisson lines model. A description of jointly invariant measures for this model when one takes time as the continuous coordinate was alternatively shown in \cite{Ferrari-Martin-2005}. Our convention to use time as the discrete coordinate allows us to connect this model to the classical version of Burke's theorem. While uniqueness of the invariant measure we show here should follow by standard methods, we do not prove uniqueness here; as we  in fact show that any jointly invariant measure converges to the SH.  
This classical idea of constructing invariant distributions from limit points of Ces\`aro averages was applied to Busemann functions in  \cite{Damron_Hanson2012} and has since been adapted in \cite{Ahlberg_Hoffman,Georgiou-Rassoul-Seppalainen-17b,Janjigian-Rassoul-2020b,Janj-Rass-Sepp-22,Groathouse-Janjigian-Rassoul-2023}. Extend the height function $h$ to allow for an arbitrary initial time by setting, for $n > m$,
\[
h_U(n,y; m,\nu) = \sup_{-\infty < x \le y}\{\nu(x) + U(x,m+1;y,n)\}.
\]
and $h_U(m,y;m,\nu) = \nu(y)$. In our original notation, $h_U(n,y;\nu) = h_U(n,y;-1,\nu)$. We observe that for any $k$-tuple of initial data $\nu_1,\ldots,\nu_k$ and $m,r \in \Z$, we have 
\be \label{eq:shift_invar}
\{h_U(n,y ;m, \nu_i): n \ge m, y \in \R, i \in \{1,\ldots,k\}\} \deq \{h_U(n + r,y ;m + r, \nu_i): n \ge m, y \in \R, i \in \{1,\ldots,k\}\},
\ee
simply because the environment is invariant under these temporal shifts. 

Let $\lambda_1,\cdots,\lambda_k \in (1,\infty)$, $i \in \{1,\ldots,k\}$, and consider the coupling of Poisson processes introduced in Section \ref{sec:Poiscoup}. For $i \in \{1,\ldots,k\}$, let $\nu_i = \nu^{\log(\lambda_i)}$; marginally, each $\nu_i$ is a Poisson point process of intensity $\lambda_i$. For $m \in \Z_{<0}$, let $P_m$ denote the probability measure on the product space $\UC^\Z \times \UC^k \times \UC^k \times \R^k_{>0}$ of the process
\begin{multline*}
\Bigl(\{F_i\}_{i \in \Z}, \{h_U(-1,x ; m,\nu_i) - h_U(-1,0 ;m,\nu_i): x \in \R\}_{1 \le i \le k}, \\
 \{h_U(0,x ; m,\nu_i) - h_U(0,0 ;m,\nu_i): x \in \R\}_{1 \le i \le k},  \\
\{h_U(0,0;m;\nu_i) - h_U(-1,0;m,\nu_i)\}_{1 \le i \le k} \Bigr),
\end{multline*}
where we recall that $F_i(x)$ is defined pointwise by \eqref{nufundef}: In particular, $F_i(0) = 0$ and $F_i(x,y] = F_i(y) - F_i(x)$. For shorthand, we will let $\Bigl(\{F_i^m\}_{i \in \Z}, \{V_i^{1,m}\},\{V_i^{2,m}\},\{Q_i^m\}_{1 \le i \le k} \Bigr)$ denote a random vector with distribution $P_m$ (noting that the $F_i$ do not change with $m$). For $x \in \R$, let 
\[
Q_i^m(x) = V_i^{2,m}(x) - V_i^{1,m}(x) + Q_i^m
\]
so that $Q_i^m(0) = Q_i^m$.
For $M \in \Z_{<0}$ let $\mbf P_M = \f{1}{M} \sum_{m = -M}^{-1} P_m$,  and let $\Bigl(\{\mbf F_i^M\}_{i \in \Z}, \{\mbf V_i^{1,M}\},\{\mbf V_i^{2,M}\},\mbf Q^M \Bigr)$ be a random vector with distribution $\mbf P_M$. Similar as before, let $ \mbf Q_i^M(x) = \mbf V_i^{2,M}(x) - \mbf V_i^{1,M}(x) + \mbf Q_i^M$. It is important to note that these measures depend on the choice of parameters $\lambda_1,\ldots,\lambda_k$, but we subsume the dependence for ease of notation. 
\begin{lemma} \label{lem:SJ_joint_exist}
The sequence of probability measures $\mbf P_M$ is tight. Let $\mbf P$ be any subsequential limit, and let 
\[
(\{\mbf F_i\}_{i \in \Z}, \{\mbf V^1_i\}_{1 \le i \le k}, \{\mbf V^2_i\}_{1 \le i \le k},\{\mbf Q_i\}_{1 \le i \le k})
\]
be a random vector with distribution $\mbf P$. Then,  $\{\mbf V^1_i\}_{1 \le i \le k} \deq \{\mbf V^2_i\}_{1 \le i \le k}$, and $\mbf P$-almost surely, for every $y \in \R$ and $1 \le i \le k$,
\be \label{eqn:QR_limit}
\mbf Q_i(y) = Q(\mbf F_0, \mbf V_i^1)(y),\quad \text{and} \quad      \mbf V_i^2(y)  = R(\mbf F_0, \mbf V_i^1)(y),
\ee
where we define $ \mbf Q_i(y) = \mbf V_i^{2}(y) - \mbf V_i^{1}(y) + \mbf Q_i$
In particular, if the random vector $\{\mbf V_i^1\}_{1 \le i \le k}$ is taken independently of the original sequence $\{F_i\}_{i \in \Z}$, then for any $x \in \R$, and $n \ge 1$,
\be \label{eqn:PoiLstat}
\{h_U(n,x+y ; \mbf V_i^1) - h_U(n,x ; \mbf V_i^1):y \in \R\}_{1 \le i \le k} \deq \{\mbf V_i^1(y): y \in \R\}_{1 \le i \le k}.
\ee
\end{lemma}
\begin{proof}
Lemma \ref{lem:PoiLevolve} implies that under each $P_m$, for each $1 \le i \le k$, $V_i^{1,m}$ and $V_i^{2,m}$ are each Poisson processes of intensity $\lambda_i$, and for $x \in \R$, $Q_i^m(x) \sim {\rm Geom}(1 - \lambda_i^{-1})$. Furthermore, the same lemma implies that, $P_m$-a.s., $Q_i^m(x) = Q(F_0^m,V_i^{1,m})(x)$ and $V_i^{2,m}(x) = R(F_0^m, V_i^{1,m})(x)$. Hence, the marginals under the averaged measure $\mbf P_M$ also have these same distributions. In particular, each marginal of $\mbf P_M$ is tight, so $\mbf P_M$ is tight, and under any subsequential limit $\mbf P$, we have these same marginal distributions. Additionally, the queuing relations also hold for the averaged measure. The shift invariance \eqref{eq:shift_invar} and the averaging for $\mbf P_M$ imply that $\{\mbf V_i^1\}_{1 \le i \le k} \deq \{\mbf V_i^2\}_{1 \le i \le k}$ under any subsequential limit $\mbf P$.

Let $\mbf P$ be any such subsequential limit along the subsequence $M_j$. Use Skorokhod representation to find a coupling of  $\Bigl(\{\mbf F_i^{M_j}\}_{i \in \Z}, \{\mbf V_i^{1,M_j}\},\{\mbf V_i^{2,M_j}\},\mbf \{Q_i^{M_j}\}_{1 \le i \le k} \Bigr)$ and  $\Bigl(\{\mbf F_i\}_{i \in \Z}, \{\mbf V_i^{1}\},\{\mbf V_i^{2}\},\{\mbf Q_i\}_{1 \le i \le k} \Bigr)$  so that the convergence holds almost surely in $\UC^\Z \times \UC^k \times \UC^k \times \R_{>0}^k$.  The convergence of each $\mbf V_i^{\ell,M_j} \to \mbf V_i^\ell$ for $\ell = 1,2$ and $\mbf F_i^{M_j} \to \mbf F_i$ is the convergence on $\UC$, namely that of local convergence of hypographs. These processes are Poisson point processes and are therefore locally constant except at a random countable set. In particular, with probability one, there are no discontinuities a rational points for any function in the prelimiting sequence or in the limit. Since the functions are locally constant (in particular, locally continuous) at those rational points, the convergence of hypographs implies the almost sure convergence
\[
\mbf V_i^{\ell,M_j}(x) \to \mbf V_i^\ell(x), \quad \text{and}\quad \mbf F_i^{M_j}(x) \to \mbf F_i(x) \quad \forall x \in \Q, \ell \in \{1,2\}. 
\]

To prove \eqref{eqn:QR_limit}, it suffices to show $\mbf Q_i(y) = Q(\mbf F_0, \mbf V_i^1)(y)$ because then, by definition of $\mbf Q_i(x)$ and the mappings $Q,R$ \eqref{eqn:PoiLqueue},
\[
\mbf V_i^2(y) = \mbf V_i^1(y) + \mbf Q_i(x) - \mbf Q_i(0) =  \mbf V_i^1(y) + Q(\mbf F_0, \mbf V_i^1)(y) - Q(\mbf F_0, \mbf V_i^1)(0) =  R(\mbf F_0, \mbf V_i^1)(y).
\]
Recall the notation $\mbf F_0(x,y] = \mbf F_0(y) - \mbf F_0(x)$. By definition, 
$
Q(\mbf F_0, \mbf V_i^1)(y) = \sup_{-\infty < x \le y}\{\mbf F_0(x,y] - \mbf V_i^1(x,y]\},
$
and because the sets of discontinuity of both $\mbf F_0$ and $\mbf V_i^1$ are isolated, it suffices to show equality just for $y \in \Q$, in which case it also suffices to take the supremum over $x \in \Q$. It suffices to show that, under the coupling we have defined, for any $y \in \Q$,
\[
\mbf Q_i(y) = \lim_{j \to \infty} \mbf Q_i^{M_j}(y) = \sup_{-\infty < x \le y: x \in \Q}\{\lim_{j \to \infty}[\mbf F_0^{M_j}(x,y] - \mbf V_i^{1,\mbf M_j}(x,y]] 
  \} = \sup_{-\infty < x \le y: x \in \Q}\{\mbf F_0^{M_j}(x,y] - \mbf V_i^1(x,y]\}.
\]
We know that, before taking limits, 
\[
\mbf Q_i^{M_j}(y) = Q(\mbf F_0^{M_j},\mbf V_i^{1,M_j})(y) = \sup_{-\infty < x \le y}\{\mbf F_0^{M_j}(x,y] - \mbf V_i^{1,\mbf M_j}(x,y]
  \}
\]
so that $\mbf Q_i^{M_j}(y) \ge \mbf F_0^{M_j}(x,y] - \mbf V_i^{1,\mbf M_j}(x,y]$ for each $x \le y$. Taking limits yields 
\[
\mbf Q_i(y) \ge \mbf F_0 (x,y] - \mbf V_i^{1}(x,y] \quad \forall x \le y, x \in \Q \quad \Longrightarrow \quad  \mbf Q_i(y) \ge \sup_{-\infty < y \le x: x \in \Q}\{\mbf F_0 (x,y] - \mbf V_i^{1}(x,y]\}.
\]
Both sides of the inequality on the right in the display above  have the ${\rm Geom }(1- \lambda_i^{-1})$ distribution; the left-hand side because each $\mbf Q_i^{M_j}(y)$ has this distribution, and the right-hand side because it is the queue length of independent Poisson service and arrival times (using Theorem \ref{thm:Burke_Poi}). Hence, the two quantities are almost surely equal, completing the proof of \eqref{eqn:QR_limit}. 

To conclude, Lemma \ref{lem:PoiLevolve} implies that the mapping $R$ governs the evolution of $h_U$ from one line to the next. Hence, \eqref{eqn:QR_limit} along with the equality $\{\mbf V_i^1\}_{1 \le i \le k} \deq \{\mbf V_i^2\}_{1 \le i \le k}$ demonstrates the stationarity of \eqref{eqn:PoiLstat} in the time parameter $n$. To see the stationarity in space (shifts by $x$), note that the $\nu_i$, by their definition in Section \eqref{sec:Poiscoup}, are jointly increment-stationary. The process $U$ is also increment-stationarity under shifts in time, so under the measure $P_m$ for $m < -1$,
\begin{align*}
&\quad \{V_i^{1,m}(x+y) - V_i^{1,m}(x):y \in \R\}_{1 \le i \le k} \deq \{h_U(-1,x+y; m,\nu_i) - h_U(-1,x;m,\nu_i)\}_{1 \le i \le k} \\
&=
\Bigl\{\sup_{-\infty < z \le x+y}[\nu_i(z) + U(z,m+1;y + x,-1)] - \sup_{-\infty < z \le x}[\nu_i(z) + U(z,m+1;x,-1)]\Bigr\}_{1 \le i \le k} \\
&\deq \Bigl\{\sup_{-\infty < z \le x+y}[\nu_i(z -x)  + U(z-x,m+1;y,-1)] - \sup_{-\infty < z \le x}[\nu_i(z-x) + U(z-x,m+1;0,-1)]\Bigr\}_{1 \le i \le k} \\
&= \{h_U(-1,y; m,\nu_i) - h_U(-1,0;m,\nu_i)\}_{1 \le i \le k} = \{V_i^{1,m}(y):y \in \R\}.
\end{align*}
This increment-stationarity passes to the process $\mbf V^{1,M}$ under the averaged measure, then to the limiting process $\mbf V^1$. 
\end{proof}

\bibliographystyle{alpha}
\bibliography{references_file}

\newcommand{\etalchar}[1]{$^{#1}$}
\begin{thebibliography}{GRASY15}

\bibitem[AAV11]{Amir_Angel_Valko11}
Gideon Amir, Omer Angel, and Benedek Valk\'{o}.
\newblock The {TASEP} speed process.
\newblock {\em Ann. Probab.}, 39(4):1205--1242, 2011.

\bibitem[ABC12]{Auffinger-Baik-Corwin-2012}
Antonio {Auffinger}, Jinho {Baik}, and Ivan {Corwin}.
\newblock {Universality for directed polymers in thin rectangles}.
\newblock {\em Preprint: arXiv:1204.4445}, 2012.

\bibitem[ACG23]{Aggarwal-Corwin-Ghosal-2023}
Amol Aggarwal, Ivan Corwin, and Promit Ghosal.
\newblock The {ASEP} speed process.
\newblock {\em Adv. Math.}, 422:Paper No. 109004, 57, 2023.

\bibitem[ACH24]{Aggarwal-Corwin-Hegde-2024}
Amol {Aggarwal}, Ivan {Corwin}, and Milind {Hegde}.
\newblock {Scaling limit of the colored ASEP and stochastic six-vertex models}.
\newblock {\em Preprint:arXiv:2403.01341}, 2024.

\bibitem[AD95]{Aldous-Diaconis-1995}
D.~Aldous and P.~Diaconis.
\newblock Hammersley's interacting particle process and longest increasing
  subsequences.
\newblock {\em Probab. Theory Related Fields}, 103(2):199--213, 1995.

\bibitem[AH16]{Ahlberg_Hoffman}
Daniel {Ahlberg} and Christopher {Hoffman}.
\newblock {Random coalescing geodesics in first-passage percolation}.
\newblock {\em Preprint: arXiv:1609.02447}, 2016.

\bibitem[AH23]{Aggarwal-Huang-2023}
Amol {Aggarwal} and Jiaoyang {Huang}.
\newblock {Strong Characterization for the Airy Line Ensemble}.
\newblock {\em Preprint:arXiv:2308.11908}, 2023.

\bibitem[AK22]{Alevy-Krishnan}
Ian Alevy and Arjun Krishnan.
\newblock Negative correlation of adjacent {B}usemann increments.
\newblock {\em Ann. Inst. Henri Poincar\'{e} Probab. Stat.}, 58(4):1942--1958,
  2022.

\bibitem[Ang06]{ANGEL-2006}
Omer Angel.
\newblock The stationary measure of a 2-type totally asymmetric exclusion
  process.
\newblock {\em J. Combin. Theory Ser. A}, 113(4):625--635, 2006.

\bibitem[ANP23]{Aggarwal-Nicoletti-Petrov-2023}
Amol {Aggarwal}, Matthew {Nicoletti}, and Leonid {Petrov}.
\newblock {Colored Interacting Particle Systems on the Ring: Stationary
  Measures from Yang-Baxter Equation}.
\newblock {\em Preprint:arXiv:2309.11865}, 2023.

\bibitem[ARAS20]{blpp_utah}
Tom Alberts, Firas Rassoul-Agha, and Mackenzie Simper.
\newblock Busemann functions and semi-infinite {O}'{C}onnell-{Y}or polymers.
\newblock {\em Bernoulli}, 26(3):1927--1955, 2020.

\bibitem[BBAP05]{Baik-BenArous-Peche-2005}
Jinho Baik, G\'{e}rard Ben~Arous, and Sandrine P\'{e}ch\'{e}.
\newblock Phase transition of the largest eigenvalue for nonnull complex sample
  covariance matrices.
\newblock {\em Ann. Probab.}, 33(5):1643--1697, 2005.

\bibitem[BBM00]{Baccelli-2000}
F.~Baccelli, A.~Borovkov, and J.~Mairesse.
\newblock Asymptotic results on infinite tandem queueing networks.
\newblock {\em Probab. Theory Related Fields}, 118(3):365--405, 2000.

\bibitem[BBO05]{Biane-Bougerol-OConnell-2005}
Philippe {Biane}, Philippe {Bougerol}, and Neil {O'Connell}.
\newblock {Littelmann paths and brownian paths}.
\newblock {\em Duke Mathematical Journal}, 130(1):127--167, 2005.

\bibitem[BC23]{Barraquand-Corwin-22}
Guillaume Barraquand and Ivan Corwin.
\newblock Stationary measures for the log-gamma polymer and {KPZ} equation in
  half-space.
\newblock {\em Ann. Probab.}, 51(5):1830--1869, 2023.

\bibitem[BCS06]{Balazs-Cator-Seppalainen-2006}
M.~Bal\'{a}zs, E.~Cator, and T.~Sepp\"{a}l\"{a}inen.
\newblock Cube root fluctuations for the corner growth model associated to the
  exclusion process.
\newblock {\em Electron. J. Probab.}, 11:no. 42, 1094--1132, 2006.

\bibitem[BDJ99]{baik-deif-joha-99}
Jinho Baik, Percy Deift, and Kurt Johansson.
\newblock On the distribution of the length of the longest increasing
  subsequence of random permutations.
\newblock {\em J. Amer. Math. Soc.}, 12(4):1119--1178, 1999.

\bibitem[Bee82]{Beer-82}
Gerald Beer.
\newblock Upper semicontinuous functions and the {S}tone approximation theorem.
\newblock {\em J. Approx. Theory}, 34(no. 1):1--11, 1982.

\bibitem[BF08]{Borodin-Ferrari-2008}
Alexei Borodin and Patrik~L. Ferrari.
\newblock Large time asymptotics of growth models on space-like paths. {I}.
  {P}ush{ASEP}.
\newblock {\em Electron. J. Probab.}, 13:no. 50, 1380--1418, 2008.

\bibitem[BF22]{busa-ferr-20}
Ofer Busani and Patrik~L. Ferrari.
\newblock Universality of the geodesic tree in last passage percolation.
\newblock {\em Ann. Probab.}, 50(1):90--130, 2022.

\bibitem[BFPS07]{Borodin-Ferrari-Prahofer-2007}
Alexei Borodin, Patrik~L. Ferrari, Michael Pr\"{a}hofer, and Tomohiro Sasamoto.
\newblock Fluctuation properties of the {TASEP} with periodic initial
  configuration.
\newblock {\em J. Stat. Phys.}, 129(5-6):1055--1080, 2007.

\bibitem[BFS08]{Borodin-Ferrari-Sasamoto-2008}
Alexei Borodin, Patrik~L. Ferrari, and Tomohiro Sasamoto.
\newblock Large time asymptotics of growth models on space-like paths. {II}.
  {PNG} and parallel {TASEP}.
\newblock {\em Comm. Math. Phys.}, 283(2):417--449, 2008.

\bibitem[BFS23]{Bates-Fan-Seppalainen}
Erik Bates, Wai Tong~(Louis) Fan, and Timo Sepp\"al\"ainen.
\newblock {Intertwining the Busemann process of the directed polymer model}.
\newblock {\em Preprint: arXiv:2307.10531}, 2023.

\bibitem[BG21]{Basu-Ganguly-2021}
Riddhipratim Basu and Shirshendu Ganguly.
\newblock Time correlation exponents in last passage percolation.
\newblock In {\em In and out of equilibrium 3. {C}elebrating {V}ladas
  {S}idoravicius}, volume~77 of {\em Progr. Probab.}, pages 101--123.
  Birkh\"{a}user/Springer, Cham, [2021] \copyright 2021.

\bibitem[BGZ21]{Basu-Ganguly-Zhang-2021}
Riddhipratim Basu, Shirshendu Ganguly, and Lingfu Zhang.
\newblock Temporal correlation in last passage percolation with flat initial
  condition via {B}rownian comparison.
\newblock {\em Comm. Math. Phys.}, 383(3):1805--1888, 2021.

\bibitem[Bha20]{Bhatia-2020}
Manan Bhatia.
\newblock Moderate deviation and exit time estimates for stationary last
  passage percolation.
\newblock {\em J. Stat. Phys.}, 181(4):1410--1432, 2020.

\bibitem[{Bha}23]{Bhatia-23}
Manan {Bhatia}.
\newblock {Duality in the directed landscape and its applications to fractal
  geometry}.
\newblock {\em Preprint: arXiv:2301.07704}, 2023.
\newblock To appear in IMRN.

\bibitem[BL21]{Baik-Liu-2021}
Jinho Baik and Zhipeng Liu.
\newblock Periodic {TASEP} with general initial conditions.
\newblock {\em Probab. Theory Related Fields}, 179(3-4):1047--1144, 2021.

\bibitem[BSS14]{Basu-Sidoravicius-Sly-2014}
Riddhipratim {Basu}, Vladas {Sidoravicius}, and Allan {Sly}.
\newblock {Last Passage Percolation with a Defect Line and the Solution of the
  Slow Bond Problem}.
\newblock {\em Preprint:arXiv:1408.3464}, 2014.

\bibitem[BSS19]{BasuSarkarSly_Coalescence}
Riddhipratim Basu, Sourav Sarkar, and Allan Sly.
\newblock Coalescence of geodesics in exactly solvable models of last passage
  percolation.
\newblock {\em J. Math. Phys.}, 60(9):093301, 22, 2019.

\bibitem[BSS22]{Busa-Sepp-Sore-22b}
Ofer {Busani}, Timo {Sepp{\"a}l{\"a}inen}, and Evan {Sorensen}.
\newblock {Scaling limit of the TASEP speed process}.
\newblock {\em Preprint:arXiv:2211.04651}, 2022.

\bibitem[BSS24]{Busa-Sepp-Sore-22a}
Ofer Busani, Timo Sepp\"{a}l\"{a}inen, and Evan Sorensen.
\newblock The stationary horizon and semi-infinite geodesics in the directed
  landscape.
\newblock {\em Ann. Probab.}, 52(1):1--66, 2024.

\bibitem[Bur56]{Burke1956}
Paul~J. Burke.
\newblock The output of a queuing system.
\newblock {\em Operations Res.}, 4:699--704 (1957), 1956.

\bibitem[{Bus}21]{Busani-2021}
Ofer {Busani}.
\newblock {Diffusive scaling limit of the Busemann process in Last Passage
  Percolation}.
\newblock {\em Preprint:arXiv:2110.03808}, 2021.
\newblock To appear in Ann. Probab.

\bibitem[CG05]{Cator-Groeneboom-2005}
Eric Cator and Piet Groeneboom.
\newblock Hammersley's process with sources and sinks.
\newblock {\em Ann. Probab.}, 33(3):879--903, 2005.

\bibitem[CG06]{Cator-Groeneboom-06}
Eric Cator and Piet Groeneboom.
\newblock Second class particles and cube root asymptotics for {H}ammersley's
  process.
\newblock {\em Ann. Probab.}, 34(4):1273--1295, 2006.

\bibitem[CH14]{CorwinHammond}
Ivan Corwin and Alan Hammond.
\newblock Brownian {G}ibbs property for {A}iry line ensembles.
\newblock {\em Invent. Math.}, 195(2):441--508, 2014.

\bibitem[CP12]{Cator-Pimentel-2012}
Eric Cator and Leandro P.~R. Pimentel.
\newblock Busemann functions and equilibrium measures in last passage
  percolation models.
\newblock {\em Probab. Theory Related Fields}, 154(1-2):89--125, 2012.

\bibitem[DH14]{Damron_Hanson2012}
Michael Damron and Jack Hanson.
\newblock Busemann functions and infinite geodesics in two-dimensional
  first-passage percolation.
\newblock {\em Comm. Math. Phys.}, 325(3):917--963, 2014.

\bibitem[DJLS93]{Derrida-Janowsky-Lebowitz-Speer-1993}
B.~Derrida, S.~A. Janowsky, J.~L. Lebowitz, and E.~R. Speer.
\newblock Exact solution of the totally asymmetric simple exclusion process:
  shock profiles.
\newblock {\em J. Statist. Phys.}, 73(5-6):813--842, 1993.

\bibitem[DMO05]{Draief-2005}
Moez Draief, Jean Mairesse, and Neil O'Connell.
\newblock Queues, stores, and tableaux.
\newblock {\em J. Appl. Probab.}, 42(4):1145--1167, 2005.

\bibitem[DNV22]{Dauvegne-Nica-Virag-2021}
Duncan Dauvergne, Mihai Nica, and B\'{a}lint Vir\'{a}g.
\newblock R{SK} in last passage percolation: a unified approach.
\newblock {\em Probab. Surv.}, 19:65--112, 2022.

\bibitem[DNV23]{Dauvergne2019UniformCT}
Duncan Dauvergne, Mihai Nica, and B\'{a}lint Vir\'{a}g.
\newblock Uniform convergence to the {A}iry line ensemble.
\newblock {\em Ann. Inst. Henri Poincar\'{e} Probab. Stat.}, 59(4):2220--2256,
  2023.

\bibitem[DOV22]{Directed_Landscape}
Duncan Dauvergne, Janosch Ortmann, and B\'{a}lint Vir\'{a}g.
\newblock The directed landscape.
\newblock {\em Acta Math.}, 229(2):201--285, 2022.

\bibitem[DSV22]{Dauvergne-Sarkar-Virag-2020}
Duncan Dauvergne, Sourav Sarkar, and B\'{a}lint Vir\'{a}g.
\newblock Three-halves variation of geodesics in the directed landscape.
\newblock {\em Ann. Probab.}, 50(5):1947--1985, 2022.

\bibitem[Dud89]{dudl}
Richard~M. Dudley.
\newblock {\em Real analysis and probability}.
\newblock The Wadsworth \& Brooks/Cole Mathematics Series. Wadsworth \&
  Brooks/Cole Advanced Books \& Software, Pacific Grove, CA, 1989.

\bibitem[DV21a]{Dauvergne-Virag-18}
Duncan Dauvergne and B\'{a}lint Vir\'{a}g.
\newblock Bulk properties of the {A}iry line ensemble.
\newblock {\em Ann. Probab.}, 49(4):1738--1777, 2021.

\bibitem[DV21b]{Dauvergne-Virag-21}
Duncan {Dauvergne} and B{\'a}lint {Vir{\'a}g}.
\newblock {The scaling limit of the longest increasing subsequence}.
\newblock {\em Preprint:arXiv:2104.08210}, 2021.

\bibitem[DV24]{Dauvergne-Virag-2024}
Duncan {Dauvergne} and B{\'a}lint {Vir{\'a}g}.
\newblock {The directed landscape from Brownian motion}.
\newblock {\em Preprint:arXiv:2405.00194}, 2024.

\bibitem[DZ21]{Dauvergne-Zhang-2021}
Duncan {Dauvergne} and Lingfu {Zhang}.
\newblock {Disjoint optimizers and the directed landscape}.
\newblock {\em Preprint:arXiv:2102.00954}, 2021.
\newblock To appear in Mem. Amer. Math. Soc.

\bibitem[Ede61]{Ede-61}
Murray Eden.
\newblock A two-dimensional growth process.
\newblock In {\em Proc. 4th {B}erkeley {S}ympos. {M}ath. {S}tatist. and
  {P}rob., {V}ol. {IV}}, pages 223--239. Univ. California Press, Berkeley,
  Calif., 1961.

\bibitem[EJS20]{Emrah-Janjigian-Seppalainen-20}
Elnur {Emrah}, Chris {Janjigian}, and Timo {Sepp{\"a}l{\"a}inen}.
\newblock {Right-tail moderate deviations in the exponential last-passage
  percolation}.
\newblock {\em Preprint:arXiv:2004.04285}, 2020.

\bibitem[EJS23]{Emrah-Janjigian-Seppalainen-21}
Elnur {Emrah}, Christopher {Janjigian}, and Timo {Sepp{\"a}l{\"a}inen}.
\newblock {Optimal-order exit point bounds in exponential last-passage
  percolation via the coupling technique}.
\newblock {\em Probab. Math. Phys.}, 4(3):609--666, 2023.

\bibitem[EK86]{ethi-kurt}
Stewart~N. Ethier and Thomas~G. Kurtz.
\newblock {\em Markov processes: Characterization and convergence}.
\newblock Wiley Series in Probability and Mathematical Statistics. John Wiley
  \& Sons Inc., New York, 1986.

\bibitem[FFK94]{Ferrari-Fontes-Kohayakawa-1994}
P.~A. Ferrari, L.~R.~G. Fontes, and Y.~Kohayakawa.
\newblock Invariant measures for a two-species asymmetric process.
\newblock {\em J. Statist. Phys.}, 76(5-6):1153--1177, 1994.

\bibitem[FGN19]{Ferrari-Ghosal-Nejjar-19}
P.~L. Ferrari, P.~Ghosal, and P.~Nejjar.
\newblock Limit law of a second class particle in {TASEP} with non-random
  initial condition.
\newblock {\em Ann. Inst. Henri Poincar\'{e} Probab. Stat.}, 55(3):1203--1225,
  2019.

\bibitem[FKS91]{Ferrari-Kipnis-Saada-1991}
P.~A. Ferrari, C.~Kipnis, and E.~Saada.
\newblock Microscopic structure of travelling waves in the asymmetric simple
  exclusion process.
\newblock {\em Ann. Probab.}, 19(1):226--244, 1991.

\bibitem[FM06]{Ferrari-Martin-2005}
Pablo~A. Ferrari and J.~B. Martin.
\newblock Multi-class processes, dual points and {$M/M/1$} queues.
\newblock {\em Markov Process. Related Fields}, 12(2):175--201, 2006.

\bibitem[FM07]{Ferrari-Martin-2007}
Pablo~A. Ferrari and James~B. Martin.
\newblock Stationary distributions of multi-type totally asymmetric exclusion
  processes.
\newblock {\em Ann. Probab.}, 35(3):807--832, 2007.

\bibitem[FM09]{Ferrari-Martin-2009}
Pablo~A. Ferrari and James~B. Martin.
\newblock Multiclass {H}ammersley-{A}ldous-{D}iaconis process and
  multiclass-customer queues.
\newblock {\em Ann. Inst. Henri Poincar\'{e} Probab. Stat.}, 45(1):250--265,
  2009.

\bibitem[FO18]{Ferrari-Occelli-18}
Patrik~L. Ferrari and Alessandra Occelli.
\newblock Universality of the {GOE} {T}racy-{W}idom distribution for {TASEP}
  with arbitrary particle density.
\newblock {\em Electron. J. Probab.}, 23:Paper No. 51, 24, 2018.

\bibitem[FS20]{Fan-Seppalainen-20}
Wai-Tong~Louis Fan and Timo Sepp\"{a}l\"{a}inen.
\newblock Joint distribution of {B}usemann functions in the exactly solvable
  corner growth model.
\newblock {\em Probab. Math. Phys.}, 1(1):55--100, 2020.

\bibitem[GJR21]{Groathouse-Janjigian-Rassoul-21}
Sean {Groathouse}, Christopher {Janjigian}, and Firas {Rassoul-Agha}.
\newblock {Non-existence of non-trivial bi-infinite geodesics in Geometric Last
  Passage Percolation}.
\newblock {\em Preprint:arXiv:2112.00161}, 2021.

\bibitem[GJR23]{Groathouse-Janjigian-Rassoul-2023}
Sean {Groathouse}, Christopher {Janjigian}, and Firas {Rassoul-Agha}.
\newblock {Existence of generalized Busemann functions and Gibbs measures for
  random walks in random potentials}.
\newblock {\em Preprint:arXiv:2306.17714}, 2023.

\bibitem[GRAS17]{Georgiou-Rassoul-Seppalainen-17b}
Nicos Georgiou, Firas Rassoul-Agha, and Timo Sepp\"{a}l\"{a}inen.
\newblock Stationary cocycles and {B}usemann functions for the corner growth
  model.
\newblock {\em Probab. Theory Related Fields}, 169(1-2):177--222, 2017.

\bibitem[GRASS23]{GRASS-23}
Sean Groathouse, Firas Rassoul-Agha, Timo {Sepp{\"a}l{\"a}inen}, and Evan
  {Sorensen}.
\newblock {Jointly invariant measures for the Kardar-Parisi-Zhang equation}.
\newblock {\em Preprint:arXiv:2309.17276}, 2023.

\bibitem[GRASY15]{geor-rass-sepp-yilm-15}
Nicos Georgiou, Firas Rassoul-Agha, Timo Sepp{\"a}l{\"a}inen, and Atilla
  Yilmaz.
\newblock Ratios of partition functions for the log-gamma polymer.
\newblock {\em Ann. Probab.}, 43(5):2282--2331, 2015.

\bibitem[GW91]{glynn1991}
Peter~W. Glynn and Ward Whitt.
\newblock Departures from many queues in series.
\newblock {\em Ann. Appl. Probab.}, 1(4):546--572, 1991.

\bibitem[Ham72]{Hammersley-1972}
J.~M. Hammersley.
\newblock A few seedlings of research.
\newblock In {\em Proceedings of the {S}ixth {B}erkeley {S}ymposium on
  {M}athematical {S}tatistics and {P}robability ({U}niv. {C}alifornia,
  {B}erkeley, {C}alif., 1970/1971), {V}ol. {I}: {T}heory of statistics}, pages
  345--394. Univ. California Press, Berkeley, CA, 1972.

\bibitem[HB76]{Hsu-Burke-1976}
J.~Hsu and P.~Burke.
\newblock Behavior of tandem buffers with geometric input and markovian output.
\newblock {\em IEEE Transactions on Communications}, 24(3):358--361, 1976.

\bibitem[Hof08]{Hoffman2008}
Christopher Hoffman.
\newblock Geodesics in first passage percolation.
\newblock {\em Ann. Appl. Probab.}, 18(5):1944--1969, 2008.

\bibitem[HS20]{Hammond-Sarkar-2020}
Alan Hammond and Sourav Sarkar.
\newblock Modulus of continuity for polymer fluctuations and weight profiles in
  {P}oissonian last passage percolation.
\newblock {\em Electron. J. Probab.}, 25:Paper No. 29, 38, 2020.

\bibitem[HW65]{Ham-Wel-65}
John~M. Hammersley and J.~A.~Dominic Welsh.
\newblock First-passage percolation, subadditive processes, stochastic
  networks, and generalized renewal theory.
\newblock In {\em Proc. {I}nternat. {R}es. {S}emin., {S}tatist. {L}ab., {U}niv.
  {C}alifornia, {B}erkeley, {C}alif}, pages 61--110. Springer-Verlag, New York,
  1965.

\bibitem[Joh00a]{joha}
Kurt Johansson.
\newblock Shape fluctuations and random matrices.
\newblock {\em Comm. Math. Phys.}, 209(2):437--476, 2000.

\bibitem[Joh00b]{Johansson-2000}
Kurt Johansson.
\newblock Transversal fluctuations for increasing subsequences on the plane.
\newblock {\em Probab. Theory Related Fields}, 116(4):445--456, 2000.

\bibitem[Joh01]{Johansson-2001}
Kurt Johansson.
\newblock Discrete orthogonal polynomial ensembles and the {P}lancherel
  measure.
\newblock {\em Ann. of Math. (2)}, 153(1):259--296, 2001.

\bibitem[JRA20]{Janjigian-Rassoul-2020b}
Christopher Janjigian and Firas Rassoul-Agha.
\newblock Busemann functions and {G}ibbs measures in directed polymer models on
  {$\mathbb Z^2$}.
\newblock {\em Ann. Probab.}, 48(2):778--816, 2020.

\bibitem[JRS22]{Janj-Rass-Sepp-22}
Christopher {Janjigian}, Firas {Rassoul-Agha}, and Timo {Sepp{\"a}l{\"a}inen}.
\newblock {Ergodicity and synchronization of the Kardar-Parisi-Zhang equation}.
\newblock {\em Preprint:arXiv:2211.06779}, 2022.

\bibitem[Kel11]{Kelly-2011}
F.~P. Kelly.
\newblock {\em Reversibility and stochastic networks}.
\newblock Cambridge Mathematical Library. Cambridge University Press,
  Cambridge, revised edition, 2011.

\bibitem[KOR02]{koni-ocon-roch-02}
Wolfgang K\"{o}nig, Neil O'Connell, and S\'{e}bastien Roch.
\newblock Non-colliding random walks, tandem queues, and discrete orthogonal
  polynomial ensembles.
\newblock {\em Electron. J. Probab.}, 7:no. 5, 24, 2002.

\bibitem[KPZ86]{Kardar-Parisi-Zhang-86}
Mehran Kardar, Giorgio Parisi, and Yi-Cheng Zhang.
\newblock Dynamic scaling of growing interfaces.
\newblock {\em Phys. Rev. Lett.}, 56:889--892, 1986.

\bibitem[Liu22]{Liu-2022}
Zhipeng Liu.
\newblock Multipoint distribution of {TASEP}.
\newblock {\em Ann. Probab.}, 50(4):1255--1321, 2022.

\bibitem[LM01]{Lowe-Merkl-2001}
Matthias L\"{o}we and Franz Merkl.
\newblock Moderate deviations for longest increasing subsequences: the upper
  tail.
\newblock {\em Comm. Pure Appl. Math.}, 54(12):1488--1520, 2001.

\bibitem[LMR02]{Lowe-Merkl-Rolles-2002}
Matthias L\"{o}we, Franz Merkl, and Silke Rolles.
\newblock Moderate deviations for longest increasing subsequences: the lower
  tail.
\newblock {\em J. Theoret. Probab.}, 15(4):1031--1047, 2002.

\bibitem[LNS23]{Landon-Sosoe-Noack-2020}
Benjamin Landon, Christian Noack, and Philippe Sosoe.
\newblock K{PZ}-type fluctuation exponents for interacting diffusions in
  equilibrium.
\newblock {\em Ann. Probab.}, 51(3):1139--1191, 2023.

\bibitem[LR10]{Ledoux-Rider-2010}
Michel Ledoux and Brian Rider.
\newblock Small deviations for beta ensembles.
\newblock {\em Electron. J. Probab.}, 15:no. 41, 1319--1343, 2010.

\bibitem[LS22]{Landon-Sosoe-22b}
Benjamin {Landon} and Philippe {Sosoe}.
\newblock {Tail bounds for the O'Connell-Yor polymer}.
\newblock {\em Preprint:arXiv:2209.12704}, 2022.

\bibitem[LS23a]{landon2023tail}
Benjamin Landon and Philippe Sosoe.
\newblock Tail estimates for the stationary stochastic six vertex model and
  {ASEP}.
\newblock {\em Preprint:arXiv:2308.16812}, 2023.

\bibitem[LS23b]{Landon-Sosoe-22a}
Benjamin Landon and Philippe Sosoe.
\newblock Upper tail bounds for stationary {KPZ} models.
\newblock {\em Comm. Math. Phys.}, 401(2):1311--1335, 2023.

\bibitem[Mar20]{Martin-2020}
James~B. Martin.
\newblock Stationary distributions of the multi-type {ASEP}.
\newblock {\em Electron. J. Probab.}, 25:Paper No. 43, 41, 2020.

\bibitem[MMK{\etalchar{+}}97]{maun-etal-97}
J.~Maunuksela, M.~Myllys, O.-P. K{\"a}hk{\"o}nen, J.~Timonen, N.~Provatas,
  M.~J. Alava, and T.~Ala-Nissila.
\newblock Kinetic roughening in slow combustion of paper.
\newblock {\em Phys. Rev. Lett.}, 79:1515Ð1518, 1997.

\bibitem[MP03]{Mairesse-Prabhakar-2003}
Jean Mairesse and Balaji Prabhakar.
\newblock The existence of fixed points for the {$\cdot/GI/1$} queue.
\newblock {\em Ann. Probab.}, 31(4):2216--2236, 2003.

\bibitem[MQR21]{KPZfixed}
Konstantin Matetski, Jeremy Quastel, and Daniel Remenik.
\newblock The {KPZ} fixed point.
\newblock {\em Acta Math.}, 227(1):115--203, 2021.

\bibitem[MSZ21]{Martin-Sly-Zhang-21}
James~B. {Martin}, Allan {Sly}, and Lingfu {Zhang}.
\newblock {Convergence of the Environment Seen from Geodesics in Exponential
  Last-Passage Percolation}.
\newblock {\em Preprint: arXiv:2106.05242}, 2021.

\bibitem[NQR20]{reflected_KPZfixed}
Mihai Nica, Jeremy Quastel, and Daniel Remenik.
\newblock One-sided reflected {B}rownian motions and the {KPZ} fixed point.
\newblock {\em Forum Math. Sigma}, 8:Paper No. e63, 16, 2020.

\bibitem[NY04]{Noumi-Yamada-2004}
Masatoshi Noumi and Yasuhiko Yamada.
\newblock Tropical robinson-schensted-knuth correspondence and birational weyl
  group actions.
\newblock {\em Adv. Stud. Pure Math.}, pages 371--442, 2004.

\bibitem[OY01]{brownian_queues}
Neil O'Connell and Marc Yor.
\newblock Brownian analogues of {B}urke's theorem.
\newblock {\em Stochastic Process. Appl.}, 96(2):285--304, 2001.

\bibitem[OY02]{rep_non_colliding}
Neil O'Connell and Marc Yor.
\newblock A representation for non-colliding random walks.
\newblock {\em Electron. Comm. Probab.}, 7:1--12, 2002.

\bibitem[PEM09]{Prolhac-Evans-Mallick-2009}
S.~Prolhac, M.~R. Evans, and K.~Mallick.
\newblock The matrix product solution of the multispecies partially asymmetric
  exclusion process.
\newblock {\em J. Phys. A}, 42(16):165004, 25, 2009.

\bibitem[Pim16]{pimentel2016}
Leandro P.~R. Pimentel.
\newblock Duality between coalescence times and exit points in last-passage
  percolation models.
\newblock {\em Ann. Probab.}, 44(5):3187--3206, 2016.

\bibitem[Pim18]{Pimentel-18}
Leandro P.~R. Pimentel.
\newblock Local behaviour of airy processes.
\newblock {\em J. Stat. Phys.}, 173(6):1614--1638, 2018.

\bibitem[Pim21]{Pimentel-21a}
Leandro P.~R. Pimentel.
\newblock Ergodicity of the {KPZ} fixed point.
\newblock {\em ALEA Lat. Am. J. Probab. Math. Stat.}, 18(1):963--983, 2021.

\bibitem[Pit75]{Pitman1975}
J.~W. Pitman.
\newblock One-dimensional {B}rownian motion and the three-dimensional {B}essel
  process.
\newblock {\em Advances in Appl. Probability}, 7(3):511--526, 1975.

\bibitem[PS02]{Prahofer-Spohn-02}
Michael Pr\"{a}hofer and Herbert Spohn.
\newblock Scale invariance of the {PNG} droplet and the {A}iry process.
\newblock {\em J. Statist. Phys.}, 108(5-6):1071--1106, 2002.
\newblock Dedicated to David Ruelle and Yasha Sinai on the occasion of their
  65th birthdays.

\bibitem[QS23]{KPZ_equation_convergence}
Jeremy Quastel and Sourav Sarkar.
\newblock Convergence of exclusion processes and the {KPZ} equation to the
  {KPZ} fixed point.
\newblock {\em J. Amer. Math. Soc.}, 36(1):251--289, 2023.

\bibitem[{Rai}00]{Rains-2000}
Eric~M. {Rains}.
\newblock {A mean identity for longest increasing subsequence problems}.
\newblock {\em Preprint: arXiv:0004082}, 2000.

\bibitem[RV21]{Rahman-Virag-21}
Mustazee {Rahman} and B\'alint {Vir\'ag}.
\newblock {Infinite geodesics, competition interfaces and the second class
  particle in the scaling limit}.
\newblock {\em Preprint:arXiv:2112.06849}, 2021.
\newblock To appear in Annales de l'Institut Henri Poincar\'{e}
  Probabilit\'{e}s et Statistiques.

\bibitem[Sep98a]{Seppalainen-1998b}
T.~Sepp\"{a}l\"{a}inen.
\newblock Hydrodynamic scaling, convex duality and asymptotic shapes of growth
  models.
\newblock {\em Markov Process. Related Fields}, 4(1):1--26, 1998.

\bibitem[Sep98b]{Seppalainen-1998}
Timo Sepp\"{a}l\"{a}inen.
\newblock Exact limiting shape for a simplified model of first-passage
  percolation on the plane.
\newblock {\em Ann. Probab.}, 26(3):1232--1250, 1998.

\bibitem[Sep12]{Seppalainen-2012}
Timo Sepp\"{a}l\"{a}inen.
\newblock Scaling for a one-dimensional directed polymer with boundary
  conditions.
\newblock {\em Ann. Probab.}, 40(1):19--73, 2012.

\bibitem[Sep18]{Sepp_lecture_notes}
Timo Sepp\"{a}l\"{a}inen.
\newblock The corner growth model with exponential weights.
\newblock In {\em Random growth models}, volume~75 of {\em Proc. Sympos. Appl.
  Math.}, pages 133--201. Amer. Math. Soc., Providence, RI, 2018.

\bibitem[Sep20]{Timo_Coalescence}
Timo Sepp\"{a}l\"{a}inen.
\newblock Existence, uniqueness and coalescence of directed planar geodesics:
  proof via the increment-stationary growth process.
\newblock {\em Ann. Inst. Henri Poincar\'{e} Probab. Stat.}, 56(3):1775--1791,
  2020.

\bibitem[Sor23]{Sorensen-thesis}
Evan Sorensen.
\newblock The stationary horizon as the central multi-type invariant measure in
  the {KPZ} universality class.
\newblock {\em PhD thesis, University of Wisconsin--Madison}, 2023.
\newblock Available at \url{https://arxiv.org/abs/2306.09584}.

\bibitem[SS20]{Seppalainen-Shen-2020}
Timo Sepp\"{a}l\"{a}inen and Xiao Shen.
\newblock Coalescence estimates for the corner growth model with exponential
  weights.
\newblock {\em Electron. J. Probab.}, 25:Paper No. 85, 31, 2020.
\newblock Corrected version \url{https://arxiv.org/abs/1911.03792}.

\bibitem[SS23a]{Seppalainen-Sorensen-21a}
Timo Sepp\"{a}l\"{a}inen and Evan Sorensen.
\newblock Busemann process and semi-infinite geodesics in {B}rownian
  last-passage percolation.
\newblock {\em Ann. Inst. Henri Poincar\'{e} Probab. Stat.}, 59(1):117--165,
  2023.

\bibitem[SS23b]{Seppalainen-Sorensen-21b}
Timo Sepp\"{a}l\"{a}inen and Evan Sorensen.
\newblock Global structure of semi-infinite geodesics and competition
  interfaces in {B}rownian last-passage percolation.
\newblock {\em Probab. Math. Phys.}, 4(3):667--760, 2023.

\bibitem[SV21]{Sarkar-Virag-21}
Sourav Sarkar and B\'{a}lint Vir\'{a}g.
\newblock Brownian absolute continuity of the {KPZ} fixed point with arbitrary
  initial condition.
\newblock {\em Ann. Probab.}, 49(4):1718--1737, 2021.

\bibitem[TS10]{take-sano-10}
Kazumasa~A. Takeuchi and Masaki Sano.
\newblock Universal fluctuations of growing interfaces: evidence in turbulent
  liquid crystals.
\newblock {\em Phys. Rev. Lett.}, 104:230601, 2010.

\bibitem[TW94]{Tra-Wid-94}
Craig~A. Tracy and Harold Widom.
\newblock Level-spacing distributions and the {A}iry kernel.
\newblock {\em Comm. Math. Phys.}, 159(1):151--174, 1994.

\bibitem[V{i}r20]{heat_and_landscape}
B{\'a}lint V{i}r\'ag.
\newblock {The heat and the landscape I}.
\newblock {\em Preprint:arXiv:2008.07241}, 2020.

\bibitem[{Wu}23]{Wu-23}
Xuan {Wu}.
\newblock {The KPZ equation and the directed landscape}.
\newblock {\em Preprint:arXiv:2301.00547}, 2023.

\bibitem[Xie22]{Xie-22}
Yongjia Xie.
\newblock Limiting distributions and deviation estimates of random walks in
  dynamic random environments.
\newblock {\em PhD Thesis, Purdue University}, page 110, 2022.

\end{thebibliography}
\end{document}